\documentclass[10pt,reqno]{article}
\RequirePackage[OT1]{fontenc}
\RequirePackage{amsthm,amsmath}
\RequirePackage[numbers]{natbib}
\RequirePackage[colorlinks,linkcolor=blue,citecolor=blue,urlcolor=blue]{hyperref}
\usepackage{amssymb}
\usepackage{anysize}
\usepackage{extarrows}
\usepackage{multirow}
\usepackage{natbib}
\usepackage{stmaryrd}
\usepackage{color}
\usepackage{amsmath}
\usepackage{enumitem}
\usepackage{mathtools}
\usepackage{dsfont} 
\usepackage{comment}
\usepackage{soul}

\numberwithin{equation}{section}
\DeclareMathOperator{\tr}{Tr}
\DeclareMathOperator{\ts}{tr}

\DeclareMathOperator{\id}{id}
\DeclareMathOperator{\nb}{Nb}
\let\limsup\relax
\DeclareMathOperator*{\limsup}{limsup}
\let\liminf\relax
\DeclareMathOperator*{\liminf}{liminf}

\newcommand{\norm}[1]{\left\Vert #1\right\Vert}

\newtheorem{theorem}{Theorem} [section]
\newtheorem{prop}[theorem]{Proposition} 
\newtheorem{defi}[theorem]{Definition} 
 
\newtheorem{lemma}[theorem]{Lemma} 
\newtheorem{rem}[theorem]{Remark} 
\newtheorem{cor}[theorem]{Corollary} 

\newtheorem{assum}[theorem]{Assumption}
\renewcommand\P{\mathbb{P}}
\newcommand\E{\mathbb{E}}
\newcommand\N{\mathbb{N}}
\newcommand\1{\mathbf{1}}
\newcommand\C{\mathbb{C}}
\newcommand\CC{\mathcal{C}}
\newcommand\M{\mathbb{M}}
\newcommand\R{\mathbb{R}}

\newcommand\A{\mathcal{A}}

\newcommand\D{\mathcal{D}}

\newcommand\MM{\mathcal{M}}
\renewcommand\i{\mathbf{i}}

\newcommand{\deq}{\mathrel{\mathop:}=}

\allowdisplaybreaks
  
\begin{document}

 \begin{minipage}{0.85\textwidth}
 	\vspace{2.5cm}
 \end{minipage}
 \begin{center}
 	\large\bf The free energy of matrix models
 	
 \end{center}

 \renewcommand{\thefootnote}{\fnsymbol{footnote}}	
 \vspace{0.8cm}
 
 \begin{center}
 	\begin{minipage}{1.5\textwidth}
 		
 		\begin{minipage}{0.3\textwidth}
 			\begin{center}
 				F\'elix Parraud\\
 				\footnotesize 
 				{KTH Royal Institute of Technology,}\\
 				{Stockholm, Sweden.} \\
 				{\it parraud@kth.se}
 			\end{center}
 		\end{minipage}
 		\begin{minipage}{0.3\textwidth}
 			\begin{center}
 				Kevin Schnelli\\
 				\footnotesize 
 				{KTH Royal Institute of Technology,}\\
  				{Stockholm, Sweden.} \\
 				{\it schnelli@kth.se}
 			\end{center}
 		\end{minipage}
 	\end{minipage}
 \end{center}
 
 \bigskip
\bigskip


%
%
%
%

E-mail of the corresponding author: felix.parraud@gmail.com

The authors have no relevant financial or non-financial interests to disclose.


\begin{abstract}

	In this paper we study multi-matrix models whose potentials are perturbations of the quadratic potential associated with independent GUE random matrices. More precisely, we compute the free energy and the expectation of the trace of polynomials evaluated in those matrices. We prove an asymptotic expansion in the inverse of the matrix dimension to any order. Out of this result we deduce new formulas for map enumerations and the microstates free entropy. Our approach is based on the interpolation method between random matrices and free operators developed in \cite{un,trois}. \\
	
	
\end{abstract}

\section{Introduction}

The interest in multi-matrix models has a long history, starting with a result of Harer and Zagier \cite{harerzag} where they used the large dimension expansion of the moments of Gaussian random matrices to compute the Euler characteristic of the moduli space of curves. A good introduction to this topic is given in the survey \cite{zvonski} by  Zvonkin. In physics,  the seminal works of t'Hooft \cite{T_ouf} and Br\'ezin, Parisi, Itzykson and Zuber \cite{parisi} related matrix models with the enumeration of maps of any genus, hence providing a purely analytical tool to solve these hard combinatorial problems. The idea is that one can view the free energy of matrix models of dimension $N$ as a formal power series in $N^{-1}$ whose coefficients are generating functions of maps on a surface of a given genus.

During the last two decades, the study of matrix models has been quite active, in \cite{guionnet-segala} and \cite{gui-seg-2} Guionnet and Maurel-Segala studied the first and then the second order of the asymptotic before giving a full expansion in \cite{segala}. More recently in \cite{conftr} they also studied the case of matrix models whose law is far from the quadratic potential. Besides, the unitary equivalents of these matrix models also have a long history starting with the Harish-Chandra-Itzykson-Zuber model, \cite{masd,masd2,masd3,masd4}, which has since then been extended to more general potentials, \cite{segalaU1,novakguio,figalli2,nouv}.

The multi-matrix models were originally introduced as a means to study matrix integrals, i.e. integrals of the following form,
$$ I_N(V) = \int \exp\left(-N\tr_N\left(V(X_1^N,\dots,X_d^N)\right)\right)\ dX_1^N\dots dX^N_d,$$
where the integral is with respect to the Lebesgue measure on the space of Hermitian matrices of size $N$. Those integrals are known to be difficult to estimate and even more so to compute. Indeed the large number of parameters combined with the non-linearity of the exponential makes it impossible to get explicit formulas and one has to find alternative strategies in order to compute approximations of this integral. In order to tackle this problem we focus on studying the case where $V$ is a self-adjoint perturbation of the quadratic potential. Then the strategy consists in introducing the matrix model associated to this potential, the law of this random matrix ensemble is defined similarly to the one of the Gaussian Unitary Ensemble (GUE) but where we replaced the quadratic potential by a more general non-commutative polynomial. More precisely, we will study perturbations of the quadratic potential, that is random matrix ensembles whose law have the following form,
\begin{equation}
	\label{cmksmcdscs}
	d\mu_{\lambda V}^N(X^N) = \frac{1}{Z_{\lambda V}^N} e^{-N\tr_N\left(\lambda V(X^N)+\frac{1}{2}\sum_{i=1}^d(X_i^N)^2\right)} dX_1^N\cdots dX_d^N
\end{equation}
where $\lambda$ is small. Thus one expects the behavior of those to be close to the one of the unperturbed quadratic potential, i.e. the case of a $d$-tuple of independent GUE random matrices. Indeed, we show that one can find operators $\nabla_V$ and $L$ such that for any polynomial $P$,
$$ \mu_{\lambda V}^N\left[ \frac{1}{N}\tr_N\left( \left(\id + \lambda \nabla_V - \frac{L}{N^2} \right)(P)(X^N) \right) \right] = \tau(P(x)), $$
where $x$ is a $d$-tuple of free semicircular variables. Thus heuristically one has that
$$-\frac{1}{N^2}\frac{d}{d\lambda}\log I_N\left(\lambda V+\frac{1}{2}\sum_{i=1}^dX_i^2\right) = \tau\left( \left(\id + \lambda \nabla_V - \frac{L}{N^2} \right)^{-1}(V)(x) \right).$$
In practice, it is unclear whether the operator above is ever invertible when $\lambda\neq 0$. However we show that the formula above still holds for sufficiently small $\lambda\geq 0$ when you replace the inverse of the operator by a Taylor expansion with respect to $N^{-2}$. Thus under some technical assumptions on the potential $V$, one has the following theorem.

\begin{theorem}
	
	\label{maintherorem}
	
	Let the following objects be given,
	\begin{itemize}
			\item $P,V\in\C\langle X_1,\dots,X_d\rangle$ such that $V$ satisfies Assumption \ref{kdmcslcs0} below,
		\item $X^N$ a family of $d$ i.i.d. GUE matrices.
	\end{itemize}
	Then there exists a constant $c_V>0$ depending only on the potential $V$ such that for $\lambda\in [0,c_V]$, for any $k\in\N$, 
	\begin{equation}
		\label{kdjncoskd}
		\frac{\E\left[\frac{1}{N}\tr_N\left(P(X^N)\right) e^{-\lambda N\tr_N\left(V(X^N)\right)}\right]}{\E\left[ e^{-\lambda N\tr_N\left(V(X^N)\right)}\right]} = \sum_{0\leq n \leq k } \frac{\alpha_n^V(\lambda,P)}{N^{2n}} + \mathcal{O}\left(N^{-2(k+1)}\right).
	\end{equation}
	In particular, we have for the free energy that
	\begin{equation}
		\label{kdjncoskd2}
		\frac{1}{N^2}\log\left(\E\left[ e^{-\lambda N\tr_N\left(V(X^N)\right)}\right]\right) = - \sum_{0\leq n \leq k } \frac{1}{N^{2n}}\int_0^{\lambda} \alpha_n^V(\nu,V)\ d\nu + \mathcal{O}\left(N^{-2(k+1)}\right).
	\end{equation}
	The coefficients $\alpha_n^V$ are obtained as follows. With the notations of Theorem \ref{mainlemma}, we have the operators
	$$
	\begin{array}{ccccc}
		\nabla_V & : & \oplus_{H} \A_{d}^H & \to & \oplus_{H} \A_{d}^{\{H,G\}} \\
		& & \oplus_H P_H & \mapsto & \oplus_H \nabla_V^{H,T_H}\left(P_H\right) \\
	\end{array}
	\quad\quad\quad
	\begin{array}{ccccc}
		L & : & \oplus_{H} \A_{d}^H & \to & \oplus_{H} \A_{d}^{\{H,F\}} \\
		& & \oplus_H P_H & \mapsto & \oplus_H L^{H,T_H}\left(P_H\right) \\
	\end{array}.
	$$
	Then
	\begin{align}
		\label{smcskmdcskmcslk}
		&\alpha_n^V(\lambda,P) = \sum_{k_0,\dots,k_n\geq 0} (-\lambda)^{k_0+\dots+k_n} \int\limits_{A_{k_0,\dots,k_n}} \tau\left(\nabla_V^{k_n}\circ L \circ \nabla_V^{k_{n-1}} \dots\circ L\circ \nabla_V^{k_0}(P)\left(x^{T_{2n+k_{n}+\dots+k_0}}\right)\right) \\
		&\quad\quad\quad\quad\quad\quad\quad\quad\quad\quad\quad\quad\quad\quad\quad\quad\quad\quad\quad\quad\quad\quad\quad\quad\quad\quad\quad\quad\quad\quad\quad\quad dt_1\cdots dt_{2n+k_0+\dots+k_n}, \nonumber
	\end{align}
	with the following integration domain:
	\begin{align}
		\label{cmskcokmcvps}
		&A_{k_0,\dots,k_n} = \Bigg\{ t_1,\dots,t_{2n+k_0+\dots+k_n}\in\R_+ \Bigg|\ 
		\forall i\in E_n, t_i\leq t_{i-1}, \\
		&\quad\quad\quad\quad\quad\quad\quad\quad\quad\quad\quad\quad\quad \forall i,j\in [1,2n+k_0+\dots+k_n]\setminus E_n,\text{ if } i< j,\text{ then }  t_i\leq t_j, \Bigg\}, \nonumber
	\end{align}
	where $E_n = \cup_{r=0}^{n-1}\{k_0+\dots+k_r+2r+2\}$.
\end{theorem}

Unlike some past works on multi-matrix models \cite{guionnet-segala,gui-seg-2,segala}, the conditions required for our result (see Assumption \ref{kdmcslcs0}) are weaker and easier to verify than convexity of the potential $V$; in particular our results apply to the trapping potentials of \cite{conftr}. Similarly, if one adds a cut-off, i.e. we bound the norm of our random matrices by a constant $K$, then the free energy of this model has the same expansion as in the previous theorem but without the need for Assumption \ref{kdmcslcs0}.

\begin{theorem}
	
	\label{maintherorem2}
	
	Let the following objects be given,
	\begin{itemize}
		\item $P,V\in\C\langle X_1,\dots,X_d\rangle$ such that $X\in\M_N(\C)^d\mapsto\tr_N(V(X))$ is real-valued for any $N$,
		\item $X^N$ a family of $d$ i.i.d. GUE matrices.
	\end{itemize}
	Then, with $\alpha_n^V$ defined as Theorem \ref{maintherorem}, for $K$ sufficiently large, there exists a constant $c_{V,K}>0$ depending only on the potential $V$ and the cut-off $K$, such that for $\lambda\in [0,c_{V,K}]$, for any $k\in\N$, 
	\begin{equation}
		\label{kdjncoskd4}
		\frac{\E\left[\frac{1}{N}\tr_N\left(P(X^N)\right) e^{-\lambda N\tr_N\left(V(X^N)\right)}  \1_{\forall i, ||X_i^N||\leq K}\right]}{\E\left[ e^{-\lambda N\tr_N\left(V(X^N)\right)}  \1_{\forall i, ||X_i^N||\leq K}\right]} = \sum_{0\leq n \leq k } \frac{\alpha_n^V(\lambda,P)}{N^{2n}} + \mathcal{O}\left(N^{-2(k+1)}\right).
	\end{equation}
	In particular, we have for the free energy that
	\begin{equation}
		\label{kdjncoskd3}
		\frac{1}{N^2}\log\left(\E\left[ e^{-\lambda N\tr_N\left(V(X^N)\right)}  \1_{\forall i, ||X_i^N||\leq K}\right]\right) = - \sum_{0\leq n \leq k } \frac{1}{N^{2n}}\int_0^{\lambda} \alpha_n^V(\nu,V)\ d\nu + \mathcal{O}\left(N^{-2(k+1)}\right).
	\end{equation}

\end{theorem}

The approach taken to prove these theorems is similar to the one used in \cite{un,deux,trois,sept}, yet it is the first time that it has been used to study matrix models. Notably, those papers have already established that GUE and Haar random matrices can be handled similarly by interpolating random matrices with free operators. Thus we expect that the strategy used to prove Theorems \ref{maintherorem} and \ref{maintherorem2} can also be used to study unitary matrix models. This method has the advantage that the constant $c_V$ does not depend on the order to which we push our Taylor expansion (i.e the number $k$ in Equations \eqref{kdjncoskd} and \eqref{kdjncoskd2}) as it is the case in \cite{segala} and \cite{novakguio}. Besides, the error term is fully explicit in every parameter which allows us to let them vary with $N$, although we do not state it in Theorem \ref{maintherorem} in order to keep it shorter.

There are several directions in which we can refine this theorem. Indeed, although we refrained from doing so in order to keep the paper shorter, it is possible to also include deterministic matrices in the potential $V$ and the polynomial $P$, and since our formulas are fully explicit, we can even consider deterministic matrices of small rank, which means that we can use this strategy to study other types of potentials which are not simply the trace of a polynomial. For example one can use identities such as $ \langle P(X)y| y \rangle = \tr_N(P(X)yy^*)$ to study scalar products. This is a potential that is of interest in the case of spherical integrals, see \cite{spe1} and \cite{spe2}. Note that the scalar product was already investigated in \cite{slefnalks}, albeit not in a potential. Another direction that we could investigate is the case of GOE random matrices, we outline how one can adapt our proofs to this case in the last section.

Theorem \ref{maintherorem} has the following corollaries. To begin with, as previously mentioned the coefficients of the expansion can be related to the generating function of maps on a surface of genus $g$, see Subsection \ref{ksjdcdnv} for definitions. More precisely thanks to Theorem 1.1 of \cite{segala}, one has the following corollary.

\begin{cor}
	\label{sdkjnvkdm}
	We write the potential as $V=\sum_{i=1}^{m}t_i q_i$ where $q_i$ are monomials. Then if $V$ is such that $X\in\M_N(\C)^d\mapsto\tr_N(V(X))$ is real-valued for any $N$, with $q_0$ a monomial, one has for $g\geq 0$ and $\lambda$ sufficiently small,
	$$ \alpha_g^V(\lambda,q_0) = \sum_{\mathbf{k}\in \N^m} \frac{(-\lambda \mathbf{t})^{\mathbf{k}}}{\mathbf{k}!} \mathcal{M}_g^{\mathbf{k}}(q_0),$$
	where $(-\lambda \mathbf{t})^{\mathbf{k}} = \prod_i (-\lambda t_i)^{k_i}$, $\mathbf{k}!= \prod k_i!$, and $\mathcal{M}_g^{\mathbf{k}}(q_0)$ is the number of maps on a surface of genus $g$ with $k_i$ vertices of type $q_i$ and one of type $q_0$, see Definition \ref{socdmsomc} and \ref{socdmsomc2}.
\end{cor}

In particular, the coefficients of those power series are equal. Thus if  $\mathcal{M}_g(Q_1,\dots,Q_n)$ is the number of maps on a surface of genus $g$ with vertices $Q_1,\dots,Q_n$ one can obtain a formula which expresses this quantity with the operators $\nabla_{Q_i}$ and $L$ defined in Theorem \ref{maintherorem}.

Another important application of Theorem \ref{maintherorem} is to study the free entropy of the limit of the non-commutative distribution associated to the matrix model. The microstates free entropy was introduced by Voiculescu in \cite{freeent} as a non-commutative analog to the entropy in classical probability and has since then had numerous applications to von Neumann algebras, see for example \cite{ent1,ent2,ent3,ent4,ent5,ent6,ent7,ent8,ent9,ent10}. In \cite{guionnet-segala} Guionnet and Maurel-Segala computed the free entropy of the functional which appears as the limit of the free energy, i.e. Equation \eqref{kdjncoskd2}. Similarly we compute the microstates free entropy $\chi_g$ of the map $P\mapsto \alpha_0^V(\lambda,P)$. Note that our definition \ref{dovosvmdsdom} of the free entropy is with respect to the Gaussian measure as in \cite{osdnc} instead of the Lebesgue measure.

\begin{cor}
	\label{doifvpelmv}
	Let $V$ be a potential satisfying Assumption \ref{kdmcslcs0}, then for $\lambda$ sufficiently small,
	\begin{align*}
		\chi_g(\alpha_0^V(\lambda,\cdot)) &:= \sup_{R>0} \inf_{n\in\N} \inf_{\varepsilon>0} \limsup_{N\to\infty} \frac{1}{N^2} \log\P\left( X^N\in \Gamma_R(\alpha_0^V(\lambda,\cdot),n,N,\varepsilon) \right) \\
		&= \lambda\ \alpha_0^V(\lambda,V) - \int_0^{\lambda} \alpha_0^V(s,V)\ ds \\
		&= \int_0^{\lambda} s\ \frac{d}{ds}\alpha_0^V(s,V)\ ds.
	\end{align*}
	Note that in the case of $\alpha_0^V(\lambda,\cdot)$, if one replaces $\limsup$ by $\liminf$ in the definition of its microstates free entropy (see Definition \ref{dovosvmdsdom}) the formula above still stands.
\end{cor}

The paper is organized as follows. In the second section we introduce all of the necessary definitions. In the third one we first give the necessary assumptions on the potential $V$ for the matrix model to be well-defined, then we prove what we call the master equation, see Theorem \ref{mainlemma}, in which we build an explicit operator $\Theta$ on the set of polynomials such that for any $P$, the distribution of $\Theta(P)$ evaluated in our random matrix model is the one of $P(x)$ where $x$ is a $d$-tuple of free semicircular variables. In the fourth section we focus on bounding what will essentially be the error term in Theorem \ref{maintherorem} which then allows us to invert the operator $\Theta$ and conclude the argument. Finally in the fifth section we prove Corollaries \ref{sdkjnvkdm} and \ref{doifvpelmv}, and we investigate the case of the GOE in the last section.

\section{General definitions}

\subsection{Basic notions of free probability}
\label{3deffree}

In order to be self-contained, we begin by recalling the following definitions from free probability. For more background on the link between those tools and Random Matrix Theory we refer to \cite{alice, deffreep, nica_speicher_2006}.

\begin{defi}~
	\label{3freeprob}
	\begin{itemize}		
		\item A \textbf{$\CC^*$-probability space} $(\A,*,\tau,\norm{.}) $ is a $\CC^*$-algebra endowed with a bounded linear map $\tau : \A \to \C$ and satisfies $\tau(\id_{\A})=1$, $\tau(a^*a)\geq 0$ and $\tau(ab) = \tau(ba) $ for any $a,b\in\A$. The map $\tau$ is called a \textbf{trace} and an element of $\A$ a \textbf{non-commutative random variable}.
		
		\item Let $\A_1,\dots,\A_n$ be $*$-subalgebras of $\A$, having the same unit as $\A$. They are said to be \textbf{free} if for all $k$, for all $a_i\in\A_{j_i}$ such that $j_1\neq j_2$, $j_2\neq j_3$, \dots , $j_{k-1}\neq j_k$:
		\begin{equation}
			\label{kddkdxkfl}
			\tau\Big( (a_1-\tau(a_1))(a_2-\tau(a_2))\dots (a_k-\tau(a_k)) \Big) = 0.
		\end{equation}
		Families of non-commutative random variables are said to be free if the $*$-subalgebras they generate are free.
		
		\item A family of non-commutative random variables $ x=(x_1,\dots ,x_d)$ is called a \textbf{free semicircular system} if the non-commutative random variables are free, self-adjoint ($x_i=x_i^*$), and for all $k$ in $\N$ and $i$, one has
		\begin{equation*}
			\tau( x_i^k) =  \int_{\R} t^k d\sigma(t),
		\end{equation*}
		with $d\sigma(t) = \frac 1 {2\pi} \sqrt{4-t^2} \ \mathbf 1_{|t|\leq2} \ dt$ the semicircle distribution. Note that thanks to Proposition 7.18 of \cite{nica_speicher_2006}, one can build a free semicircular system for any $d$.
		
	\end{itemize}
	
\end{defi}

Let us also fix a few notations concerning matrices.

\begin{defi}
	\label{3tra} ~
	\begin{itemize}
		\item $\M_N(\C)$ is the space of complex square matrices of size $N$.
		\item $\M_N(\C)_{sa}$ is the subspace of Hermitian matrices.
		\item $\tr_N$ is the non-normalized trace on $\M_N(\C)$.
		\item $\ts_N$ is the normalized trace on $\M_N(\C)$.
		\item $(e_u)_{1\leq u\leq N}$ is the canonical basis of $\C^N$.
		\item We denote $E_{r,s}=e_re_s^*$ the matrix with $1$ in the $(r,s)$ entry and zeros in all the other entries.
	\end{itemize}
\end{defi}

\subsection{The free product of $\M_N(\C)$ and a free semicircular system}

In order to interpolate between matrices and free operators, we need to construct a space in which they can exist simultaneously. One could simply use Theorem 7.9 of \cite{nica_speicher_2006} to build the free product $\M_N(\C) * \mathcal{C}_d$ of $\M_N(\C)$ with $\mathcal{C}_d$ the $\CC^*$-algebra generated by a system of $d$ free semicircular variables, however it will be useful in the proof of Theorem \ref{mainlemma} to have a more explicit construction.

We fix $d,N\in\N$, thanks to the help of the so-called full Fock space, i.e Proposition 7.18 of \cite{nica_speicher_2006}, one can easily build an explicit  $\mathcal{C}^*$-probability space $(\A,*,\tau,\norm{.}) $ where $\tau$ is a faithful trace and in which there exists a free semicircular system  $(x^i_{r,s})_{1\leq i\leq d, 1\leq r\leq s\leq N}\cup (y^i_{r,s})_{1\leq i\leq d, 1\leq r< s\leq N} $.

Next we fix $\A_N=\M_N(\A)$, thus if $\1$ is the unit of $\A$, one can easily view $\M_N(\C)$ as a subalgebra of $\A_N$ thanks to the morphism $(a_{r,s})\in\M_N(\C)\mapsto (a_{r,s}\1)\in\A_N$. We also define $x_i^N\in\A_N$ with
\begin{equation}
	\label{sokmcmwcs}
	\sqrt{N}\ (x_i^N)_{r,s} = \left\{
	\begin{array}{ll}
		\frac{x_{r,s}^i+\i\ y_{r,s}^i}{\sqrt{2}} & \mbox{if } r<s, \\
		x^i_{r,s} & \mbox{if } r=s, \\
		\frac{x_{s,r}^i-\i\ y_{s,r}^i}{\sqrt{2}} & \mbox{if } r>s.
	\end{array}
	\right.
\end{equation}
We endow $\A_N$ with the involution ${(a_{i,j})_{1\leq i,j\leq N}}^* = (a_{j,i}^*)_{1\leq i,j\leq N}$ and the trace 
\begin{equation}
	\label{tracetaun}
	\tau_N : A\in\A_N\mapsto\tau\left(\frac{1}{N}\tr_N(A)\right),\quad \A_N=\M_N(\A).
\end{equation}
Then one has the following result.

\begin{prop}
	\label{dojvdsomv}
	With the trace $\tau_N$ and the involution defined as above, $\A_N$ is a $\mathcal{C}^*$-probability space. Besides the family $(x_i^N)_{1\leq i\leq d}$ is a free semicircular system, and it is free from $\M_N(\C)$.
\end{prop}

\begin{proof}
	If we consider $X^{kN}_i$ defined as in Equation \eqref{sokmcmwcs} but where we replaced every free semicircular variable  by independent GUE random matrices of size $k$ (see Definition \ref{3GUEdef}), then $X^{kN}$ is a $d$-tuple of independent GUE random matrices of size $kN$. Consequently, thanks to Theorem 5.4.5 of \cite{alice}, for any polynomial $P$ and for any deterministic matrices $Z_1,\dots,Z_q\in\M_N(\C)$, almost surely
	$$\lim_{k\to\infty} \frac{1}{kN}\tr_{kN}\left(P(X^{kN},Z_1\otimes I_k,\dots,Z_q\otimes I_k)\right) = \widetilde{\tau_{N}}\left(P(x,Z^N)\right)$$
	where $x$ is a free semicircular system of $d$ variables free from $\M_N(\C)$, and $\widetilde{\tau_N}$ is the trace on the free product of $\M_N(\C)$ and the $\CC^*$-algebra generated by the free semicircular system $x$. Besides, since 
	$$ \frac{1}{kN}\tr_{kN}(\cdot) = \frac{1}{k}\tr_{k}\left(\left(\frac{1}{N}\tr_{N}\otimes \id_{\M_k(\C)}\right)(\cdot)\right),$$
	one also has that
	$$\lim_{k\to\infty} \frac{1}{kN}\tr_{kN}\left(P(X^{kN},Z_1\otimes I_k,\dots,Z_q\otimes I_k)\right) =\tau\left( \frac{1}{N}\tr_{N}\left(P(x^N,Z^N)\right)\right).$$
	Hence the conclusion.
\end{proof}

In the rest of this paper, we drop the superscript $N$ in $x_i^N$ since their distribution does not depend on $N$.

\subsection{GUE random matrices}

\label{lfdkmvdlmlsm}

In this subsection we introduce a random matrix ensemble of interest and state a few useful properties about it. 

\begin{defi}
	\label{3GUEdef}
	A GUE random matrix $X^N$ of size $N$ is a self-adjoint matrix whose coefficients are random variables with the following laws:
	\begin{itemize}
		\item For $1\leq i\leq N$, the random variables $\sqrt{N} X^N_{i,i}$ are independent centered Gaussian random variables of 
		variance $1$.
		\item For $1\leq i<j\leq N$, the random variables $\sqrt{2N}\ \Re{X^N_{i,j}}$ and $\sqrt{2N}\ \Im{X^N_{i,j}}$ are independent 
		centered Gaussian random variables of variance $1$, independent of  $\left(X^N_{i,i}\right)_i$.
	\end{itemize}
\end{defi}

When doing computations with Gaussian variables, the main tool that we use is Gaussian integration by parts, which can be summarized into the following formula: If $Z$ is a real centered Gaussian random variable with variance one and $f\in \mathcal{C}^1(\R)$, then by integration by parts, as long as $\E[|Z f(Z)|]$ and $\E[|f'(Z)|]$ are finite,
\begin{equation}
	\label{3IPPG}
	\E[Z f(Z)] = \frac{1}{\sqrt{2\pi}}\int_{\R} f(x)\ xe^{-x^2/2}dx = \frac{1}{\sqrt{2\pi}}\int_{\R} f'(x) e^{-x^2/2}dx = \E[f'(Z)].
\end{equation}

\noindent As a direct consequence, if $x$ and $y$ are centered Gaussian variables with variance one, and 
$Z = \frac{x+\i y}{\sqrt{2}}$, then with $f\in \mathcal{C}^1(\C)$,
\begin{equation}
	\label{3IPPG2}
	\E[Z f(Z,\bar{Z})] = \E[\partial_1 f(Z,\bar{Z})]\quad \text{ and  }\quad \E[\bar{Z} f(Z,\bar{Z})] = \E[\partial_2 f(Z,\bar{Z})].
\end{equation}

\noindent For example we have that given a GUE random matrix $X^N$, one can write $X^N = \frac{1}{\sqrt{N}} (x_{r,s})_{1\leq r,s\leq N}$ and then for any polynomial $Q$,
\begin{align}
	\label{smcosmcds}
	\E\left[ \frac{1}{N}\tr_N\left(X^N\ Q(X^N)\right) \right] &= \frac{1}{N^{3/2}} \sum_{r,s} \E\left[ x_{r,s}\ \tr_N\left(E_{r,s}\ Q(X^N)\right) \right] \\
	&= \frac{1}{N^{3/2}} \sum_{r,s} \E\left[ \tr_N\left(E_{r,s}\ \partial_{x_{r,s}} Q(X^N)\right) \right] \nonumber \\
	&= \frac{1}{N^{2}} \sum_{r,s} \E\left[ \tr_N\left(E_{r,s}\ \partial Q(X^N) \# E_{s,r}\right) \right] \nonumber \\
	&= \frac{1}{N^{2}} \sum_{r,s} \E\left[ e_s^* \left(\partial Q(X^N) \# e_s e_r^*\right) e_r \right] \nonumber \\
	&= \E\left[ \left(\frac{1}{N}\tr_N\right)^{\otimes 2}\left(\partial Q(X^N)\right) \right], \nonumber
\end{align}
with the notation $\#$ as in Definition \ref{sknclks}, $\partial$ the non-commutative derivative defined in Definition \ref{ksdfnlsmc}, and  $(e_u)_{1\leq u\leq  N}$ the canonical basis of $\C^N$.

\subsection{The microstates free entropy}

For $d\geq 1$, we denote by $\A_d=\C\langle X_1,\dots, X_d\rangle$ the algebra of $d$-variables non-commutative polynomials. 

\begin{defi}
	A \textbf{non-commutative law} is a linear map $\lambda:\A_d \to \C$ such that
	\begin{itemize}
		\item $\lambda$ is unital, i.e. $\lambda(1)=1$,
		\item $\lambda$ is completely positive, i.e. for any matrix $Q$ with entries in $\A_d$ the matrix $\lambda(Q^*Q)$ is positive semi-definite,
		\item $\lambda$ is tracial, that is for all $P,Q$, $\lambda(PQ)=\lambda(QP)$.
	\end{itemize}
\end{defi}

\begin{defi}
	Let $x=(x_1,\dots,x_d)$ be bounded self-adjoint elements of a tracial von Neumann algebra $(\MM,\tau)$. Then the non-commutative law of $x$ is the map
	$$ \begin{array}{ccccc}
		\lambda_x & : & \A_d & \to & \C \\
		& & P & \mapsto & \tau(P(x)) \\
	\end{array}. $$
\end{defi}

Given a non-commutative law $\lambda$, we define the microstates $\Gamma_R(\lambda,n,N,\varepsilon)$ for $n,N\in\N$ and $\varepsilon>0$ as the set of self-adjoint matrices $A_1,\dots,A_d$ with $\norm{A_i}\leq R$ and such that for any $1\leq p\leq n$, $i_1,\dots,i_p\in [1,d]^p$,
\begin{equation}
	\label{soicdjsomcs}
	|\lambda(X_{i_1}\cdots X_{i_p}) - \ts_N(A_{i_1}\cdots A_{i_p})| <\varepsilon.
\end{equation}
This then allows us to define the microstates free entropy.
\begin{defi}
	\label{dovosvmdsdom}
	Given a non-commutative law $\lambda$ and $X^N$ a $d$-tuple of GUE random matrices, its microstates free entropy is
	$$ \chi_g(\lambda) = \sup_{R>0} \inf_{n\in\N} \inf_{\varepsilon>0} \limsup_{N\to\infty} \frac{1}{N^2} \log\P\left( X^N\in \Gamma_R(\lambda,n,N,\varepsilon) \right). $$
\end{defi}
Note that the original definition of the microstates free entropy of Voiculescu in \cite{freeent} was with respect to the Lebesgue measure instead of the GUE. However those definitions only differ by a quadratic term, see Lemma 2.11 of \cite{duvgui}.

\subsection{Combinatorics and non-commutative derivatives}

Non-commutative derivatives are widely used tools in Free Probability, see for example the work of Voiculescu, \cite{refdif} and \cite{refdif2}. In this subsection, we build a very specific one which we need to define properly the coefficients of the expansion. In the next definition we introduce combinatorial objects which appear in the proof of Theorem \ref{maintherorem}. However we advise the first time reader to skip the beginning of this subsection and to jump to Definition \ref{ksdfnlsmc}. Until Subsection \ref{sdocosalsw} we will not need all of the notations introduced in this subsection. Indeed although necessary the notations introduced here are heavy.

\begin{defi}
	\label{3biz}
	Let $S$ be a set whose elements are all subsets of $\N$. Let $c_S$ be the largest of those integers, and $n$ the largest cardinal of the elements of $S$. Then we define for $j\in[1,n]$,
	\begin{align*}
		F_{j}^{1}(S) &= \Big\{ \{I_m+c_S,\dots,I_{j-1}+c_S,I_{j}+c_S ,I_{j},\dots,I_{n},3c_S+1\}   \ \Big|\ I=\{I_m,\dots,I_{n}\}\in S \Big\}, \\
		F_{n+1}^{1}(S) &= \Big\{ \{I_m+c_S,\dots,I_{n}+c_S,3c_S+2,3c_S+1\}   \ \Big|\ I=\{I_m,\dots,I_{n}\}\in S \Big\}, \\
		F_{j}^{2}(S) &= \Big\{ \{I_m+2c_S,\dots,I_{j-1}+2c_S,I_{j}+2c_S,I_{j},\dots,I_{n},3c_S+1\}   \ \Big|\ I=\{I_m,\dots,I_{n}\}\in S \Big\}, \\
		F_{n+1}^{2}(S) &= \Big\{ \{I_m+2c_S,\dots,I_{n}+2c_S,3c_S+3,3c_S+1\}   \ \Big|\ I=\{I_m,\dots,I_{n}\}\in S \Big\}, \\
		G^+(S) &= \Big\{ \{I_m,\dots,I_{n},c_S+1\}   \ \Big|\ I=\{I_m,\dots,I_{n}\}\in S \Big\}.
	\end{align*}
	\noindent We similarly define $\widetilde{F}_{j}^{1}(S)$ and $\widetilde{F}_{j}^{2}(S)$ by adding $3c_S+3$ to every integer in every set. Then we define
	$$ F_j(S) = F_{j}^{1}(S)\cup F_{j}^{2}(S)\cup \widetilde{F}_{j}^{1}(S)\cup \widetilde{F}_{j}^{2}(S), $$
	$$ F(S) = \bigcup_{0\leq j\leq n} F_j(S), $$
	$$ G(S) = G^+(S)\cup \{\emptyset\}. $$
	
	\noindent Then given a sequence $H=\{H_1,\dots,H_k\}$ with $H_i\in \{F,G,F_j\}$, we set
	$$ J_{H} = H_k\circ H_{k-1} \circ\dots \circ H_1(\{\emptyset\}). $$
	Finally, we also set
	$$ J_H^h = \{ I\in J_H\ |\ \# I = n - h\}.$$

\end{defi}

The intuition behind those definitions corresponds to the following heuristic.

\begin{rem}
	To better understand the heuristic behind the set $J_H$, one can view its construction in the following way. Let us assume that we have a family of random variables $(X_i)_{i\geq 1}$ which are all independent copies of a random variable $X$, then for every element $I=\{I_m,\dots,I_{n}\}\in J_H$, one can associate an ordered sequence of length $n-m+1$ of those random variables, i.e. $(X_{I_m},\dots,X_{I_{n}})$. Then the construction of the family $J_{\{H,H_{k+1}\}}$ is associated to the following process:
	\begin{itemize}
		\item If $H_{k+1} = G^+$, given a list $\{I_m,\dots,I_{n}\}\in J_H$ and its associated sequence of random variables $(X_{I_m},\dots,X_{I_{n}})$, let us add a new independent copy of $X$ to this sequence. Thus if we do so for every sequence, the new family of sequences defined with this process will be indexed by $J_{\{H,G^+\}}$.
		\item If $H_{k+1} = F^1_j$, given a list $\{I_m,\dots,I_{n}\}\in J_H$ and its associated sequence of random variables $(X_{I_m},\dots,X_{I_{n}})$, let first us add a new independent copy of $X$ at the end of the sequence. Secondly we insert another independent copy in the $j$-th position. Finally we replace all of the variables whose position is strictly smaller than $j$ by independent copies of $X$. Thus if we do so for every sequence, the new family of sequences defined with this process will be indexed by $J_{\{H,F_j^1\}}$.
		\item The case of $H_{k+1} = F^2_j$ is defined similarly with the difference that every new random variable that we replaced or inserted are once again replaced by another independent copy with the exception of the one at the end of the sequence.	
		\item Finally the case of $H_{k+1} = \widetilde{F}^1_j$ (respectively $H_{k+1} = F^2_j$) consists in taking independent copies of every sequences built in the case of $H_{k+1} = F^1_j$ (respectively $H_{k+1} = F^2_j$).
	\end{itemize}
	This construction will appear naturally in the proof of Theorem \ref{mainlemma}.
\end{rem}

We can now define the non-commutative polynomials associated to these sets.

\begin{defi}
	\label{sidcosc}
	Let $H$ be as in the previous definition. We define the following quantities:
	\begin{itemize}
		\item $\A_{d}^H = \C\langle X_{i,I},\ 1\leq i\leq d, I\in J_H \rangle$ the space of non-commutative polynomials. Besides, if $H$ is empty, i.e. $H=\emptyset$, then $J_H=\{\emptyset\}$ and $\A_{d}^{\emptyset} =\A_{d}$.
		\item Since $F_{j}^{1}, F_{j}^{2}, \widetilde{F}_{j}^{1}, \widetilde{F}_{j}^{2}$ and $G^+$ induce a bijection between the set $S$ and the resulting one, with $H=\{H_1,\dots,H_k\}$ such that for all $i$, $H_i$ is one of the previously mentioned functions, with $X = (X_{i,I})_{i\in [1,d],I\in J_{H}}$, we define 
		$$H_{k+1}(X) = (X_{i,I})_{i\in [1,d],I\in J_{\{H_1,\dots,H_{k+1}\}}}.$$
	\end{itemize}
\end{defi}

There is naturally a notion of degree associated to those polynomials.

\begin{defi}
	\label{degree2}
	Given $M\in \A_{d}^H$ a monomial, we denote $\deg M$ the length of $M$ as a word in the variables $(X_{i,I})_{i\in[1,d], I\in J_H}$.
\end{defi}

The previous definitions are not exactly intuitive, however this construction will appear naturally in Subsection \ref{sdocosalsw}. We also refer to Remark 2.19 of \cite{trois} for some intuition.

\begin{defi}
	\label{3biz2}
	We define the \textbf{non-commutative derivative} $\partial_{i,I}:\A_{d}^H\to\A_{d}^H\otimes\A_{d}^H$ as 
	$$\forall P,Q,\quad \partial_{i,I} (PQ) = \partial_{i,I} P \times \left(1\otimes Q\right)	 + \left(P\otimes 1\right) \times \partial_{i,I} Q,$$
	and $\forall k\in [1,d],\ K\in J_H$,
	$$ \partial_{i,I} X_{k,K} = \1_{i=k}\1_{I=K}\  1\otimes 1. $$
	We also set
	$$ \partial_{i,h} = \sum_{I\in J_H^h} \partial_{i,I},\quad \partial_{i} = \sum_{I\in J_H} \partial_{i,I}. $$
	We then define the \textbf{cyclic derivatives}
	$$\D_{i} = m \circ \partial_{i},\quad \D_{i,h} = m \circ \partial_{i,h},\quad \D_{i,I} = m \circ \partial_{i,I},$$
	with $m: A\otimes B\mapsto BA$.
\end{defi}

Note that this definition is a straightforward generalization of the following usual definition of non-commutative derivative.

\begin{defi}
	\label{ksdfnlsmc}
	We define the non-commutative derivative $\partial_{i}:\C\langle X_1,\dots,X_d\rangle \to\C\langle X_1,\dots,X_d\rangle\otimes\C\langle X_1,\dots,X_d\rangle$ as 
	$$\forall P,Q,\quad \partial_{i} (PQ) = \partial_{i} P \times \left(1\otimes Q\right)	 + \left(P\otimes 1\right) \times \partial_{i} Q,$$
	and $\forall k\in [1,d]$,
	$$ \partial_{i} X_{k} = \1_{i=k} 1\otimes 1. $$
	We then define the cyclic derivatives $\D_{i} = m \circ \partial_{i}$ with $m: A\otimes B\mapsto BA$.
\end{defi}

When handling non-commutative derivatives, the following definition will often be helpful.

\begin{defi}
	\label{sknclks}
	Let $\A$ be an algebra, given $A,B,C\in\A$, we define
	$$ (A\otimes B)\# C= ACB.$$
	Besides, we extend this definition to any finite sum of simple tensors by linearity, that is
	$$ \sum_{i\in I} (A_i\otimes B_i) \#C = \sum_{i\in I} A_i C B_i. $$
\end{defi}

Finally to conclude this subsection, we note that non-commutative derivatives are related to the so-called Schwinger-Dyson equations on semicircular variables thanks to the following proposition. One can find a proof in Lemma 5.4.7 of \cite{alice}.

\begin{prop}
	\label{3SDE}
	Let $ x=(x_1,\dots ,x_d)$ be a free semicircular system, $y = (y_1,\dots,y_r)$ be non-commutative random variables free from $x$, if the 
	family $(x,y)$ belongs to a $\CC^*$-probability space $(\A,*, \tau,\norm{.}) $, then for any polynomial $Q$ and $i\in [1,d]$,
	$$ \tau(Q(x,y)\ x_i) = \tau\otimes\tau(\partial_i Q(x,y)).$$
\end{prop}

\section{Building blocks of the proofs}

The aim of this section is to state and prove the building blocks of the main results of this paper. We start by adapting the usual Schwinger-Dyson Equation for models with a potential to a version of the model with a cut-off. In the next subsection we give some assumptions on our potentials for our model of random matrices to be well-defined. Finally we state and prove the so-called ``master equation'' which will let us prove the asymptotic expansion by iterating it.

\subsection{Schwinger-Dyson Equation for models with a cut-off}

In order to study matrix models with a cut-off, we need to find an alternative to Gaussian integration by parts introduced in Subsection \ref{lfdkmvdlmlsm}. We do so in the following lemma through a change of variables. This is a method which is well-known when deriving so-called loop-equations, see \cite{loopeq} for multiple examples.

\begin{lemma}
	\label{ippcut}
	Let $V$ be a potential, we set
	$$ d\mu^N_{V,K}(X) = \frac{1}{Z_V^N}e^{- N\tr_N(V(X))} \1_{\forall i, ||X_i^N||\leq K} d\mu^N(X),$$
	where $\mu^N$ is the law of a $d$-tuple of independent GUE random matrices. Then with $X^N$ a $d$-tuple of independent GUE random matrices, for any constant $M<K$, one has
	\begin{align*}
		&\left|\E\left[ \left(\ts_N\otimes\ts_N(\partial_i P(X^N)) - \ts_N\left(P(X^N)(\D_i V(X^N)+X_i^N)\right)\right) e^{- N\tr_N\left(V(X^N)\right)} \1_{\forall i, ||X_i^N||\leq K} \right]\right| \\
		&\leq \frac{2}{N^2(K-M)} \norm{P}_K\ \mu_{V,K}^N\left(\max_{i\leq i\leq d} \norm{X_i} \geq M\right) \E\left[ e^{- N\tr_N\left(V(X^N)\right)} \1_{\forall i, ||X_i^N||\leq K}\right],
	\end{align*}
	where $\norm{P}_K$ is defined as follows: if $P=\sum_{R \text{ monomial}} c_R R$ with $c_R\in\C$, then 
	\begin{equation}
		\label{sodcmosmcowkmcd}
		\norm{P}_K = \sum_{R \text{ monomial}} |c_M| K^{\deg R}.
	\end{equation}
\end{lemma}

\begin{proof}
	To begin with, note that
	\begin{align*}
		&\E\left[P(X^N) e^{- N\tr_N\left(V(X^N)\right)} \1_{\forall i, ||X_i^N||\leq K} \right] \\
		&= \frac{1}{Z_N} \int_{\forall i, ||X_i^N||\leq K} P(X) e^{- N\tr_N\left(V(X)+ \frac{1}{2}\sum_{i=1}^dX_i^2\right)} dX_1\dots dX_d, 
	\end{align*}
	where the integral is over the set of Hermitian matrices of size $N$. Thus we set $E=\{X_1,\dots,X_d\in\M_N(\C)_{sa}\ |\ \forall i, ||X_i^N||\leq K\}$ and $h:\R\to\R$ a $\mathcal{C}^1$-function such that $h=1$ on an open neighborhood of $[-M,M]$ and $h=0$ on an open neighborhood of the complementary set of $[-L,L]$ where $L\in (M,K)$. Then for $\varepsilon>0$ and $H=(H_1,\dots,H_d)\in \M_N(\C)_{sa}^d$, we define
	$$\begin{array}{ccccc}
		\Phi_{\varepsilon} & : & E & \to & E \\
		& & X & \mapsto & \left(X_i + \varepsilon H_i h\left(\norm{X_i}_{2p}\right)\right)_{1\leq i\leq d}  \\
	\end{array}$$
	where $\norm{X}_{2p} = \left(\ts_N(X^{2p})\right)^{\frac{1}{2p}}$. Note that one has
	$$ N^{-\frac{1}{2p}} \leq \frac{\norm{X}_{2p}}{\norm{X}} \leq 1.$$
	Consequently, for $p$ sufficiently large, $h\left(\norm{X_i}_{2p}\right)$ is equal to $1$ as long as $\norm{X_i} \leq M$ and $h\left(\norm{X_i}_{2p}\right)$ is equal to $0$ if $\norm{X_i} \geq L$. Thus we claim $\Phi_{\varepsilon}$ is a diffeomorphism of $E$ if $\varepsilon$ is sufficiently small. Indeed, with $\norm{(X_1,\dots,X_d)} = \max_i \norm{X_i}$, we have the following:
	\begin{itemize}
		\item For all $X\in E$, $\norm{\Phi_{\varepsilon}(X)} \leq \norm{X}\vee \left(L+ \varepsilon \norm{H}\right)$ which is smaller than $K$, as long as $\varepsilon \norm{H}$ is smaller than $K-L$. Hence $\Phi_{\varepsilon}(E)\subset E$.
		\item Let us look for $\lambda_1,\dots,\lambda_d \geq 0$ such that $\Phi_{\varepsilon}(X-\lambda H) = X$, i.e. we want to find $\lambda = (\lambda_1,\dots,\lambda_d)$ such that for all $i$,
		$$ g(\lambda_i)\deq\varepsilon h(\norm{X_i-\lambda_iH_i}_{2p}) - \lambda_i =0.$$
		However $g(0)\geq 0$ and $\lim_{\lambda\to\infty} g(\lambda) = -\infty$, hence by continuity of $g$ one can always find such a $\lambda_i$. Besides, if $g(\lambda_i)= 0$, then either $\lambda_i>0$, hence $h(\norm{X_i- \lambda_i H_i}_{2p})>0$ and consequently $\norm{X_i- \lambda_i H_i} \leq L <K$, otherwise $\lambda_i=0$ and then $\norm{X_i-\lambda_iH_i} = \norm{X_i} \leq K$. Consequently, $X-\lambda H\in E$ and $\Phi_{\varepsilon}$ is surjective.
		\item Let $X,Y\in E$ be such that $\Phi_{\varepsilon}(X)= \Phi_{\varepsilon}(Y)$, then 
		\begin{align*}
			\norm{X-Y} &\leq \varepsilon \norm{H}\ \max_i \left| h\left(\norm{X_i}_{2p}\right) - h\left(\norm{Y_i}_{2p}\right) \right| \\
			&\leq \varepsilon \norm{H} \sup_{t\in\R}|h'(t)|\ \max_i \left| \norm{X_i}_{2p} - \norm{Y_i}_{2p} \right| \\
			&\leq \varepsilon \norm{H} \sup_{t\in\R}|h'(t)|\  \norm{X-Y}.
		\end{align*}
		Thus as long as $\varepsilon \norm{H} \sup_{t\in\R}|h'(t)| <1$, $\Phi_{\varepsilon}$ is injective.
	\end{itemize}
	
	Thus for any $\varepsilon>0$ sufficiently small, one has that
	\begin{align*}
		&\E\left[P(X^N) e^{- N\tr_N\left(V(X^N)\right)} \1_{\forall i, ||X_i^N||\leq K} \right] \\
		&= \frac{1}{Z_N} \int_{\forall i, ||X_i^N||\leq K} \ts_N\left(P(\Phi_{\varepsilon}(X))\right) e^{- N\tr_N\left(V(\Phi_{\varepsilon}(X))+ \frac{1}{2}\sum_{i=1}^d\Phi_{\varepsilon}(X)_i^2\right)}  |\det \mathrm{Jac}(\Phi_{\varepsilon})| dX_1\dots dX_d, \\
		&= \E\left[P(X^N) e^{- N\tr_N\left(V(X^N)\right)} \1_{\forall i, ||X_i^N||\leq K} \right] \\ 
		&\quad +\varepsilon\sum_{i=1}^d \E\left[ \left(\partial_i P(X^N)\# H_i - NP(X^N) \tr_N\left((\D_i V(X^N)+X_i^N)H_i\right)\right) e^{- N\tr_N\left(V(X^N)\right)} \1_{\forall i, ||X_i^N||\leq K} \right] \\
		&\quad +\varepsilon\sum_{i=1}^d \E\left[ P(X^N) h'(\norm{X_i}_{2p}) \frac{\ts_N(X_i^{2p-1}H_i)}{\ts_N(X_i^{2p})}\norm{X}_{2p}  e^{- N\tr_N\left(V(X^N)\right)} \1_{\forall i, ||X_i^N||\leq K} \right] \\
		&\quad + \mathcal{O}(\varepsilon^2),
	\end{align*}
	with the notation $\#$ as in Definition \ref{sknclks}. Consequently, 
	\begin{align}
		\label{lksmcsmssss}
		&\sum_{i=1}^d \E\left[ \left(\partial_i P(X^N)\# H_i - NP(X^N) \tr_N\left((\D_i V(X^N)+X_i^N)H_i\right)\right) e^{- N\tr_N\left(V(X^N)\right)} \1_{\forall i, ||X_i^N||\leq K} \right] \\
		&=- \sum_{i=1}^d \E\left[ P(X^N) h'(\norm{X_i}_{2p}) \frac{\ts_N(X_i^{2p-1}H_i)}{\ts_N(X_i^{2p})}\norm{X}_{2p}  e^{- N\tr_N\left(V(X^N)\right)} \1_{\forall i, ||X_i^N||\leq K} \right]. \nonumber
	\end{align}
	Since the equation above is linear with respect to $H$ and that every matrix can be written as a linear combination of Hermitian matrices ( $A = \frac{A+A^*}{2}-\i \frac{\i A-\i A^*}{2}$), the formula above remains true even if $H$ is not Hermitian. Thus if we set $H_j= E_{r,s}$ if $j=i$ and $0$ else, then thanks to the equation above, one has that
	\begin{align*}
		&\E\left[ \left(\partial_i P(X^N)\# E_{r,s} - NP(X^N) \tr_N\left((\D_i V(X^N)+X_i^N)E_{r,s}\right)\right) e^{- N\tr_N\left(V(X^N)\right)} \1_{\forall i, ||X_i^N||\leq K} \right] \\
		&= - \E\left[ P(X^N) h'(\norm{X_i}_{2p}) \frac{\ts_N(X_i^{2p-1}E_{r,s})}{\ts_N(X_i^{2p})}\norm{X}_{2p}  e^{- N\tr_N\left(V(X^N)\right)} \1_{\forall i, ||X_i^N||\leq K} \right].
	\end{align*}
	Thus by multiplying by $e_s$ on the right and $e_r^*$ on the left, after summing over $r$ and $s$, as well as dividing by $N^2$, one has that
	\begin{align}
		&\E\left[ \left(\ts_N\otimes\ts_N(\partial_i P(X^N)) - \ts_N\left(P(X^N)(\D_i V(X^N)+X_i^N)\right)\right) e^{- N\tr_N\left(V(X^N)\right)} \1_{\forall i, ||X_i^N||\leq K} \right] \\
		&= - \frac{1}{N^2} \E\left[ h'(\norm{X_i}_{2p}) \frac{\ts_N(P(X^N)X_i^{2p-1})}{\ts_N(X_i^{2p})}\norm{X}_{2p}  e^{- N\tr_N\left(V(X^N)\right)} \1_{\forall i, ||X_i^N||\leq K} \right]. \nonumber
	\end{align}
	We therefore have,
	\begin{align*}
		&\left|\E\left[ \left(\ts_N\otimes\ts_N(\partial_i P(X^N)) - \ts_N\left(P(X^N)(\D_i V(X^N)+X_i^N)\right)\right) e^{- N\tr_N\left(V(X^N)\right)} \1_{\forall i, ||X_i^N||\leq K} \right]\right| \\
		&\leq \frac{1}{N^2} \E\left[ |h'(\norm{X_i}_{2p})| \norm{P(X^N)}  e^{- N\tr_N\left(V(X^N)\right)} \1_{\forall i, ||X_i^N||\leq K} \right] \\
		&\leq \frac{\sup_{t\in\R}|h'(t)|}{N^2} \E\left[ \1_{\norm{X_i}\geq M} \norm{P(X^N)}  e^{- N\tr_N\left(V(X^N)\right)} \1_{\forall i, ||X_i^N||\leq K} \right] \\
		&\leq \frac{\sup_{t\in\R}|h'(t)|}{N^2} \norm{P}_K\ \mu_{V,K}^N\left(\max_{i\leq i\leq d} \norm{X_i} \geq M\right) \E\left[ e^{- N\tr_N\left(V(X^N)\right)} \1_{\forall i, ||X_i^N||\leq K}\right].
	\end{align*}
	Finally, let us remember that $h:\R\to\R$ is any $\mathcal{C}^1$-function such that $h=1$ on an open neighborhood of $[-M,M]$ and $h=0$ on an open neighborhood of the complementary set of $[-L,L]$ where $L\in (M,K)$. Thus if $M<a<b<K$, and $g$ is the primitive which takes value $1$ in $0$ of the function equal to $0$ on $(-\infty,a]\cup [b,\infty)$, $-\frac{2}{b-a}$ in $\frac{a+b}{2}$ and piecewise affine otherwise, then one can take $h:t\in\R\mapsto g(t)g(-t)$. And for all $M<a<b<K$, one thus has that 
	$$ \sup_{t\in\R}|h'(t)| \leq \frac{2}{b-a}.$$
	Hence the conclusion by letting $a$ go to $M$ and $b$ to $K$.
\end{proof}

\subsection{Assumption on the potential}

In order to study matrix models without a cut-off it is first necessary to ensure that our random variables are actually integrable, i.e that for any polynomial $P$,
$$ \E\left[\left|\ts_N\left(P(X^N)\right)\right| e^{-\lambda N\tr_N\left(V(X^N)\right)}\right] <\infty. $$
Besides, the proof of expansions such as the one of Theorem \ref{maintherorem} usually require some sort of concentration estimates. In this paper we will assume that our potential $V$ satisfies the following assumption.

\begin{assum}
	\label{kdmcslcs0}
	We say that a potential $V$ is well-behaved if $X\in\M_N(\C)_{sa}^d\mapsto \tr_N(V(X))$ is real-valued and there exists a constant $C$ and a sequence $u_N\gg \log(N)$ such that for any $N\in\N$, $\lambda\in [0,1]$, $i\in [1,d]$ and $k\leq u_N$,
	$$ \frac{\E\left[\ts_N\left((X_i^N)^{2k}\right) e^{-\lambda N\tr_N\left(V(X^N)\right)}\right]}{ \E\left[ e^{-\lambda N\tr_N\left(V(X^N)\right)}\right]} \leq C^k, $$
	where $X^N$ is a $d$-tuple of independent GUE random matrices.
\end{assum}

This assumption is known to be satisfied for class of polynomial potentials, see notably the recent work of Guionnet and Maurel-Segala \cite{conftr} which introduced the notion of confining and $(\eta,A,I)$-trapping polynomials. In particular, Theorem 2.2 of \cite{conftr} implies that polynomials which are $(\eta,A,I)$-trapping also satisfy Assumption \ref{kdmcslcs0}. It is further well-known that trace-convex polynomials satisfy Assumption \ref{kdmcslcs0}. More precisely we have the following assumption.

\begin{assum}
	\label{kdmcslcs}
	We say that a potential $V$ is $c$-convex with $c>0$ if it is trace self-adjoint, i.e. that $X\in\M_N(\C)_{sa}^d\mapsto\tr_N(V(X))$ is real-valued for any $N$, and $$(X_1,\dots,X_d)\in\M_N(\C)_{sa}^d\mapsto \tr_N\left(V(X)+\frac{1-c}{2}\sum_{i=1}^{d}X_i^2\right)$$ is convex.
\end{assum}

Then Assumption \ref{kdmcslcs} implies Assumption \ref{kdmcslcs0} thanks to the following lemma which is a rather direct consequence of Lemma 2.2 of \cite{gui-seg-2}.

\begin{lemma}
	\label{semvpeme}
	We introduce the probability measure
	$$ d\mu^N_{V}(X) = \frac{1}{Z_V^N}e^{- N\tr_N(V(X))} d\mu^N(X),$$
	where $\mu^N$ is the law of a $d$-tuple of independent GUE random matrices, and $Z_V^N$ a normalizing constant. Then if $V$ is $c$-convex, there exist $\alpha,\lambda_0>0$ and $M_0<\infty$ such that for all $\lambda\in[0,1]$, $M\geq M_0$ and all integer $N$,
	\begin{equation}
		\label{skjdmcskc}
		\mu_{\lambda V}^N\left(\max_{i\leq i\leq d} \norm{X_i} > M\right) \leq e^{-\alpha MN}.
	\end{equation}
	Besides with $X^N$ a $d$-tuple of independent GUE random matrices, there exists a constant $C$ such that for all $\lambda\in[0,1]$, $i\in [1,d]$ and $k\leq \frac{\alpha}{2}N$.
	\begin{equation}
		\E\left[ \norm{X_i^N}^k e^{-\lambda N\tr_N(V(X^N))} \right] \leq C^k \E\left[ e^{-\lambda N\tr_N(V(X^N))} \right].
	\end{equation}
\end{lemma}

\begin{proof}
	If $V$ is $c$-convex then so is $\lambda V$ for $\lambda\in [0,1]$, thus one can use Lemma 2.2 of \cite{gui-seg-2} with $\eta$ the maximum of the coefficients of $V$ which immediately yields Equation \eqref{skjdmcskc}. Besides,
	\begin{align*}
		\frac{\E\left[ \norm{X_i^N}^k e^{-\lambda N\tr_N(V(X^N))} \right]}{\E\left[ e^{-\lambda N\tr_N(V(X^N))}\right]} &= k \int_0^{\infty} \mu_{\lambda V}^N\left(\norm{X_i}>u\right) u^{k-1} du \\
		&\leq k M_0^k + k \int_{M_0}^{\infty} e^{-\alpha uN} u^{k-1} du \\
		&\leq k M_0^k + k \int_{M_0}^{\infty} e^{u(k-\alpha N)} du \\
		&\leq k M_0^k + k \int_{M_0}^{\infty} e^{-u\alpha N/2} du \\
		&\leq k M_0^k + 1.
	\end{align*}
	Hence the conclusion.
\end{proof}

In the case where we have a cut-off, we do not need the assumptions introduced previously. However, it will be useful to have some control on the quantity
$$ \mu_{V,K}^N\left(\max_{i\leq i\leq d} \norm{X_i} \geq M\right),$$
which appears in Lemma \ref{ippcut}. Thanks to Proposition 7.1 of \cite{gui-seg-2}, one has that the following lemma.

\begin{lemma}
	\label{vkmslmlskmvlsmv}
	Given $V$ a trace self-adjoint polynomial, there exist nonnegative constants $K_0,M_0,\alpha$ such that for $K\geq K_0$, we can find a constant $c_{V,K}$ such that for $\lambda\in [-c_{V,K},c_{V,K}]$, for all $M\geq M_0$,
	$$ \mu_{\lambda V,K}^N\left(\max_{i\leq i\leq d} \norm{X_i} \geq M\right) \leq e^{-\alpha M N}.$$	
\end{lemma}

\subsection{The master equation}

\label{sdocosalsw}

The objective of this subsection is to prove the so-called master equation in Theorem \ref{mainlemma} below, as well as its equivalent for random matrix models with a cut-off in Theorem \ref{mainlemma2}. As we will see, we can deduce Theorem \ref{maintherorem} by iterating it repetitively while controlling the error term. Moreover it is worth noting that unlike Theorem \ref{maintherorem}, we hardly need any assumption on the potential $V$. Indeed, besides a way to check that our random variables are integrable, which is a much weaker assumption than Assumption \ref{kdmcslcs0}, we do not even need $V$ to be polynomial. It would be enough to assume that $V$ is once differentiable in the usual sense of the term as a function on $\M_N(\C)_{sa}^d$. Finally, recall Definitions \ref{3biz} and \ref{sidcosc} as well as the notations of Proposition \ref{dojvdsomv}.

\begin{theorem}
	\label{mainlemma}
	Let the following objects be given:
	\begin{itemize}
		\item $P\in\A_{d}^H$,
		\item $V\in \A_{d}$ such that for any polynomial $Q$,
		$$ \E\left[\left|\ts_N\left(Q(X^N)\right) e^{-\lambda N\tr_N\left(V(X^N)\right)} \right| \right] <\infty, $$
		\item $X^N$ a family of $d$ independent GUE matrices,
		\item $x,x^1,x^2,\dots\in\A_N$ an infinite sequence of free semicircular systems of $d$ variables, freely independent of each other.
	\end{itemize}
	Then with $n$ the largest cardinal of the elements of $J_H$, $T_n = \{t_1,\dots,t_{n}\}$ a sequence of non-negative number,  $\widetilde{t}_1,\dots,\widetilde{t}_{n}$ the elements of $T_n$ ordered increasingly, for $m\leq n+1$, $i\in [1,d]$ and $I = \{I_m,\dots,I_{n}\}\in J_H$, with $t_0=0$, we set
	$$ X_{i,I}^{N,T_{n}} =e^{\widetilde{t}_{m-1}/2}  \left(\sum_{l=m}^{n} (e^{-\widetilde{t}_{l-1}} -e^{-\widetilde{t}_l})^{1/2} x^{I_{l}}_i + e^{-\widetilde{t}_{n} /2} X_i^N \right), $$
	$$ x_{i,I}^{T_{n}} = e^{\widetilde{t}_{m-1}/2} \left(\sum_{l=m}^{n} (e^{-\widetilde{t}_{l-1}} -e^{-\widetilde{t}_l})^{1/2} x^{I_{l}}_i + e^{- \widetilde{t}_{n}/2} x_i \right). $$
	Then
	\begin{align}
		\label{sicdsmsc}
		&\E\left[ \tau_N\left(P\left(X^{N,T_n}\right)\right)\ e^{-N\tr_N(V(X^N))} \right]  \\
		= &\ \tau\left(P(x^{T_n})\right)\ \E\left[ e^{-N\tr_N(V(X^N))} \right] \nonumber \\
		&- \int_{\widetilde{t}_n}^{\infty} \E\left[ \tau_N\left(\nabla_V^{H,T_{n+1}}(P)\left(X^{N,T_{n+1}}\right)\right)\ e^{-N\tr_N(V(X^N))} \right] dt_{n+1} \nonumber \\
		&+ \frac{1}{N^2} \int_{\widetilde{t}_{n}}^{\infty}\int_0^{t_{n+1}} \E\left[ \tau_N\left(L^{H,T_{n+2}}(P)\left(X^{N,T_{n+2}}\right)\right)\ e^{-N\tr_N(V(X^N))} \right] dt_{n+2} dt_{n+1}, \nonumber
	\end{align}
	where $\nabla_V^{H,T_{n+1}}: \A_{d}^H\to \A_{d}^{\{H,G\}}$ is given by
	\begin{align}
		\nabla_V^{H,T_{n+1}}(Q) \deq \frac{1}{2}\sum_{1\leq i\leq d}\ \sum_{0\leq h\leq n}\ e^{(\widetilde{t}_h-t_{n+1})/2}\ \D_{i,h}Q(G^+(X))\ \D_iV(X^{\emptyset}),
	\end{align}
	with $X^{\emptyset} = (X_{i,\emptyset})_{i\in [1,d]}$ and $G^+(X)$ is as in Definition \ref{sidcosc}. The operator $L^{H,T_{n+2}}$ is defined as follows. We first define the operators $L_s^{H,T_{n+2}}:\A_{d}^{H}\to \A_{d}^{\{H,F_s\}}$, for $s$ from $1$ to $n+1$ by
	\begin{align*}
		&L_s^{H,T_{n+2}}(Q) \deq \frac{1}{2} \sum_{1\leq i,j\leq d} \sum_{\substack{0\leq h,k \leq n \\ 0\leq x,y\leq s-1}} e^{(\widetilde{t}_h+\widetilde{t}_k+\widetilde{t}_y+\widetilde{t}_x)/2-t_{n+1}-t_{n+2}} \\
		&\quad\quad\quad\quad\quad\quad\quad\quad\quad\quad\quad\quad\quad\quad\quad\times \sum_{\substack{I\in J_H^x, J\in J_H^y \\ \text{ such that } I_s=J_s}} \Theta^{F_s^1,\widetilde{F}_s^1,\widetilde{F}_s^2,F_s^2}\Big([\partial_{j,I} \otimes \partial_{j,J}]\circ \partial_{i,k}\circ\D_{i,h}Q\Big),
	\end{align*}
	where for $A,B,C,D\in \A_{d}^H$,
	$$ \Theta^{F_s^1,\widetilde{F}_s^1,\widetilde{F}_s^2,F_s^2}(A\otimes B\otimes C \otimes D) = B(F_s^1(X)) A(\widetilde{F}_s^1(X)) D(\widetilde{F}_s^2(X)) C(F_s^2(X)). $$
	
	\noindent Note that since $I\in J_H$ is always written $I=\{I_m,\dots,I_n\}$, the condition "$I,J\in J_H$, such that $I_{n+1}=J_{n+1}$" is satisfied for any $I,J$. Finally, we define $L^{H,T_{n+2}}:\A_{d}^{H}\to \A_{d}^{\{H,F\}}$ as
	\begin{equation}
		\label{3fullop}
		L^{H,T_{n+2}}(Q) \deq \sum_{1\leq s\leq n+1} \1_{[\widetilde{t}_{s-1},\widetilde{t}_s]}(t_{n+2})\ L_s^{H,T_{n+2}}(Q).
	\end{equation}
	Note that $\widetilde{t}_s$ above is the $l$-th largest element of $T_{n+1}$ and not of $T_{n+2}$.
	
\end{theorem}

\begin{proof}
	
	Since there is a bijection between $J_H$ and $J_{\{H,G^+\}}$, one can interpolate $X^{N,T_n}$ and $x^{N,T_n}$ with $G^+(X)^{N,\{T_n,\widetilde{t}_n+t\}}$, where $G^+(X)$ is as in Definition \ref{sidcosc}. Thus one has
	\begin{align}
		\label{kdscsl}
		&\E\left[ \tau_N\left(P\left(X^{N,T_{n}}\right)\right)\ e^{-N\tr_N(V(X^N))} \right] \nonumber \\
		= &\ \tau\left(P(x^{T_n})\right)\ \E\left[ e^{-N\tr_N(V(X^N))} \right] \\
		&- \int_{0}^{\infty} \E\left[ \frac{d}{dt}\tau_N\left(P\left(G^+(X)^{N,\{T_n,\widetilde{t}_n+t\}}\right)\right)\ e^{-N\tr_N(V(X^N))} \right] dt. \nonumber
	\end{align}
	Moverover, one has 
	\begin{align*}
		&\frac{d}{dt}\tau_N\left(P\left(G^+(X)^{N,\{T_n,\widetilde{t}_n+t\}}\right)\right) \\
		&= \frac{e^{-(t+\widetilde{t}_{n})/2}}{2} \sum_{1\leq i\leq d} \sum_{0\leq h\leq n}\ e^{\widetilde{t}_h/2}\ \tau_N\left(\D_{i,h} P\left(G^+(X)^{N,\{T_n,\widetilde{t}_n+t\}}\right) \left( \frac{e^{-t/2}x_i}{(1-e^{-t})^{1/2}} - X_i^N\right) \right).
	\end{align*}
	
	\noindent We also have thanks to Proposition \ref{3SDE},
	\begin{align*}
		&e^{-t/2}(1-e^{-t})^{-1/2} \tau_N\left(\D_{i,h} P\left(G^+(X)^{N,\{T_n,\widetilde{t}_n+t\}}\right) x_i\right) \\
		&= e^{-(t+\widetilde{t}_{n})/2} \sum_{0\leq k\leq n}\ e^{\widetilde{t}_k/2}\ \tau_N\otimes\tau_N\left(\partial_{i,k}\D_{i,h} P\left(G^+(X)^{N,\{T_n,\widetilde{t}_n+t\}}\right)\right),
	\end{align*}
	and with the notation $\#$ as in Definition \ref{sknclks}, one has
	\begin{align}
		\label{lkmsdvmdss}
		&\E\left[ \tau_N\left(\D_{i,h} P\left(G^+(X)^{N,\{T_n,\widetilde{t}_n+t\}}\right) X_i^N \right)\ e^{-N\tr_N(V(X^N))} \right] \\
		&= e^{-(t+\widetilde{t}_{n})/2} \sum_{0\leq k\leq n}\ e^{\widetilde{t}_k/2}\ \E\left[ \frac{1}{N}\sum_{1\leq u,v\leq N} \tau_N\left(E_{u,v}\ \partial_{i,k}\D_{i,h} P\left(G^+(X)^{N,\{T_n,\widetilde{t}_n+t\}}\right)\# E_{v,u}\right) e^{-N\tr_N(V(X^N))} \right] \nonumber\\
		&\ \ \  - \E\left[ \tau_N\left(\D_{i,h} P\left(G^+(X)^{N,\{T_n,\widetilde{t}_n+t\}}\right) \D_iV(X^N) \right)\ e^{-N\tr_N(V(X^N))} \right]. \nonumber
	\end{align}
	Thus by plugging in the last three equations in Equation \eqref{kdscsl}, we get that
	\begin{align*}
		&\E\left[ \tau_N\left(P\left(X^{N,T_{n}}\right)\right)\ e^{-N\tr_N(V(X^N))} \right] \\
		= &\ \tau\left(P(x^{T_n})\right)\ \E\left[ e^{-N\tr_N(V(X^N))} \right] \\
		&- \int_{0}^{\infty} \frac{e^{-(t+\widetilde{t}_{n})/2}}{2} \sum_{1\leq i\leq d} \sum_{0\leq h\leq n}\ e^{\widetilde{t}_h/2}\ \E\left[ \tau_N\left(\D_{i,h} P\left(G^+(X)^{N,\{T_n,\widetilde{t}_n+t\}}\right)\D_iV(X^N) \right)\ e^{-N\tr_N(V(X^N))} \right] dt \\
		&+ \int_{0}^{\infty} \frac{e^{-t-\widetilde{t}_{n}}}{2} \sum_{1\leq i\leq d} \sum_{0\leq h,k \leq n} e^{(\widetilde{t}_h+\widetilde{t}_k)/2} \E\Bigg[  \Bigg(\frac{1}{N}\sum_{1\leq u,v\leq N} \tau_N\left(E_{u,v}\ \partial_{i,k}\D_{i,h} P\left(G^+(X)^{N,\{T_n,\widetilde{t}_n+t\}}\right)\# E_{v,u}\right) \\
		&\quad\quad\quad\quad\quad\quad\quad\quad\quad\quad\quad\quad\quad- \tau_N\otimes\tau_N\left(\partial_{i,k}\D_{i,h} P\left(G^+(X)^{N,\{T_n,\widetilde{t}_n+t\}}\right)\right) \Bigg) e^{-N\tr_N(V(X^N))} \Bigg] dt.
	\end{align*}
	Thus after a change of variables, and renaming $t$ into $t_{n+1}$, we get that
	\begin{align}
		\label{lskdmcosmcosancdoakc}
		&\E\left[ \tau_N\left(P\left(X^{N,T_{n}}\right)\right)\ e^{-N\tr_N(V(X^N))} \right] \\
		= &\ \tau\left(P(x^{T_n})\right)\ \E\left[ e^{-N\tr_N(V(X^N))} \right] \nonumber\\
		&- \int_{\widetilde{t}_n}^{\infty} \E\left[ \tau_N\left(\nabla_V^{H,T_{n+1}}(P) \left(X^{N,\{T_{n},t_{n+1}\}}\right)\right)\ e^{-N\tr_N(V(X^N))} \right] dt_{n+1} \nonumber \\
		&+ \int_{\widetilde{t}_{n}}^{\infty} \frac{e^{-t_{n+1}}}{2} \sum_{1\leq i\leq d} \sum_{0\leq h,k \leq n} e^{(\widetilde{t}_h+\widetilde{t}_k)/2} \E\Bigg[ \Bigg(\frac{1}{N}\sum_{1\leq u,v\leq N} \tau_N\left(E_{u,v}\ \partial_{i,k}\D_{i,h} P\left(G^+(X)^{N,\{T_n,t_{n+1}\}}\right)\# E_{v,u}\right) \nonumber\\
		&\quad\quad\quad\quad\quad\quad\quad\quad\quad\quad\quad\quad\quad- \tau_N\otimes\tau_N\left(\partial_{i,k}\D_{i,h} P\left(G^+(X)^{N,\{T_n,t_{n+1}\}}\right)\right) \Bigg) e^{-N\tr_N(V(X^N))} \Bigg] dt_{n+1}. \nonumber
	\end{align}
	
	\noindent Thus in order to get Equation \eqref{sicdsmsc}, it only remains to study the last two lines of the previous equation. In order to do so, given polynomials $A$ and $B$, let us now study the quantity 
	\begin{align*}
		\Lambda_N &\deq \frac{1}{N}\sum_{1\leq u,v\leq N} \tau_N\left(E_{u,v} A\left(G^+(X)^{N,\{T_n,t_{n+1}\}}\right) E_{v,u} B\left(G^+(X)^{N,\{T_n,t_{n+1}\}}\right)\right) \\
		&\quad- \tau_N\left(A\left(G^+(X)^{N,\{T_n,t_{n+1}\}}\right)\right) \tau_N\left( B\left(G^+(X)^{N,\{T_n,t_{n+1}\}}\right)\right).
	\end{align*}
	Note that by traciality,
	\begin{align*}
		\tau_N\left(A\left(G^+(X)^{N,\{T_n,t_{n+1}\}}\right)\right) &= \sum_{v=1}^N \tau_N\left( A\left(G^+(X)^{N,\{T_n,t_{n+1}\}}\right) E_{v,v} \right) \\
		&= \sum_{v=1}^N \frac{1}{N}\tau\left( e_v^*A\left(G^+(X)^{N,\{T_n,t_{n+1}\}}\right) e_v \right).
	\end{align*}
	In particular, the quantity $e_v^*A\left(G^+(X)^{N,\{T_n,t_{n+1}\}}\right) e_v$ is an element of $\A$ whereas $A\left(G^+(X)^{N,\{T_n,t_{n+1}\}}\right)$ is an element of $\A_N$. Hence $\Lambda_N$ is equal to
	\begin{align*}
		&\frac{1}{N}\sum_{1\leq u,v\leq N} \Bigg( \tau_N\left(E_{u,v} A\left(G^+(X)^{N,\{T_n,t_{n+1}\}}\right) E_{v,u} B\left(G^+(X)^{N,\{T_n,t_{n+1}\}}\right)\right) \\
		&\quad\quad\quad\quad\quad\quad - \frac{1}{N}\tau\left( e_v^*A\left(G^+(X)^{N,\{T_n,t_{n+1}\}}\right) e_v \right) \tau\left( e_u^*B\left(G^+(X)^{N,\{T_n,t_{n+1}\}}\right) e_u\right) \Bigg) \\
		&= \frac{1}{N}\sum_{1\leq u,v\leq N} \tau_N\left(E_{u,v} A\left(F_1^1(X)^{N,\{T_n,t_{n+1},0\}}\right) E_{v,u} B\left(F_1^2(X)^{N,\{T_n,t_{n+1},0\}}\right)\right) \\
		&\quad\quad\quad\quad\quad\quad- \tau_N\left( E_{u,v}A\left(F_{n+1}^1(X)^{N,\{T_n,t_{n+1},t_{n+1}\}}\right) E_{v,u} B\left(F_{n+1}^2(X)^{N,\{T_n,t_{n+1},t_{n+1}\}}\right) \right) \\
		&= \frac{1}{N}\sum_{1\leq u,v\leq N} \sum_{1\leq s\leq n+1} \tau_N\left(E_{u,v} A\left(F_s^1(X)^{N,\{T_n,t_{n+1},\widetilde{t}_{s-1}\}}\right) E_{v,u} B\left(F_s^2(X)^{N,\{T_n,t_{n+1},\widetilde{t}_{s-1}\}}\right)\right) \\
		&\quad\quad\quad\quad\quad\quad\quad\quad\quad\quad- \tau_N\left( E_{u,v}A\left(F_{n+1}^1(X)^{N,\{T_n,t_{n+1},\widetilde{t}_{s}\}}\right) E_{v,u} B\left(F_{n+1}^2(X)^{N,\{T_n,t_{n+1},\widetilde{t}_{s}\}}\right) \right) \\
		&= - \frac{1}{N}\sum_{1\leq u,v\leq N} \sum_{1\leq s\leq n+1} \int_{\widetilde{t}_{s-1}}^{\widetilde{t}_{s}} \frac{d}{dt}\tau_N\left(E_{u,v} A\left(F_s^1(X)^{N,\{T_n,t,t_{n+1}\}}\right) E_{v,u} B\left(F_s^2(X)^{N,\{T_n,t_{n+1},t\}}\right)\right) dt. 
	\end{align*}
	
	\noindent Besides, for $t\in [\widetilde{t}_{s-1},\widetilde{t}_{s}]$ we have that
	\begin{align}
		\label{cwooksmcowmcow}
		&\frac{d}{dt}\tau_N\left(E_{u,v} A\left(F_s^1(X)^{N,\{T_n,t_{n+1},t\}}\right) E_{v,u} B\left(F_s^2(X)^{N,\{T_n,t_{n+1},t\}}\right)\right) \\
		&= \frac{1}{2}\sum_{1\leq j\leq d} \sum_{0\leq x\leq s-1} e^{\widetilde{t}_x/2-t} \sum_{A=A_1 X_{j,I} A_2,\ I\in J_H^x} \tau_N\Bigg( \left(\frac{x_j^{I_s+c_{J_H}}}{(e^{-\widetilde{t}_{s-1}}-e^{-t})^{1/2}} - \frac{x_j^{I_s}}{(e^{-t}-e^{-\widetilde{t}_{s}})^{1/2}} \right) \nonumber\\ 
		&\quad\quad\quad\quad\quad\quad\quad A_2\left(F_s^1(X)^{N,\{T_n,t_{n+1},t\}}\right) E_{v,u} B\left(F_s^2(X)^{N,\{T_n,t_{n+1},t\}}\right) E_{u,v} A_1\left(F_s^1(X)^{N,\{T_n,t_{n+1},t\}}\right)\Bigg) \nonumber\\
		&\ + \frac{1}{2}\sum_{1\leq j\leq d} \sum_{0\leq x\leq s-1} e^{\widetilde{t}_x/2-t} \sum_{B=B_1 X_{j,I} B_2,\ I\in J_H^x} \tau_N\Bigg( \left(\frac{x_j^{I_s+2c_{J_H}}}{(e^{-\widetilde{t}_{s-1}}-e^{-t})^{1/2}} - \frac{x_j^{I_s}}{(e^{-t}-e^{-\widetilde{t}_{s}})^{1/2}} \right) \nonumber\\ 
		&\quad\quad\quad\quad\quad\quad\quad  B_2\left(F_s^2(X)^{N,\{T_n,t_{n+1},t\}}\right) E_{u,v} A\left(F_s^1(X)^{N,\{T_n,t_{n+1},t\}}\right) E_{v,u} B_1\left(F_s^2(X)^{N,\{T_n,t_{n+1},t\}}\right)\Bigg). \nonumber
	\end{align}
	Then thanks again to Proposition \ref{3SDE},
	\begin{align*}
		&\tau_N\Bigg( \frac{x_j^{I_s}}{(e^{-t}-e^{-\widetilde{t}_{s}})^{1/2}} A_2\left(F_s^1(X)^{N,\{T_n,t_{n+1},t\}}\right) E_{v,u} B\left(F_s^2(X)^{N,\{T_n,t_{n+1},t\}}\right) E_{u,v} A_1\left(F_s^1(X)^{N,\{T_n,t_{n+1},t\}}\right)\Bigg) \\
		&= \sum_{0\leq y\leq s-1} e^{\widetilde{t}_y/2} \sum_{\substack{A_2=A_3 X_{j,J} A_4,\\ J\in J_H^y \text{ such that } I_s=J_s}} \tau_N\Big(A_3\left(F_s^1(X)^{N,\{T_n,t_{n+1},t\}}\right)\Big) \tau_N\Big( A_4\left(F_s^1(X)^{N,\{T_n,t_{n+1},t\}}\right) \\
		&\quad\quad\quad\quad\quad\quad\quad\quad\quad\quad\quad\quad\quad\quad\quad  E_{v,u} B\left(F_s^2(X)^{N,\{T_n,t_{n+1},t\}}\right) E_{u,v} A_1\left(F_s^1(X)^{N,\{T_n,t_{n+1},t\}}\right)\Big) \\
		&\ + \sum_{0\leq y\leq s-1} e^{\widetilde{t}_y/2} \sum_{\substack{A_1=A_3 X_{j,J} A_4,\\ J\in J_H^y \text{ such that } I_s=J_s}} \tau_N\Big(A_2\left(F_s^1(X)^{N,\{T_n,t_{n+1},t\}}\right) E_{v,u} B\left(F_s^2(X)^{N,\{T_n,t_{n+1},t\}}\right) \\
		&\quad\quad\quad\quad\quad\quad\quad\quad\quad\quad\quad\quad\quad\quad\quad  E_{u,v} A_3\left(F_s^1(X)^{N,\{T_n,t_{n+1},t\}}\right)\Big) \tau_N\Big(A_4\left(F_s^1(X)^{N,\{T_n,t_{n+1},t\}}\right)\Big) \\
		&\ + \sum_{0\leq y\leq s-1} e^{\widetilde{t}_y/2} \sum_{\substack{B=B_1 X_{j,J} B_2,\\ J\in J_H^y \text{ such that } I_s=J_s}} \tau_N\Big( A_2\left(F_s^1(X)^{N,\{T_n,t_{n+1},t\}}\right) E_{v,u} B_1\left(F_s^2(X)^{N,\{T_n,t_{n+1},t\}}\right)\Big) \\
		&\quad\quad\quad\quad\quad\quad\quad\quad\quad\quad\quad\quad\quad\quad\quad \tau_N\Big(B_2\left(F_s^2(X)^{N,\{T_n,t_{n+1},t\}}\right) E_{u,v} A_1\left(F_s^1(X)^{N,\{T_n,t_{n+1},t\}}\right)\Big),
	\end{align*}
	and similarly
	\begin{align*}
		&\tau_N\Bigg( \frac{x_j^{I_s+c_{J_H}}}{(e^{-\widetilde{t}_{s-1}}-e^{-t})^{1/2}} A_2\left(F_s^1(X)^{N,\{T_n,t_{n+1},t\}}\right) E_{v,u} B\left(F_s^2(X)^{N,\{T_n,t_{n+1},t\}}\right) E_{u,v} A_1\left(F_s^1(X)^{N,\{T_n,t_{n+1},t\}}\right)\Bigg) \\
		&= \sum_{0\leq y\leq s-1} e^{\widetilde{t}_y/2} \sum_{\substack{A_2=A_3 X_{j,J} A_4,\\ J\in J_H^y \text{ such that } I_s=J_s}} \tau_N\Big(A_3\left(F_s^1(X)^{N,\{T_n,t_{n+1},t\}}\right)\Big) \tau_N\Big( A_4\left(F_s^1(X)^{N,\{T_n,t_{n+1},t\}}\right) \\
		&\quad\quad\quad\quad\quad\quad\quad\quad\quad\quad\quad\quad\quad\quad\quad  E_{v,u} B\left(F_s^2(X)^{N,\{T_n,t_{n+1},t\}}\right) E_{u,v} A_1\left(F_s^1(X)^{N,\{T_n,t_{n+1},t\}}\right)\Big) \\
		&\ + \sum_{0\leq y\leq s-1} e^{\widetilde{t}_y/2} \sum_{\substack{A_1=A_3 X_{j,J} A_4,\\ J\in J_H^y \text{ such that } I_s=J_s}} \tau_N\Big(A_2\left(F_s^1(X)^{N,\{T_n,t_{n+1},t\}}\right) E_{v,u} B\left(F_s^2(X)^{N,\{T_n,t_{n+1},t\}}\right) \\
		&\quad\quad\quad\quad\quad\quad\quad\quad\quad\quad\quad\quad\quad\quad\quad  E_{u,v} A_3\left(F_s^1(X)^{N,\{T_n,t_{n+1},t\}}\right)\Big) \tau_N\Big(A_4\left(F_s^1(X)^{N,\{T_n,t_{n+1},t\}}\right)\Big).
	\end{align*}
	
	\noindent Thus by doing the same computations for the last two lines of \eqref{cwooksmcowmcow} for $t\in [\widetilde{t}_{s-1},\widetilde{t}_{s}]$ we have that
	\begin{align*}
		&\frac{d}{dt}\tau_N\left(E_{u,v} A\left(F_s^1(X)^{N,\{T_n,t_{n+1},t\}}\right) E_{v,u} B\left(F_s^2(X)^{N,\{T_n,t_{n+1},t\}}\right)\right) \\
		&= -\sum_{1\leq j\leq d} \sum_{0\leq x,y\leq s-1} e^{(\widetilde{t}_y+\widetilde{t}_x)/2-t} \sum_{\substack{A=A_1 X_{j,I} A_2, B=B_1 X_{j,J} B_2,\\ I\in J_H^x, J\in J_H^y \text{ such that } I_s=J_s}} \\
		&\quad\quad\quad\quad\quad\quad\quad\quad\quad\quad\quad\quad\quad\quad\quad\quad \tau_N\Big( A_2\left(F_s^1(X)^{N,\{T_n,t_{n+1},t\}}\right) E_{v,u} B_1\left(F_s^2(X)^{N,\{T_n,t_{n+1},t\}}\right)\Big) \\
		&\quad\quad\quad\quad\quad\quad\quad\quad\quad\quad\quad\quad\quad\quad\quad\quad \tau_N\Big(B_2\left(F_s^2(X)^{N,\{T_n,t_{n+1},t\}}\right) E_{u,v} A_1\left(F_s^1(X)^{N,\{T_n,t_{n+1},t\}}\right)\Big) \\
		&= -\sum_{1\leq j\leq d} \sum_{0\leq x,y\leq s-1} e^{(\widetilde{t}_y+\widetilde{t}_x)/2-t} \sum_{\substack{A=A_1 X_{j,I} A_2, B=B_1 X_{j,J} B_2,\\ I\in J_H^x, J\in J_H^y \text{ such that } I_s=J_s}} \\
		&\quad\quad\quad\quad\quad\quad\quad\quad\quad\quad\quad\quad\quad\quad\quad\quad \frac{1}{N}\tau\Big( e_u^* B_1\left(F_s^2(X)^{N,\{T_n,t_{n+1},t\}}\right) A_2\left(F_s^1(X)^{N,\{T_n,t_{n+1},t\}}\right) e_v \Big) \\
		&\quad\quad\quad\quad\quad\quad\quad\quad\quad\quad\quad\quad\quad\quad\quad\quad \frac{1}{N}\tau\Big( e_v^* A_1\left(F_s^1(X)^{N,\{T_n,t_{n+1},t\}}\right) B_2\left(F_s^2(X)^{N,\{T_n,t_{n+1},t\}}\right) e_u \Big) \\
		&= -\frac{1}{N}\sum_{1\leq j\leq d} \sum_{0\leq x,y\leq s-1} e^{(\widetilde{t}_y+\widetilde{t}_x)/2-t} \sum_{\substack{A=A_1 X_{j,I} A_2, B=B_1 X_{j,J} B_2,\\ I\in J_H^x, J\in J_H^y \text{ such that } I_s=J_s}} \\
		&\quad\quad\quad\quad\quad\quad\quad\quad\quad\quad\quad\quad\quad\quad\quad\quad \tau_N\Big( B_1\left(F_s^2(X)^{N,\{T_n,t_{n+1},t\}}\right) A_2\left(F_s^1(X)^{N,\{T_n,t_{n+1},t\}}\right) E_{v,v} \Big) \\
		&\quad\quad\quad\quad\quad\quad\quad\quad\quad\quad\quad\quad\quad\quad\quad\quad\quad\quad A_1\left(\widetilde{F}_s^1(X)^{N,\{T_n,t_{n+1},t\}}\right) B_2\left(\widetilde{F}_s^2(X)^{N,\{T_n,t_{n+1},t\}}\right) E_{u,u} \Big).
	\end{align*}
	
	\noindent Hence we have that,
	\begin{align*}
		\Lambda_N &= \frac{1}{N^2} \sum_{1\leq s\leq n+1} \int_{\widetilde{t}_{s-1}}^{\widetilde{t}_{s}} \sum_{1\leq j\leq d} \sum_{0\leq x,y\leq s-1} e^{(\widetilde{t}_y+\widetilde{t}_x)/2-t} \sum_{\substack{I\in J_H^x, J\in J_H^y \\ \text{ such that } I_s=J_s}} \sum_{A=A_1 X_{j,I} A_2, B=B_1 X_{j,J} B_2} \\
		&\quad\quad \tau_N\Big( A_2\left(F_s^1(X)^{N,\{T_n,t_{n+1},t\}}\right)  \Big) A_1\left(\widetilde{F}_s^1(X)^{N,\{T_n,t_{n+1},t\}}\right) \\
		&\quad\quad\quad\quad\quad\quad\quad\quad\quad\quad\quad\quad\quad\quad\quad\quad\quad\quad\quad\quad B_2\left(\widetilde{F}_s^2(X)^{N,\{T_n,t_{n+1},t\}}\right) B_1\left(F_s^2(X)^{N,\{T_n,t_{n+1},t\}}\right)\Big) dt \\
		&= \frac{1}{N^2} \sum_{1\leq s\leq n+1} \int_{\widetilde{t}_{s-1}}^{\widetilde{t}_{s}} \sum_{1\leq j\leq d} \sum_{0\leq x,y\leq s-1} e^{(\widetilde{t}_y+\widetilde{t}_x)/2-t} \\
		&\quad\quad\quad\quad\quad\quad\quad\quad\quad\quad\quad\quad\quad\quad\sum_{\substack{I\in J_H^x, J\in J_H^y \\ \text{ such that } I_s=J_s}} \tau_N\Big( \Theta^{F_s^1,\widetilde{F}_s^1,\widetilde{F}_s^2,F_s^2}\left[\partial_{j,I}A \otimes \partial_{j,J}B\right]\left(X^{N,\{T_n,t_{n+1},t\}}\right)\Big) dt.
	\end{align*}
	Thus, we have 
	\begin{align*}
		&\int_{\widetilde{t}_{n}}^{\infty} \frac{e^{-t_{n+1}}}{2} \sum_{1\leq i\leq d} \sum_{0\leq h,k \leq n} e^{(\widetilde{t}_h+\widetilde{t}_k)/2} \E\Bigg[ \Bigg(\frac{1}{N}\sum_{1\leq u,v\leq N} \tau_N\left(E_{u,v}\ \partial_{i,k}\D_{i,h} P\left(G^+(X)^{N,\{T_n,t_{n+1}\}}\right)\# E_{v,u}\right) \nonumber\\
		&\quad\quad\quad\quad\quad\quad\quad\quad\quad\quad\quad\quad\quad- \tau_N\otimes\tau_N\left(\partial_{i,k}\D_{i,h} P\left(G^+(X)^{N,\{T_n,t_{n+1}\}}\right)\right) \Bigg) e^{-N\tr_N(V(X^N))} \Bigg] dt_{n+1} \\
		&= \frac{1}{N^2}\int_{\widetilde{t}_{n}}^{\infty} \frac{e^{-t_{n+1}}}{2} \sum_{1\leq i,j\leq d} \sum_{0\leq h,k \leq n} e^{(\widetilde{t}_h+\widetilde{t}_k)/2} \sum_{1\leq s\leq n+1} \int_{\widetilde{t}_{s-1}}^{\widetilde{t}_{s}} \sum_{0\leq x,y\leq s-1} e^{(\widetilde{t}_y+\widetilde{t}_x)/2-t} \sum_{\substack{I\in J_H^x, J\in J_H^y \\ \text{ such that } I_s=J_s}} \\
		&\quad\quad\E\Big[ \tau_N\Big( \Theta^{F_s^1,\widetilde{F}_s^1,\widetilde{F}_s^2,F_s^2}\left[[\partial_{j,I} \otimes \partial_{j,J}]\circ \partial_{i,k}\circ\D_{i,h}P \right] \left(X^{N,\{T_n,t_{n+1},t\}}\right)\Big)  e^{-N\tr_N(V(X^N))} \Big] dt\ dt_{n+1} \\
		&= \frac{1}{N^2}\int_{\widetilde{t}_{n}}^{\infty} \sum_{1\leq s\leq n+1} \int_{\widetilde{t}_{s-1}}^{\widetilde{t}_{s}} \E\Big[ \tau_N\Big(L_s^{H,T_{n+2}}\left(X^{N,\{T_n,t_{n+1},t\}}\right)\Big)  e^{-N\tr_N(V(X^N))} \Big] dt\ dt_{n+1} \\
		&= \frac{1}{N^2}\int_{\widetilde{t}_{n}}^{\infty} \int_{0}^{t_{n+1}} \E\Big[ \tau_N\Big(L^{H,T_{n+2}}\left(X^{N,\{T_n,t_{n+1},t\}}\right)\Big)  e^{-N\tr_N(V(X^N))} \Big] dt\ dt_{n+1}.
	\end{align*}
	Hence the conclusion by renaming $t$ into $t_{n+2}$ in the last line and by plugging this result back into \eqref{lskdmcosmcosancdoakc}.
\end{proof}

In the case where one has a cut-off, we have a very similar theorem with the difference that since one has to use Lemma \ref{ippcut} instead of Gaussian integration by parts in Equation \eqref{lkmsdvmdss}, there is an error term due to the cut-off.

\begin{theorem}
	\label{mainlemma2}
	Let the following objects be given:
	\begin{itemize}
		\item $P\in\A_{d}^H$, $V\in\A_{d}$,
		\item $X^N$ a family of $d$ independent GUE matrices,
		\item $x,x^1,x^2,\dots\in\A_N$ an infinite sequence of free semicircular systems of $d$ variables, freely independent of each other.
	\end{itemize}
	Then with the notations of Theorem \ref{mainlemma},
	\begin{align}
		\label{sicdsmsc2}
		&\E\left[ \tau_N\left(P\left(X^{N,T_n}\right)\right)\ e^{-N\tr_N(V(X^N))} \1_{\forall i, ||X_i^N||\leq K} \right] \nonumber \\
		= &\ \tau\left(P(x^{T_n})\right)\ \E\left[ e^{-N\tr_N(V(X^N))} \1_{\forall i, ||X_i^N||\leq K} \right] \\
		&- \int_{\widetilde{t}_n}^{\infty} \E\left[ \tau_N\left(\nabla_V^{H,T_{n+1}}(P)\left(X^{N,T_{n+1}}\right)\right)\ e^{-N\tr_N(V(X^N))} \1_{\forall i, ||X_i^N||\leq K} \right] dt_{n+1} \nonumber \\
		&+ \frac{1}{N^2} \int_{\widetilde{t}_{n}}^{\infty}\int_0^{t_{n+1}} \E\left[ \tau_N\left(L^{H,T_{n+2}}(P)\left(X^{N,T_{n+2}}\right)\right)\ e^{-N\tr_N(V(X^N))} \1_{\forall i, ||X_i^N||\leq K} \right] dt_{n+2} dt_{n+1} \nonumber \\ \nonumber
		&+ \frac{\mathcal{E}_N}{N^2}\E\left[ e^{- N\tr_N\left(V(X^N)\right)} \1_{\forall i, ||X_i^N||\leq K}\right].
	\end{align}
	where there exist constants $D$ and $C_{M,K}$ for any $M<K$ such that
	$$ \left|\mathcal{E}_N\right| \leq C_{M,K} \norm{P}_{D}\ \mu_{V,K}^N\left(\max_{i\leq i\leq d} \norm{X_i} \geq M\right). $$
\end{theorem}

\section{The proof of Theorem \ref{maintherorem}}

\subsection{A careful estimate of the error term}

Since we define the coefficients $\alpha_n^V(\lambda,P)$ as a power series in $\lambda$ we need to prove that it actually converges. This turned out to be one of the most difficult parts of the proof. Indeed, it is possible to upper bound the $k$-th coefficient of this power series by an integral in $k$ variables (see Lemma \ref{sdlmcldsfmvs}); but one then has to show that this quantity does not grow too fast. The proof is divided into two parts: in the first lemma we bound this integral by another one which is easier to study, and in the second part we find an analytic function whose coefficients of the power series at zero control the quantities we want to study, hence proving the desired result. Although this strategy does let us conclude, a more direct proof that erases the need for Lemma \refeq{odvfkmssm} and \ref{lkmdsocms}, and directly prove Lemma \ref{sdlmcldsfmvs}, would be welcome. 

\begin{lemma}
	\label{odvfkmssm}
	We fix $k\geq 1,p\geq 2$, then with
	$$ E_{k} = \Big\{ (n_1,\dots,n_k)\in\N^k\ \Big|\ \forall i, 1\leq n_i \leq i\Big\}, $$
	$$ E_{k,p} = \Big\{ (n_1,\dots,n_k)\in\N^k\ \Big|\ \forall i, 1\leq n_i \leq i,\text{ and } \forall j\geq 1, \#\{i\ |\ n_i=j\} \leq p\Big\}, $$
	one has,
	\begin{align}
		\int_{t_n\geq \dots \geq t_1\geq 0} \sum_{(n_1,\dots,n_k)\in E_{k}} e^{-\sum\limits_{1\leq i\leq k} t_i-t_{n_i-1}}&\ dt_1\dots dt_k \\
		&\leq 2^k \int_{t_n\geq \dots \geq t_1\geq 0} \sum_{(n_1,\dots,n_k)\in E_{k,p}} e^{-\sum\limits_{1\leq i\leq k} t_i-t_{n_i-1}}\ dt_1\dots dt_k. \nonumber
	\end{align}
\end{lemma}

\begin{proof}
	Thanks to the change of variables $(t_1,\dots,t_n)\mapsto (t_1,t_1+t_2,\dots,t_1+\dots+t_n)$, one has that 
	\begin{align}
		\label{socdosmccca}
		&\int_{t_n\geq \dots \geq t_1\geq 0} \sum_{(n_1,\dots,n_k)\in E_{k}} e^{-\sum\limits_{1\leq i\leq k} t_i-t_{n_i-1}}\ dt_1\dots dt_k \nonumber \\
		&= \int_{\R_+^k} \sum_{(n_1,\dots,n_k)\in E_{k}} e^{-\sum\limits_{1\leq i\leq k} t_{n_i}+\dots+t_i}\ dt_1\dots dt_k  \\
		&=\sum_{(n_1,\dots,n_k) \in E_k} \prod_{1\leq j\leq k} \frac{1}{\#\{ i\ |\ n_i \leq j \leq i \}}. \nonumber
	\end{align}
	We write $(m_1,\dots,m_k)\geq (n_1,\dots,n_k)$ if for any $i$, $m_i\geq n_i$, then
	\begin{equation}
		\label{smcsomc}
		\prod_{1\leq j\leq k} \frac{1}{\#\{ i\ |\ m_i \leq j \leq i \}} \geq \prod_{1\leq j\leq k} \frac{1}{\#\{ i\ |\ n_i \leq j \leq i \}}.
	\end{equation}
	
	\noindent Next we define by induction the function $f:\cup_{k\geq 1} E_k \to \cup_{k\geq 1} E_{k,p}$ which maps $E_k$ to $E_{k,p}$,
	\begin{itemize}
		\item for $k=1$, $f(1)= (1)$,
		\item for $k>1$, we define $f(n_1,\dots,n_k) = (f(n_1,\dots,n_{k-1}),m_k)$ where we pick $m_k$ with the following process. We set $(m_1,\dots,m_{k-1})= f(n_1,\dots,n_{k-1})$, then if $\#\{i<k\ |\ m_i=n_k\} < p$, we fix $m_k=n_k$. Else we define the following sequence for $l\in [1,n_k]$, 
		\begin{equation}
			\label{skdjncsknd}
			u_{n_k}=1,\quad \forall l<n_k,\ u_l = \left\{ \begin{array}{ll}
				u_{l+1} - 1 & \mbox{if } \#\{i<k\ |\ m_i=l\} < p \\
				u_{l+1} + 1 & \mbox{else}
			\end{array}
			\right. ,
		\end{equation}
		as well as
		$$r = \max_{l\in [1,n_k]} u_l.$$
		Then with,
		$$ v_{0}=n_k,\quad \forall l>0,\ v_l = \min\{ n>v_{l-1}\ |\ \#\{j<k\ |\ m_j=n\} < p \},$$
		we set $ m_k = v_{r}$.
	\end{itemize}

	\begin{rem}
		The definition of the function $f$ might seem technical but it describes a very natural process. First one can view the elements of $E_k$ in the following way. Given a matrix of size $k$ we want to pick a coefficient in each row under the condition that on the row $i$ the element we pick is in the first $i$ columns. Then by choosing the coefficients $(i,n_i)$ for $i$ from $1$ to $k$ we define an element of $E_k$ with the list $(n_1,\dots,n_k)$. 
		
		Thus, in the definition of $f$, in order to pick $m_k$ given $(m_1,\dots,m_{k-1})$, we proceed in the following way. First we check if we have chosen less than $p$ indices in the column $n_k$, i.e if $\#\{i<k\ |\ m_i=n_k\} < p$; if this is the case then we simply set $m_k=n_k$. Otherwise, we say that the column $n_k$ is full. In which case, one  wants to pick a column $m_k$ after the $n_k$-th one which is not full. But we also want to do so in a way that is almost injective, i.e. such that given $(m_1,\dots,m_{k-1})$, for any $\widetilde{n}_k\in[1,k]$, the only way for us to set $\widetilde{m}_k=m_k$ is if $\widetilde{n}_k=m_k \text{ or } n_k$. To do so, we first look at how many columns before the $n_k$-th one are full, then the integer $r$ that we define above tells us how many columns ahead we need to go in order to pick $m_k$. Finally the sequence $(v_l)_{0\leq l\leq r}$ gives us the $r$-th non-full column after the $n_k$-th one and we finally take $m_k=v_r$.
		
		Note that those properties of the function $f$ described above are not trivial, we prove them in the rest of the proof. 
	\end{rem}

	First and foremost let us explain why this process is well-defined, i.e. why if $\#\{i<k\ |\ m_i=n_k\} = p$ then the length of the sequence $v$ is larger than $r$ and thus why one can pick $v_r$ smaller or equal to $k$. If the maximum $r$ of the sequence $(u_l)_{l\in [1,n_k]}$ is reached at $L$, then one has $L=n_k-(r-1)-2s$ where $s$ is the number of indices $l\in [L,n_k]$ such that $\#\{i<k\ |\ m_i=l\} < p$. Thus there are at least $p(r+s)$ indices $i<k$ such that $m_i\in[L,n_k]$, however since $m_i\leq i$, there can be at most $k-L$ indices $i<k$ such that $m_i\geq L$, consequently with $t$ the number of indices $i<k$ such that $m_i>n_k$,
	$$ k-n_k+r+2s -1 = k-L \geq t+ p(r+s).$$
	And hence since $p\geq 2$, one has that $k-n_k\geq 1+t+(p-1)r$. Thus if $q$ is the number of indices $l>n_k$ such that $\#\{j<k\ |\ m_j=l\} = p$, one has that $t\geq qp$, and thus $k-n_k\geq 1+qp+(p-1)r$. Consequently,
	$$k-n_k \geq 1+r+q.$$
	In particular one can find $r$ indices larger than $n_k$ such that their respective columns are not full, hence $v_r$ is well-defined.
	
	Next we define the map $h_{n_1,\dots,n_{k-1}}: [1,k]\to [1,k]$ that to an integer $n$ associates the last component of the vector $f(n_1,\dots,n_{k-1},n)$ (i.e. the integer $m_k$ that the process above yields), then while $h_{n_1,\dots,n_{k-1}}$ is not injective, for any $n\in [1,k]$, $h_{n_1,\dots,n_{k-1}}^{-1}(\{n\})$ has at most two elements. Indeed if $n$ is such that $\#\{j<k\ |\ m_j=n\} =p$, then $h_{n_1,\dots,n_{k-1}}^{-1}(\{n\})$ is empty, else if $n$ is such that $\#\{j<k\ |\ m_j=n\} <p$, then $n\in h_{n_1,\dots,n_{k-1}}^{-1}(\{n\})$. Besides, if $a<b<n$ are such that $h_{n_1,\dots,n_{k-1}}(a)=h_{n_1,\dots,n_{k-1}}(b)=n$, then let us denote $r_a$ and $r_b$ the quantities which appears in the previous process, then with $q_1$ (respectively $q_2$) the number of indices $l\in [a+1,b]$ (respectively $[b+1,n]$) such that $\#\{i<k\ |\ m_i=l\} =p$, then 
	$$ n = a+r_a+q_1+q_2 = b+r_b+q_2.$$
	Thus $r_b = r_a + a-b+q_1$. However, if $(u_l^a)_{l\in [1,a]}$ (respectively $(u_l^b)_{l\in [1,b]}$) is the sequence in \eqref{skdjncsknd} associated to $a$ (respectively $b$), then for all $l\in [1,a]$, $u_l^a = u_l^b - u_{a+1}^b = u_l^b - (b-a-2q_1)$, hence by taking the maximum, one has that $r_b \geq r_a +a-b + 2 q_1$. Hence $q_1\leq 0$, yielding a contradiction since $h_{n_1,\dots,n_{k-1}}(b)=n\neq b$, thus $\#\{i<k\ |\ m_i=b\} =p$ and consequently $q_1\geq 1$. Thus we do indeed have that for any $n\in [1,k]$, $h_{n_1,\dots,n_{k-1}}^{-1}(\{n\})$ has at most two elements.
	
	Consequently, we have defined a function $f: E_k\to E_{k,p}$ such that for any $(n_1,\dots,n_k)\in E_k$, $f(n_1,\dots,n_k)\geq (n_1,\dots,n_k)$. Let us fix $(n_1,\dots,n_k)\in E_{k,p}$, then we claim that $f^{-1}(\{(n_1,\dots,n_k)\})$ has at most $2^k$ elements. Indeed by induction, if we assume that there are at most $2^{k-1}$ elements in $f^{-1}(\{(n_1,\dots,n_{k-1})\})$, then since $h_{n_1,\dots,n_{k-1}}^{-1}(\{n_k\})$ has cardinality at most two, then $f^{-1}(\{(n_1,\dots,n_{k})\})$ has cardinal at most $2^k$. Consequently, thanks to Equations \eqref{socdosmccca} and \eqref{smcsomc},
	\begin{align*}
		&\int_{t_n\geq \dots \geq t_1\geq 0} \sum_{(n_1,\dots,n_k)\in E_k} e^{-\sum\limits_{1\leq i\leq k} t_i-t_{n_i-1}}\ dt_1\dots dt_k \\
		&= \sum_{(n_1,\dots,n_k) \in E_k} \prod_{1\leq j\leq k} \frac{1}{\#\{ i\ |\ n_i \leq j \leq i \}} \\
		\leq& \sum_{(n_1,\dots,n_k) \in E_{k,p}} \#f^{-1}(\{(n_1,\dots,n_k)\}) \prod_{1\leq j\leq k} \frac{1}{\#\{ i\ |\ n_i \leq j \leq i \}} \\
		\leq&\ 2^k \sum_{(n_1,\dots,n_k) \in E_{k,p}} \prod_{1\leq j\leq k} \frac{1}{\#\{ i\ |\ n_i \leq j \leq i \}} \\
		=&\ 2^k \int_{t_n\geq \dots \geq t_1\geq 0} \sum_{(n_1,\dots,n_k)\in E_{k,p}} e^{-\sum\limits_{1\leq i\leq k} t_i-t_{n_i-1}}\ dt_1\dots dt_k.
	\end{align*}
\end{proof}

We now need the following lemma. Note that in order to prove it we use results from \cite{guionnet-segala} which studies similar random matrix models. By using their results for a very specific model, one can show that the integral $I_{k,p}$ defined below is bounded by the coefficients of a power series with a non-zero radius of convergence, hence proving Lemma \ref{lkmdsocms}.

\begin{lemma}
	\label{lkmdsocms}
	Fix $p\geq 2,k\geq 1$. Then there exists a constant $K_p$ such that with
	$$ E_{k,p} = \Big\{ (n_1,\dots,n_k)\in\N^k\ \Big|\ \forall i, 1\leq n_i \leq i,\text{ and } \forall j\geq 1, \#\{i\ |\ n_i=j\} \leq p\Big\}, $$
	one has,
	\begin{equation}
		I_{k,p} = \int_{t_n\geq \dots \geq t_1\geq 0} \sum_{(n_1,\dots,n_k)\in E_{k,p}} e^{-\sum\limits_{1\leq i\leq k} t_i-t_{n_i-1}}\ dt_1\dots dt_k \leq  (K_p)^k.
	\end{equation}
\end{lemma}

\begin{proof}
	Since for any $p$, $E_{k,p}\subset E_{k,p+1}$, one has that $I_{k,p} \leq I_{k,p+1}$, hence we can always assume that $p$ is odd, then we set $V(X) \deq \alpha_p X^2 + \sum_{l=1}^{p+1}X^l$ where we picked $\alpha_p\geq 0$ such that $V$ is strictly convex which one can check by computing the second derivative, split into cases $|X| \leq 1$ and $|X| \geq 1$, and pair the consecutive even and odd powered terms. Then thanks to Klein's lemma (see Lemma 4.4.12 of \cite{alice}), the function $X\in\M_N(\C)_{sa} \mapsto \tr_N(V(X))$ is also convex. In particular, for any $k$ and $\lambda\geq0$, 
	$$ \E\left[\tr_N(X^{2k}) e^{-N\lambda\tr_N(V(X))}\right] <\infty.$$
	By applying Theorem \ref{mainlemma} $n+1$ times, one has that
	\begin{align}
		\label{lskmclkmc}
		&\E\left[\ts_N\left(\sum_{l=0}^{p}X^l\right) e^{-N\lambda\tr_N(V(X))}\right] \nonumber\\
		&= \E\left[e^{-N\lambda\tr_N(V(X))}\right]\sum_{0\leq k\leq n} \int_{t_k\geq \dots\geq t_1\geq 0} (-\lambda)^k \tau\left(\nabla_V^{\{G,\dots,G\},T_k}\circ\dots\circ\nabla_V^{\{G\},T_1}\left(\sum_{l=0}^{p}X^l\right)(x^{T_k}) \right)dt_1\dots dt_k \nonumber\\
		&+ (-\lambda)^{n+1}\int_{t_{n+1}\geq \dots\geq t_1\geq 0} \E\Big[ \tau_N\left(\nabla_V^{\{G,\dots,G\},T_{n+1}}\circ\dots\circ\nabla_V^{\{G\},T_1}\left(\sum_{l=0}^{p}X^l\right)\left(X^{N,T_{n+1}}\right)\right) \\
		&\quad\quad\quad\quad\quad\quad\quad\quad\quad\quad\quad\quad\quad\quad\quad\quad\quad\quad\quad\quad\quad\quad\quad\quad\quad\quad\quad\quad\quad\quad\quad e^{-N\tr_N(V(X^N))} \Big] dt_1\dots dt_{n+1} \nonumber \\
		&+ \frac{1}{N^2} \sum_{0\leq k\leq n} (-\lambda)^k \int_{\substack{t_{k+1}\geq \dots\geq t_1\geq 0,\\ t_{k+1}\geq t_{k+2}\geq 0}} \E\Big[ \tau_N\Big(L^{H,T_{k+2}}\circ\nabla_V^{\{G,\dots,G\},T_k}\circ\cdots \nonumber\\
		&\quad\quad\quad\quad\quad\quad\quad\quad\quad\quad\quad\quad\quad\cdots\circ\nabla_V^{\{G\},T_1}\left(\sum_{l=0}^{p}X^l\right)\left(X^{N,\{T_{k+2}\}}\right)\Big)\ e^{-N\tr_N(V(X^N))} \Big] dt_1\dots dt_{k+2}. \nonumber
	\end{align}
	Besides, one can apply Theorem 3.4 of \cite{guionnet-segala} in combination with Theorem 2.3 of the same paper which states that for any polynomial $Q_1,\dots,Q_l\in\A_{d}$,
	$$ \lim_{N\to\infty}\frac{\E\left[\ts_N(Q_1(X^N))\dots \ts_N(Q_l(X^N)) e^{-N\lambda\tr_N(V(X))}\right]}{\E\left[ e^{-N\lambda\tr_N(V(X))}\right]} = \tau_{\lambda}(Q_1)\dots \tau_{\lambda}(Q_l),$$
	where $\tau_{\lambda}(Q)$ is an analytic function in $\lambda$. Thus since the last line in Equation \eqref{lskmclkmc} is normalized by $N^2$, we have that
	\begin{align*}
		&\tau_{\lambda}\left(\sum_{l=0}^{p}X^l\right) = \sum_{0\leq k\leq n} (-\lambda)^k \int_{t_k\geq \dots\geq t_1\geq 0}  \tau\left(\nabla_V^{\{G,\dots,G\},T_n}\circ\dots\circ\nabla_V^{\{G\},T_1}\left(\sum_{l=0}^{p}X^l\right)(x^{T_k}) \right) \\
		&\quad\quad\quad\quad\quad\quad\quad\quad\quad\quad\quad\quad\quad\quad\quad\quad\quad\quad\quad\quad\quad\quad\quad\quad\quad\quad\quad\quad\quad\quad\quad\quad dt_1\dots dt_k + (-\lambda)^{n+1}g_{p,n}(\lambda), \nonumber
	\end{align*}
	for some analytic function $g_{p,n}$. Since $\lambda\mapsto \tau_{\lambda}(X^{2p}+X^{2p-1})$ is analytic and the equation above is true for any $n$, this implies that 
	\begin{equation}
		\label{kjdnvflskmvls}
		\tau_{\lambda}\left(\sum_{l=0}^{p}X^l\right) = \sum_{k\geq 0} (-\lambda)^k \int_{t_k\geq \dots\geq t_1\geq 0}  \tau\left(\nabla_V^{\{G,\dots,G\},T_k}\circ\dots\circ\nabla_V^{\{G\},T_1}\left(\sum_{l=0}^{p}X^l\right)(x^{T_k}) \right) dt_1\dots dt_k.
	\end{equation}
	We now show by induction that 
	$$ 2^k \tau\left(\nabla_V^{\{G,\dots,G\},T_k}\circ\dots\circ\nabla_V^{\{G\},T_1}\left(\sum_{l=0}^{p}X^l\right)(x^{T_k}) \right) \geq \sum_{(n_1,\dots,n_k)\in E_{k,p}} e^{-\sum\limits_{1\leq i\leq k} (t_i-t_{n_i-1})/2}. $$
	More precisely, let us assume that for a given $k$:
	\begin{itemize}
		\item One can write
		\begin{align*}
			& 2^k \nabla_V^{\{G,\dots,G\},T_k}\circ\dots\circ\nabla_V^{\{G\},T_1}\left(\sum_{l=0}^{p}X^l\right) \\
			&= \sum_{(n_1,\dots,n_k)\in E_k} e^{-\sum\limits_{1\leq i\leq k} (t_i-t_{n_i-1})/2}\ P_{(n_1,\dots,n_k)}\left(X_{\emptyset},X_{\{k\}},X_{\{k,k-1\}},\dots,X_{\{k,\dots,1\}}\right),
		\end{align*}
		where $P_{(n_1,\dots,n_k)}$ is a polynomial.
		\item One can assume that if $(n_1,\dots,n_k)\in E_k\backslash E_{k,p}$, i.e. when there exists $j\in [1,k]$ such that $\#\{i\leq k\ |\ n_i=j\} > p$, then $P_{(n_1,\dots,n_k)}=0$.
		\item Else if $(n_1,\dots,n_k)\in E_{k,p}$, one can write 
		$$P_{(n_1,\dots,n_k)} = \sum_{M\in F} c_M M,$$
		where $F$ is a collection of monomials and for any $M\in F$, $c_M\in [1,\infty)$. Besides, for any $l_j\in[0,p-\#\{i\leq k\ |\ n_i=j\}]$, $1\leq j\leq k+1$, one can find a monomial $M\in F$ such that $\deg_{X_{\{k,\dots,j\}}}M=l_j$, reciprocally for every monomial $M\in F$,  $\deg_{X_{\{k,\dots,j\}}}M\in [0,p-\#\{i\leq k\ |\ n_i=j\}]$.
	\end{itemize}
	
	Next, let us prove that if it is true for a given $k$ then it is true for $k+1$:
	\begin{itemize}
		\item We begin with \begin{align*}
			& 2^{k+1} \nabla_V^{\{G,\dots,G\},T_{k+1}}\circ\dots\circ\nabla_V^{\{G\},T_1}\left(\sum_{l=0}^{p}X^l\right) \\
			&= \sum_{(n_1,\dots,n_k)\in E_k} e^{-\sum\limits_{1\leq i\leq k} (t_i-t_{n_i-1})/2}\sum_{n_{k+1}=1}^{k+1} e^{t_{n_{k+1}-1}-t_{k+1}} \\
			&\quad\quad\quad\quad\quad\quad\quad\quad\quad\quad\quad\quad \left[\D_{\{k,\dots,n_{k+1}\}} \left(P_{(n_1,\dots,n_k)} \right) \right]\left(X_{\{k+1\}},X_{\{k+1,k\}},\dots,X_{\{k+1,\dots,1\}}\right)\ \D V(X_{\emptyset}) \\
			&= \sum_{(n_1,\dots,n_{k+1})\in E_k} e^{-\sum\limits_{1\leq i\leq k+1} (t_i-t_{n_i-1})/2} \left[\D_{\{k,\dots,n_{k+1}\}} \left(P_{(n_1,\dots,n_k)} \right) \right]\left(X_{\{k+1\}},X_{\{k+1,k\}},\dots,X_{\{k+1,\dots,1\}}\right)\\
			&\quad\quad\quad\quad\quad\quad\quad\quad\quad\quad\quad\quad\quad\quad\quad\quad\quad\quad\quad\quad\quad\quad\quad\quad\quad\quad\quad\quad\quad\quad\quad\quad\quad\quad\quad\quad\quad\quad\quad \D V(X_{\emptyset}),
		\end{align*}
		and then set
		$$ P_{(n_1,\dots,n_{k+1})} = \left[\D_{\{k,\dots,n_{k+1}\}} \left(P_{(n_1,\dots,n_k)} \right) \right]\left(X_{\{k+1\}},X_{\{k+1,k\}},\dots,X_{\{k+1,\dots,1\}}\right)\ \D V(X_{\emptyset}).$$
		
		\item Since by assumption, for every $j$, $\deg_{X_{\{k,\dots,j\}}}P_{(n_1,\dots,n_k)}\in [0,p-\#\{i\leq k\ |\ n_i=j\}]$, then if $(n_1,\dots,n_{k+1})\in E_{k+1}\backslash E_{k+1,p}$, then either $(n_1,\dots,n_k)\in E_k\backslash E_{k,p}$ and $P_{(n_1,\dots,n_k)}=0$, or $\#\{i\leq k\ |\ n_i=n_{k+1}\} = p$ and $\deg_{X_{\{k,\dots,n_{k+1}\}}}P_{(n_1,\dots,n_k)}=0$. In both cases, this implies that $ P_{(n_1,\dots,n_{k+1})}=0$ if $(n_1,\dots,n_{k+1})\in E_{k+1}\backslash E_{k+1,p}$.
		
		\item Else if $(n_1,\dots,n_{k+1})\in E_{k+1,p}$, one can write 
		$$P_{(n_1,\dots,n_k)} = \sum_{M\in F} c_M M,$$
		hence
		$$P_{(n_1,\dots,n_{k+1})} = \sum_{M\in F} c_M \left[\D_{\{k,\dots,n_{k+1}\}}M \right]\left(X_{\{k+1\}},X_{\{k+1,k\}},\dots,X_{\{k+1,\dots,1\}}\right)\ \D V(X_{\emptyset}).$$
		Consequently, one can write
		$$P_{(n_1,\dots,n_{k+1})} = \sum_{M\in G} c_M M,$$
		where $G$ is a collection of monomials and for any $M\in G$, $c_M\in [1,\infty)$. Besides, for any $l_j\in[0,p-\#\{i\leq k\ |\ n_i=j\}]$, $1\leq j\leq k+1$, one can find a monomial $M\in F$ such that $\deg_{X_{\{k,\dots,j\}}}M=l_j$. Consequently $\left[\D_{\{k,\dots,n_{k+1}\}}M \right]\left(X_{\{k+1\}},X_{\{k+1,k\}},\dots,X_{\{k+1,\dots,1\}}\right)\ \D V(X_{\emptyset})$ yields a monomial such that $\deg_{X_{\{k+1,\dots,n_{k+1}\}}}M=l_{n_{k+1}}-1$, $\deg_{X_{\emptyset}}M$ can be any integer in $[0,p-\#\{i\leq k+1\ |\ n_i=k+1\}]$ since $\D V= \alpha_p X + \sum_{i=0}^{p}(i+1)X^i$, and $\deg_{X_{\{k+1,\dots,j\}}}M=l_j$ else.
		
		Reciprocally, by following the same reasoning, for every monomial $M\in G$, for any $j\leq k+2$, $\deg_{X_{\{k+1,\dots,j\}}}M\in [0,p-\#\{i\leq k\ |\ n_i=j\}]$.
	\end{itemize}
	
	Moverover, thanks to the Schwinger-Dyson equations, see Proposition~\ref{3SDE}, one has that the trace of a monomial in free semi-circular variables is a non-negative integer. Consequently, if $(n_1,\dots,n_k)\in E_{k,p}$, let $M$ be a monomial such that for all $j\leq k+1$, $\deg_{X_{\{k,\dots,n_{k+1}\}}}P_{(n_1,\dots,n_k)}=0$, i.e $M=1$, then
	$$ \tau\left(P_{(n_1,\dots,n_k)}(x^{T_k})\right) \geq c_M \geq 1.$$
	Thus one has that
	\begin{align*}
		&\frac{1}{2^k}\int_{t_k\geq \dots\geq t_1\geq 0}  \tau\left(\nabla_V^{\{G,\dots,G\},T_k}\circ\dots\circ\nabla_V^{\{G\},T_1}\left(\sum_{l=0}^{p}X^l\right)(x^{T_k}) \right) dt_1\dots dt_k \\
		&\geq \frac{1}{2^k} \int_{t_k\geq \dots\geq t_1\geq 0}  \sum_{(n_1,\dots,n_k)\in E_{k,p}} e^{-\sum\limits_{1\leq i\leq k} (t_i-t_{n_i-1})/2}\ dt_1\dots dt_k \\
		&= \int_{t_k\geq \dots\geq t_1\geq 0}  \sum_{(n_1,\dots,n_k)\in E_{k,p}} e^{-\sum\limits_{1\leq i\leq k} t_i-t_{n_i-1}}\ dt_1\dots dt_k \\
		&= I_{k,p}.
	\end{align*}
	Consequently thanks to \eqref{kjdnvflskmvls}, one has that for $\lambda$ small enough,
	$$ \sum_{k\geq 0} I_{k,p} \lambda^k <\infty.$$
	Hence there exists a constant $K_p$ such that $I_{k,p} \leq (K_p)^k$.	
\end{proof}

As a corollary of Lemmas \ref{odvfkmssm} and \ref{lkmdsocms} we immediately get that 

\begin{lemma}
	\label{sdlmcldsfmvs}
	There exists a constant $C$ such that with
	$$ E_{k} = \Big\{ (n_1,\dots,n_k)\in\N^k\ \Big|\ \forall i, 1\leq n_i \leq i\Big\}, $$
	one has,
	\begin{equation*}
		I_k=\int_{t_n\geq \dots \geq t_1\geq 0} \sum_{(n_1,\dots,n_k)\in E_{k}} e^{-\sum\limits_{1\leq i\leq k} t_i-t_{n_i-1}}\ dt_1\dots dt_k \leq  C^k.
	\end{equation*}
\end{lemma}

\subsection{Proof of Theorem \ref{maintherorem} and \ref{maintherorem2}}

Note that until this section we did not need Assumption \ref{kdmcslcs0}. However we need it in the following in order to control the error term.

\begin{proof}[Proof of Theorem \ref{maintherorem}]
	
	By applying Theorem \ref{mainlemma} repeatedly one gets that for any $K>0$,
	\begin{align}
		\label{mswkmcowmc}
		&\frac{\E\left[\ts_N\left(P(X^N)\right) e^{-\lambda N\tr_N\left(V(X^N)\right)}\right]}{\E\left[ e^{-\lambda N\tr_N\left(V(X^N)\right)}\right]} \\
		&= \sum_{0\leq l \leq n } \frac{1}{N^{2l}} \sum_{0\leq k_0,\dots,k_l\leq K} (-\lambda)^{k_0+\dots+k_l} \int_{A_{k_0,\dots,k_l}} \tau\Big((\nabla_V)^{k_l}\circ L\circ (\nabla_V)^{k_{l-1}}\circ\dots \nonumber \\
		&\quad\quad\quad\quad\quad\quad\quad\quad\quad\quad\quad\quad\quad\quad\quad\quad\quad\quad\quad\quad \dots \circ L\circ (\nabla_V)^{k_0}(P)\left(x^{T_{2l+k_l+\dots+k_0}}\right)\Big) dt_1\dots dt_{2l+k_0+\dots+k_l} \nonumber \\
		&\quad + \frac{1}{N^{2(n+1)}} \sum_{0\leq k_0,\dots,k_{n}\leq K} (-\lambda)^{k_0+\dots+k_{n}} \int_{A_{k_0,\dots,k_{n},0}} \E\Big[\tau_N\Big(L\circ (\nabla_V)^{k_n}\circ L\circ (\nabla_V)^{k_{n-1}}\circ\dots \nonumber \\
		&\quad\quad\quad\quad\quad\quad\quad\quad \dots \circ L\circ (\nabla_V)^{k_0}(P)\left(X^{N,T_{2(n+1)+k_{n}+\dots+k_0}}\right)\Big)  e^{-\lambda N\tr_N\left(V(X^N)\right)}\Big] dt_1\dots dt_{2(n+1)+k_0+\dots+k_n} \nonumber \\
		&\quad\quad\quad\quad\quad\quad\quad\quad\quad\quad\quad\quad\quad\quad\quad\quad\quad\quad\quad\quad\quad\quad\quad\quad\quad\quad\quad\quad\quad\quad\quad\quad\quad \times\E\left[ e^{-\lambda N\tr_N\left(V(X^N)\right)}\right]^{-1}, \nonumber
	\end{align}

	\noindent where $A_{k_0,\dots,k_{l}}$ is as in Equation \eqref{cmskcokmcvps}, and $A_{k_0,\dots,k_{n},0}=A_{k_0,\dots,k_{n+1}}$ with $k_{n+1}=0$. Then let us set a few definitions. First and foremost, given $M\in\A^H_{d}$ a monomial, recall that we denote by $\deg M$ the length of $M$ as a word in the variables $X_{i,I}$, see Definition \ref{degree2}. We also define $\deg^{h}M$ as the number of occurrences of variables  belonging to the family $(X_{i,I})_{i\in [1,d],I\in J_H^h}$. Then for $Q\in\A^H_{d}$, we can write 
	$$ Q = \sum_{1\leq i\leq \nb(Q)} c_i M_i $$
	where $c_i\in\C$ and $M\in\A^H_{d}$ are monomials (not necessarily distinct). We also define $C_{\mathrm{max}}(Q) = \max \{1, \sup_i |c_i|\}$ and
	$$ D_N = 2+\max \left\{\norm{X_i^N}\right\}_{1\leq i\leq d},$$
	then we get that 
	\begin{equation}
		\label{3majorgross}
		\norm{Q(X^{N,T_H})} \leq \nb(Q) \times  C_{\mathrm{max}}(Q) \times D_N^{\deg(Q)} .
	\end{equation}
	In other words, Equation \eqref{3majorgross} gives us an upper bound of the norm of a polynomial evaluated in $X^{N,T_H}$ which can be written as a linear combination of at most $\nb(Q)$ monomials of degree smaller or equal to $\deg(Q)$ and coefficients bounded by $C_{\mathrm{max}}(Q)$. Besides, one can always write,
	\begin{align*}
		\nabla_V^{H,T_{n+1}}(Q) = \sum_{0\leq h\leq n}\ e^{(\widetilde{t}_h-t_{n+1})/2} Q_h,
	\end{align*}
	where $Q_h$ is such that
	\begin{itemize}
		\item $\deg(Q_h) \leq \deg(Q) + \deg(V) -2$,
		\item $\deg^{0}(Q_h) = \deg(V)-1$,
		\item for any $\widetilde{h}\in [1,n+1]$, $\deg^{\widetilde{h}}(Q_h) \leq \deg^{\widetilde{h}-1}(Q)$,
		\item $\nb(Q_h) \leq \nb(Q) \times \deg^h(Q) \times \nb(V) \times \deg(V)$,
		\item $C_{\mathrm{max}}(Q_h) \leq \frac{1}{2}C_{\mathrm{max}}(Q) C_{\mathrm{max}}(V)$.
	\end{itemize}
	Similarly, one has that
	\begin{align*}
		&L^{H,T_{n+2}}(Q) = \sum_{1\leq s\leq n+1} \1_{[\widetilde{t}_{s-1},\widetilde{t}_s]}(t_{n+2})\ \sum_{\substack{0\leq h,k \leq n \\ 0\leq x,y\leq s-1}} e^{(\widetilde{t}_h+\widetilde{t}_k+\widetilde{t}_y+\widetilde{t}_x)/2-t_{n+1}-t_{n+2}}\ Q_{s,h,k,x,y},
	\end{align*}
	where $Q_{s,h,k,x,y}$ is such that,
	\begin{itemize}
		\item $\deg(Q_{s,h,k,x,y}) \leq \deg(Q) -4$,
		\item $\deg^{0}(Q_{s,h,k,x,y}) = \deg^{1}(Q_{s,h,k,x,y}) = 0$,
		\item for any $\widetilde{h}\in [2,n+2]$, $\deg^{\widetilde{h}}(Q_{s,h,k,x,y}) \leq \deg^{\widetilde{h}-2}(Q)$,
		\item $\nb(Q_{s,h,k,x,y}) \leq \nb(Q) \times \deg^h(Q)\deg^k(Q)\deg^x(Q)\deg^y(Q) $,
		\item $C_{\mathrm{max}}(Q_{s,h,k,x,y}) \leq \frac{1}{2}C_{\mathrm{max}}(Q)$.
	\end{itemize}
	Thus by combining those results with Equation \eqref{3majorgross}, if we set 
	\begin{itemize}
		\item $K_j=\sum_{z=0}^{j}k_z$,
		\item for $i\in [K_{j-1}+1,K_j]$, $r_i=t_{2j+i}$,
		\item $F_n = \cup_{r=0}^{n-1}\{K_r+2r+1,K_r+2r+2\}$,
	\end{itemize}
	 then we get by induction that
	\begin{align}
		\label{sdckmsldc}
		&\norm{L\circ(\nabla_V)^{k_n}\circ L\circ (\nabla_V)^{k_{n-1}}\circ  \dots \circ L\circ (\nabla_V)^{k_0}(P)\left(X^{N,T_{2n+K_{n}}}\right)} \\
		&\leq \prod_{0\leq j\leq n}\Bigg(\sum_{ \substack{l_i\in [1,i],\\ K_{j-1}+1\leq i\leq K_j}} e^{-\sum\limits_{K_{j-1}+1\leq i\leq K_j} (r_i-r_{l_i-1})/2} \nonumber\\
		&\quad\quad\quad\quad\quad\times \sum_{1\leq s\leq K_j+2j+1} \1_{[\widetilde{t}_{s-1},\widetilde{t}_s]}\left(t_{K_j+2j+2}\right) \sum_{\substack{0\leq h,k \leq K_j+2j,\ 0\leq x,y\leq s-1 \\ h,k,x,y\notin F_n}} e^{({t}_h+{t}_k+{t}_y+{t}_x)/2-t_{K_j+2j+1}-t_{K_j+2j+2}} \Bigg) \nonumber\\
		&\quad\times \left(\frac{1}{2}C_{\mathrm{max}}(V)\nb(V)\deg(V)^2\right)^{k_0+\dots+k_n} \deg(V)^{4(n+1)} \deg(P)! \nb(P) C_{\mathrm{max}}(P) D_N^{\deg P + (k_0+\dots+k_n)(\deg V -1)}\nonumber
	\end{align}
	Note that every time we apply the operator $\nabla_V$ we gain a factor $\frac{C_{\mathrm{max}}(V)}{2}$ since $C_{\mathrm{max}}(Q_h)$ is upper bounded by $ \frac{1}{2}C_{\mathrm{max}}(Q) C_{\mathrm{max}}(V)$. Similarly $\nb(Q_h)$ is upper bounded by $\nb(Q) \times \deg^h(Q) \times \nb(V) \times \deg(V)$, and by induction $\deg^h(Q)$ is upper bounded by the maximum of $\deg(V)-1$ and $\deg P$. However one can only differentiate $P$ that many times before its degree is smaller than the one of $V$. Taking all of those remarks into consideration, after repeatedly applying the operator $\nabla_V$ we get the term $\left(\frac{1}{2}C_{\mathrm{max}}(V)\nb(V)\deg(V)^2\right)^{k_0+\dots+k_n} \deg(P)!$. We also apply the operator $L$ a total of $n+1$ times, and given that $C_{\mathrm{max}}(Q_{s,h,k,x,y})$ is upper bounded by $\frac{1}{2}C_{\mathrm{max}}(Q)$ whereas $\nb(Q_{s,hx,k,x,y})$ is upper bounded by $\nb(Q) \times \deg^h(Q)\deg^k(Q)\deg^x(Q)\deg^y(Q) $ which makes us gain a factor $\deg(V)^4$, we get the factor $\deg(V)^{4(n+1)}$.

	Next, for a given $j\in [0,n-1]$, with $t_1,\dots,t_{K_j+2j}\in A_{k_0,\dots k_j}$,
	\begin{align}
		\label{kemvdclewf}
		&\int_{t_{K_j+2j}}^{\infty}\int_0^{t_{K_j+2j+1}} \sum_{1\leq s\leq K_j+2j+1} \1_{[\widetilde{t}_{s-1},\widetilde{t}_s]}\left(t_{K_j+2j+2}\right) \\
		&\quad\quad\quad\quad\quad\quad\quad\quad\quad\quad\quad \sum_{\substack{0\leq h,k \leq K_j+2j,\ 0\leq x,y\leq s-1 \nonumber \\ h,k,x,y\notin F_n}} e^{({t}_h+{t}_k+{t}_y+{t}_x)/2-t_{K_j+2j+1}-t_{K_j+2j+2}}\ dt_{K_j+2j+2} dt_{K_j+2j+1} \nonumber \\
		&\int_{r_{K_j}}^{\infty}\int_0^{t_{K_j+2j+1}} \sum_{1\leq s\leq K_j+1} \1_{[r_{s-1},r_s]}\left(t_{K_j+2j+2}\right) \nonumber \\
		&\quad\quad\quad\quad\quad\quad\quad\quad\quad\quad\quad \sum_{0\leq h,k \leq K_j,\ 0\leq x,y\leq s-1 } e^{(r_h+r_k+r_y+r_x)/2-t_{K_j+2j+1}-t_{K_j+2j+2}}\ dt_{K_j+2j+2} dt_{K_j+2j+1} \nonumber \\
		&= \sum_{0\leq h,k \leq K_j,\ 0\leq x,y< s\leq K_j+1 } \int_{r_{K_j}}^{\infty}\int_{r_{s-1}}^{r_s} e^{(r_h+r_k+r_y+r_x)/2-t_{K_j+2j+1}-t_{K_j+2j+2}}\ dt_{K_j+2j+2} dt_{K_j+2j+1} \nonumber \\
		&= \sum_{0\leq h,k,x,y \leq K_j} \int_{r_{K_j}}^{\infty} \sum_{x\vee y <s\leq  K_j+1} \int_{r_{s-1}}^{r_s} e^{({r}_h+{r}_k+{r}_y+{r}_x)/2-t_{K_j+2j+1}-t_{K_j+2j+2}}\ dt_{K_j+2j+2} dt_{K_j+2j+1} \nonumber \\	
		&\leq \sum_{0\leq h,k,x,y \leq K_j} \int_{r_{K_j}}^{\infty} \int_{r_{x\vee y}}^{\infty} e^{({r}_h+{r}_k+{r}_y+{r}_x)/2-t_{K_j+2j+1}-t_{K_j+2j+2}}\ dt_{K_j+2j+2} dt_{K_j+2j+1} \nonumber \\
		&= \sum_{0\leq h,k,x,y \leq K_j} e^{({r}_h+{r}_k+{r}_y+{r}_x)/2-r_{K_j}-r_{x\vee y}} \nonumber\\
		&\leq (K_j+1)^4. \nonumber	
	\end{align}
	Note that in the last line, we used that thanks to the definition of $A_{k_0,\dots,k_{l}}$ in Equation \eqref{cmskcokmcvps}, one can assume that the sequence $(r_1,\dots,r_{K_j})$ is non-decreasing, and hence $e^{({r}_h+{r}_k+{r}_y+{r}_x)/2-r_{K_j}-r_{x\vee y}}\leq 1$.
	
	Thus by plugging this result back in Equation \eqref{sdckmsldc}, one gets that
	\begin{align*}
		&\norm{(\nabla_V)^{k_n}\circ L\circ (\nabla_V)^{k_{n-1}}\circ  \dots \circ L\circ (\nabla_V)^{k_0}(P)\left(X^{N,T_{2(n+1)+K_{n}}}\right)} \\
		&\leq (K_0+1)^4\times (K_1+1)^4\times\dots\times (K_n+1)^4 \sum_{l_i\in [1,i], i\in [1,K_n]} e^{-\sum\limits_{i\in [1,K_n]} (r_i-r_{l_i-1})/2} \\
		&\quad\times \left(\frac{1}{2}C_{\mathrm{max}}(V)\nb(V)\deg(V)^2\right)^{K_n} \left(\deg(V)\right)^{4(n+1)} \deg(P)!\ \nb(P) C_{\mathrm{max}}(P)\times  D_N^{\deg P + K_n(\deg V -1)}.
	\end{align*}
	Hence one has that,
	\begin{align*}
		&\left|\int_{A_{k_0,\dots,k_n,0}} \tau_N\Big((\nabla_V)^{k_n}\circ L\circ (\nabla_V)^{k_{n-1}}\circ \dots \circ L\circ (\nabla_V)^{k_0}(P)\left(X^{N,T_{2(n+1)+k_{n}+\dots+k_0}}\right)\Big) dt_1\dots dt_{2(n+1)+k_0+\dots+k_n}\right| \\
		&\leq (K_n+1)^{4(n+1)} I_{K_n}\left(C_{\mathrm{max}}(V)\nb(V)\deg(V)^2\right)^{K_n} \left(\deg(V)\right)^{4(n+1)} \deg(P)!\ \nb(P) C_{\mathrm{max}}(P) \times D_N^{\deg P + K_n(\deg V -1)}.
	\end{align*}
	Thus thanks to Lemma \ref{sdlmcldsfmvs}, there exists a universal constant $C_1$ such that
	\begin{align*}
		&\left|\int_{A_{k_0,\dots,k_n,0}} \tau_N\Big((\nabla_V)^{k_n}\circ L\circ (\nabla_V)^{k_{n-1}}\circ \dots \circ L\circ (\nabla_V)^{k_0}(P)\left(X^{N,T_{2(n+1)+k_{n}+\dots+k_0}}\right)\Big) dt_1\dots dt_{2(n+1)+k_0+\dots+k_n}\right| \\
		&\leq (K_n+1)^{4(n+1)} \left(C_1 \times C_{\mathrm{max}}(V)\nb(V)\deg(V)^2\right)^{K_n} \left(\deg(V)\right)^{4(n+1)} \deg(P)!\ \nb(P) C_{\mathrm{max}}(P) \times D_N^{\deg P + K_n(\deg V -1)}.
	\end{align*}
	Besides, thanks Assumption \ref{kdmcslcs0}, there exists a constant $C$ and a sequence $u_N$ such that for any $\lambda\in [0,1]$, for any $k\leq u_N$,
	\begin{align*}
		&\frac{\E\left[D_N^k\ e^{-\lambda N\tr_N\left(V(X^N)\right)}\right]}{\E\left[ e^{-\lambda N\tr_N\left(V(X^N)\right)}\right]} \\
		&\leq \left(\frac{\E\left[D_N^{2u_N} e^{-\lambda N\tr_N\left(V(X^N)\right)}\right]}{\E\left[ e^{-\lambda N\tr_N\left(V(X^N)\right)}\right]}\right)^{\frac{k}{2u_N}} \\
		&\leq 2^k \left(2^{2u_N}+\frac{\E\left[ \left(\tr_N((X_1^N)^{2u_N})+\dots+\tr_N((X_d^N)^{2u_N})\right) e^{-\lambda N\tr_N\left(V(X^N)\right)}\right]}{\E\left[ e^{-\lambda N\tr_N\left(V(X^N)\right)}\right]}\right)^{\frac{k}{2u_N}} \\
		&\leq \left(4+2(dN)^{\frac{1}{2u_N}}C\right)^k.
	\end{align*}
	Thus we want to fix $K$ in \eqref{mswkmcowmc} such that 
	$$\deg P + K_n (\deg V -1) \leq u_N.$$
	Thus we fix $K = \beta u_N$ where $\beta$ is a constant such that $ n\beta(\deg V -1) <1$ and for $N$ large enough the equation above is satisfied. Consequently, there exists a constant $C_2$ which only depends on $V$ such that
	\begin{align*}
		&\Bigg| \int_{A_{k_0,\dots,k_n,0}} \E\Big[\tau_N\Big((\nabla_V)^{k_n}\circ L\circ (\nabla_V)^{k_{n-1}}\circ\dots L\circ (\nabla_V)^{k_0}(P)\left(X^{N,T_{2(n+1)+k_{n}+\dots+k_0}}\right)\Big)  e^{-\lambda N\tr_N\left(V(X^N)\right)}\Big] \\
		&\quad\quad\quad\quad\quad\quad\quad\quad\quad\quad\quad\quad\quad\quad\quad\quad\quad\quad\quad\quad\quad\quad dt_1\dots dt_{2n+k_0+\dots+k_n}\quad \times\E\left[ e^{-\lambda N\tr_N\left(V(X^N)\right)}\right]^{-1} \Bigg| \\
		&\leq (K_n+1)^{4(n+1)} \left(C_2 \times C_{\mathrm{max}}(V)\nb(V)\deg(V)^2\right)^{K_n} \left(\deg(V)\right)^{4(n+1)} \deg(P)!\ \nb(P) C_{\mathrm{max}}(P) C^{\deg P}.
	\end{align*}
	Besides, for $a<1$,
	\begin{align}
		\label{ksncksnck}
		\sum_{0\leq k_0,\dots,k_{n}\leq K} (K_n+1)^{4(n+1)} a^{K_n} &\leq \sum_{k\geq 0} \sum_{k_0+\dots+k_{n}= k} (k+1)^{4(n+1)} a^{k} \\
		&\leq \frac{d^{4(n+1)}}{(da)^{4(n+1)}}\left(\sum_{k\geq 0} \sum_{k_0+\dots+k_{n}= k} a^{k+4(n+1)}\right) \nonumber\\
		&= \frac{d^{4(n+1)}}{(da)^{4(n+1)}}\left(a^{4(n+1)} \sum_{k_0,\dots,k_{n}\geq 0} a^{k_0+\dots+k_n}\right) \nonumber\\
		&= \frac{d^{4(n+1)}}{(da)^{4(n+1)}}\left(\frac{a^4}{1-a}\right)^{n+1} \nonumber\\
		&\leq  \frac{(5(n+1))^{4(n+1)}}{(1-a)^{5(n+1)}}. \nonumber
	\end{align}
	
	\noindent Thus for $\lambda<\left(C_2 \times C_{\mathrm{max}}(V)\nb(V)\deg(V)\right)^{-1}$, one has 
	\begin{align*}
		&\frac{\E\left[\ts_N\left(P(X^N)\right) e^{-\lambda N\tr_N\left(V(X^N)\right)}\right]}{\E\left[ e^{-\lambda N\tr_N\left(V(X^N)\right)}\right]} \\
		&= \sum_{0\leq l \leq n } \frac{1}{N^{2l}} \sum_{0\leq k_0,\dots,k_l\leq K} (-\lambda)^{k_0+\dots+k_l} \int_{A_{k_0,\dots,k_l}} \tau\Big((\nabla_V)^{k_l}\circ L\circ (\nabla_V)^{k_{l-1}}\circ\dots \nonumber \\
		&\quad\quad\quad\quad\quad\quad\quad\quad\quad\quad\quad\quad\quad\quad\quad\quad\quad\quad\quad\quad \dots \circ L\circ (\nabla_V)^{k_0}(P)\left(x^{T_{2l+k_l+\dots+k_0}}\right)\Big) dt_1\dots dt_{2l+k_0+\dots+k_l} \nonumber \\
		&\quad + \mathcal O\left(\frac{1}{N^{2(n+1)}}\right).
	\end{align*}
	Besides, with the very same kind of computations, with $\alpha_l^V(\lambda,P)$ defined as in \eqref{smcskmdcskmcslk},
	\begin{align*}
		&\Bigg|\alpha_l^V(\lambda,P) - \sum_{0\leq k_0,\dots,k_l\leq K} (-\lambda)^{k_0+\dots+k_l} \int_{A_{k_0,\dots,k_l}} \tau\Big((\nabla_V)^{k_l}\circ L\circ (\nabla_V)^{k_{l-1}}\circ\dots \\
		&\quad\quad\quad\quad\quad\quad\quad\quad\quad\quad\quad\quad\quad\quad\quad\quad\quad\quad\quad\quad \dots \circ L\circ (\nabla_V)^{k_0}(P)\left(x^{T_{2l+k_l+\dots+k_0}}\right)\Big) dt_1\dots dt_{2l+k_0+\dots+k_l} \Bigg|\\
		&= \Bigg|\sum_{\exists i\in [0,l], k_i>K} (-\lambda)^{k_0+\dots+k_l} \int_{A_{k_0,\dots,k_l}} \tau\Big((\nabla_V)^{k_l}\circ L\circ (\nabla_V)^{k_{l-1}}\circ\dots \\
		&\quad\quad\quad\quad\quad\quad\quad\quad\quad\quad\quad\quad\quad\quad\quad\quad\quad\quad\quad\quad \dots \circ L\circ (\nabla_V)^{k_0}(P)\left(x^{T_{2l+k_l+\dots+k_0}}\right)\Big) dt_1\dots dt_{2l+k_0+\dots+k_l} \Bigg|\\
		&\leq \sum_{k_0+\dots +k_l>K} \lambda^{k_0+\dots+k_l} \Bigg| \int_{A_{k_0,\dots,k_l}} \tau\Big((\nabla_V)^{k_l}\circ L\circ (\nabla_V)^{k_{l-1}}\circ\dots \\
		&\quad\quad\quad\quad\quad\quad\quad\quad\quad\quad\quad\quad\quad\quad\quad\quad\quad\quad\quad\quad \dots \circ L\circ (\nabla_V)^{k_0}(P)\left(x^{T_{2l+k_l+\dots+k_0}}\right)\Big) dt_1\dots dt_{2l+k_0+\dots+k_l} \Bigg| \\
		&\leq \sum_{k_0+\dots +k_l>K} \lambda^{K_l} (K_l+1)^{4l} I_{K_l}\left(2^{\deg V -1} C_{\mathrm{max}}(V)\nb(V)\deg(V)^2\right)^{K_l} \left(\deg(V)\right)^{4l} \deg(P)!\ \nb(P) C_{\mathrm{max}}(P) 2^{\deg P} \\
		&\leq C_{P,l,V} \sum_{k_0+\dots +k_l>K} (K_l+1)^{4l} \left(\lambda C \times 2^{\deg V -1} C_{\mathrm{max}}(V)\nb(V)\deg(V)^2\right)^{K_l} \\
		&\leq C_{P,l,V} \sum_{k> K} \sum_{k_0+\dots +k_l = k} (k+1)^{4l} \left(\lambda C \times 2^{\deg V -1} C_{\mathrm{max}}(V)\nb(V)\deg(V)^2\right)^{k}. \\
	\end{align*}
	Next, similarly to Equation \eqref{ksncksnck}, we have that
	\begin{align*}
		\sum_{k_0+\dots +k_l>K} (K_l+1)^{4l} a^{K_l} &\leq \sum_{k_0>K/l} \sum_{k_1,\dots,k_l\geq 0} (K_l+1)^{4l} a^{K_l} \\
		&\leq \frac{d^{4l}}{(da)^{4l}} \left(\sum_{k_0>K/l} \sum_{k_1,\dots,k_l\geq 0} a^{K_l+4l}\right) \\
		&\leq \frac{d^{4l}}{(da)^{4l}} \left(\frac{a^{K/l+4l}}{(1-a)^{l+1}}\right) \\
	\end{align*}
	Thus for $N$ large enough and $\lambda < \left(C \times 2^{\deg V -1} C_{\mathrm{max}}(V)\nb(V)\deg(V)^2\right)^{-1}$,
	\begin{align*}
		&\Bigg|\alpha_l^V(\lambda,P) - \sum_{0\leq k_0,\dots,k_l\leq K} (-\lambda)^{k_0+\dots+k_l} \int_{A_{k_0,\dots,k_l}} \tau\Big((\nabla_V)^{k_l}\circ L\circ (\nabla_V)^{k_{l-1}}\circ\dots \\
		&\quad\quad\quad\quad\quad\quad\quad\quad\quad\quad\quad\quad\quad\quad\quad\quad\quad\quad\quad\quad \dots \circ L\circ (\nabla_V)^{k_0}(P)\left(x^{T_{2l+k_l+\dots+k_0}}\right)\Big) dt_1\dots dt_{2l+k_0+\dots+k_l} \Bigg|\\
		&\leq C_{P,l,V,\lambda} \left(\lambda C \times 2^{\deg V -1} C_{\mathrm{max}}(V)\nb(V)\deg(V)^2\right)^{K/l}.
	\end{align*}
	In particular since we picked $K= \beta u_N\gg \log(N)$, the quantity above is of order $\mathcal{O}\left(N^{-2(n+1)}\right)$ for any $n$. Hence the conclusion.
\end{proof}

\begin{proof}[Proof of Theorem \ref{maintherorem2}]
	The proof is very similar to the previous one. However since we have a cut-off, we have to use Theorem \ref{mainlemma2} instead of Theorem \ref{mainlemma}, which yields the following equation.
	\begin{align}
		\label{mswkmcowmc2}
		&\frac{\E\left[\ts_N\left(P(X^N)\right) e^{-\lambda N\tr_N\left(V(X^N)\right)} \1_{\forall i, ||X_i^N||\leq K}\right]}{\E\left[ e^{-\lambda N\tr_N\left(V(X^N)\right)} \1_{\forall i, ||X_i^N||\leq K}\right]} \\
		&= \sum_{0\leq l \leq n } \frac{1}{N^{2l}}\ \alpha_l^V(\lambda,P) \nonumber \\
		&\quad + \frac{1}{N^{2(n+1)}} \sum_{k_0,\dots,k_n\geq 0} (-\lambda)^{k_0+\dots+k_{n}} \int_{A_{k_0,\dots,k_{n},0}} \E\Big[\tau_N\Big(L\circ(\nabla_V)^{k_n}\circ L\circ (\nabla_V)^{k_{n-1}}\circ\cdots \nonumber \\
		&\quad\quad\quad \dots \circ L\circ (\nabla_V)^{k_0}(P)\left(X^{N,T_{2(n+1)+k_{n}+\dots+k_0}}\right)\Big)  e^{-\lambda N\tr_N\left(V(X^N)\right)} \1_{\forall i, ||X_i^N||\leq K}\Big] dt_1\dots dt_{2(n+1)+k_0+\dots+k_n} \nonumber \\
		&\quad\quad\quad\quad\quad\quad\quad\quad\quad\quad\quad\quad\quad\quad\quad\quad\quad\quad\quad\quad\quad\quad\quad\quad\quad\quad\quad\quad\quad \times\E\left[ e^{-\lambda N\tr_N\left(V(X^N)\right)} \1_{\forall i, ||X_i^N||\leq K} \right]^{-1} \nonumber \\
		&\quad+ \sum_{0\leq l \leq n } \frac{1}{N^{2(l+1)}} \sum_{k_0,\dots,k_l\geq 0} \lambda^{k_0+\dots+k_l} \int_{A_{k_0,\dots,k_l}} \mathcal{E}_{k_0,\dots,k_l}\ dt_1\dots dt_{2l+k_0+\dots+k_l}, \nonumber
	\end{align}
	where with the notations of Theorem \ref{mainlemma2}, for any constant $M<K$, one has that 
	$$ \left|\mathcal{E}_{k_0,\dots,k_l}\right| \leq C_{M,K} \norm{(\nabla_V)^{k_l}\circ L\circ (\nabla_V)^{k_{l-1}}\circ\dots \circ L\circ (\nabla_V)^{k_0}(P)}_{D}\ \mu_{V,K}^N\left(\max_{i\leq i\leq d} \norm{X_i} \geq M\right).$$
	But then, with the very same proof as the one of Equation \eqref{sdckmsldc}, we get that
	\begin{align}
		\label{sdckmsldc2}
		&\norm{(\nabla_V)^{k_l}\circ L\circ (\nabla_V)^{k_{n-1}}\circ  \dots \circ L\circ (\nabla_V)^{k_0}(P)}_{D} \\
		&\leq \prod_{0\leq j\leq l-1}\Bigg(\sum_{ \substack{l_i\in [1,i],\\ K_{j-1}+1\leq i\leq K_j}} e^{-\sum\limits_{K_{j-1}+1\leq i\leq K_j} (r_i-r_{l_i-1})/2} \nonumber\\
		&\quad\quad\quad\quad\quad\times \sum_{1\leq s\leq K_j+2j+1} \1_{[\widetilde{t}_{s-1},\widetilde{t}_s]}\left(t_{K_j+2j+2}\right) \sum_{\substack{0\leq h,k \leq K_j+2j,\ 0\leq x,y\leq s-1 \\ h,k,x,y\notin F_n}} e^{({t}_h+{t}_k+{t}_y+{t}_x)/2-t_{K_j+2j+1}-t_{K_j+2j+2}} \Bigg) \nonumber\\
		&\quad\times\sum_{ \substack{l_i\in [1,i],\\ K_{l-1}+1\leq i\leq K_l}} e^{-\sum\limits_{K_{n-1}+1\leq i\leq K_n} (r_i-r_{l_i-1})/2} \nonumber\\
		&\quad\times \left(\frac{1}{2}C_{\mathrm{max}}(V)\nb(V)\deg(V)^2\right)^{K_l} \left(\deg(V)\right)^{4l} \deg(P)! \nb(P) C_{\mathrm{max}}(P)\times D^{\deg P + K_l(\deg V -1)}\nonumber
	\end{align}

	\noindent Next, thanks to the estimates in \eqref{kemvdclewf},
	\begin{align*}
		&\int_{A_{k_0,\dots,k_l}} \norm{(\nabla_V)^{k_l}\circ L\circ (\nabla_V)^{k_{n-1}}\circ  \dots \circ L\circ (\nabla_V)^{k_0}(P)}_{L}\ dt_1\dots dt_{2l+k_0+\dots+k_n} \\
		&\leq (K_l+1)^{4l} I_{K_l}\left(C_{\mathrm{max}}(V)\nb(V)\deg(V)^2\right)^{K_l} \left(\deg(V)\right)^{4l} \deg(P)!\ \nb(P) C_{\mathrm{max}}(P) \times D^{\deg P + K_l(\deg V -1)}.
	\end{align*}
	Thus after using Lemma \ref{sdlmcldsfmvs}, we get that there exists constants $C_V$ and $C_P$ such that
	\begin{align*}
		&\int_{A_{k_0,\dots,k_l}} \norm{(\nabla_V)^{k_l}\circ L\circ (\nabla_V)^{k_{n-1}}\circ  \dots \circ L\circ (\nabla_V)^{k_0}(P)}_{D}\ dt_1\dots dt_{2l+k_0+\dots+k_n} \\
		&\leq (K_l+1)^{4l} \left(C_V D^{\deg V -1}\right)^{K_l} \left(\deg(V)\right)^{4l} C_P.
	\end{align*}
	Thus for $\lambda< (C_V L^{\deg V -1})^{-1}$, similarly to \eqref{ksncksnck},
	\begin{align*}
		&\sum_{k_0,\dots,k_l\geq 0} \lambda^{k_0+\dots+k_l} \int_{A_{k_0,\dots,k_l}} \mathcal{E}_{k_0,\dots,k_l}\ dt_1\dots dt_{2l+k_0+\dots+k_l} \\
		&\leq C_{M,K} \times \mu_{V,K}^N\left(\max_{i\leq i\leq d} \norm{X_i} \geq M\right) \left(\deg(V)\right)^{4l} C_P \sum_{k_0,\dots,k_l\geq 0} (K_l+1)^{4l} \left(\lambda C_V D^{\deg V -1}\right)^{K_l} \\
		&\leq C_{M,K} \times \mu_{V,K}^N\left(\max_{i\leq i\leq d} \norm{X_i} \geq M\right) \left(\deg(V)\right)^{4l} C_P \frac{(5l)^{4l}}{(1-\lambda C_V D^{\deg V -1})^{5l}}.
	\end{align*}
	Thus the last line of Equation \eqref{mswkmcowmc2} is of order $\mathcal{O}(e^{-cN})$ for some constant $c$ as long as $K$ is sufficiently large thanks to Lemma \ref{vkmslmlskmvlsmv}. Besides, we also have that
	\begin{align*}
		&\Bigg|\sum_{k_0,\dots,k_n\geq 0} (-\lambda)^{k_0+\dots+k_{n}} \int_{A_{k_0,\dots,k_{n},0}} \E\Big[\tau_N\Big(L\circ(\nabla_V)^{k_n}\circ L\circ (\nabla_V)^{k_{n-1}}\circ\cdots \nonumber \\
		&\quad\quad\quad \dots \circ L\circ (\nabla_V)^{k_0}(P)\left(X^{N,T_{2(n+1)+k_{n}+\dots+k_0}}\right)\Big)  e^{-\lambda N\tr_N\left(V(X^N)\right)} \1_{\forall i, ||X_i^N||\leq K}\Big] dt_1\dots dt_{2(n+1)+k_0+\dots+k_n} \nonumber \\
		&\quad\quad\quad\quad\quad\quad\quad\quad\quad\quad\quad\quad\quad\quad\quad\quad\quad\quad\quad\quad\quad\quad\quad\quad\quad\quad\quad\quad \times\E\left[ e^{-\lambda N\tr_N\left(V(X^N)\right)} \1_{\forall i, ||X_i^N||\leq K} \right]^{-1} \Bigg| \nonumber \\
		&\leq \sum_{k_0,\dots,k_n\geq 0} \lambda^{k_0+\dots+k_{n}} \int_{A_{k_0,\dots,k_{n},0}} \norm{(L\circ(\nabla_V)^{k_n}\circ\cdots \circ L\circ (\nabla_V)^{k_0}(P)}_{K+2} dt_1\dots dt_{2(n+1)+k_0+\dots+k_n} \\
		&\leq \left(\deg(V)\right)^{4(n+1)} C_P \frac{(5(n+1))^{4(n+1)}}{(1-\lambda C_V (K+2)^{\deg V -1})^{5(n+1)}}.
	\end{align*}
	Hence the conclusion.	
\end{proof}

\section{Proof of Corollaries \ref{sdkjnvkdm} and \ref{doifvpelmv}}

Before getting to the actual proof we will need the following proposition.

\begin{prop}
	\label{osmocsmcdiwnjd}
	Given $P,Q,V\in\A_d$, one can view those polynomials as elements of $\A_{d+1}$, then as long as $\lambda$ is small enough for the power series 
	$$\alpha_0^V(\lambda,\cdot) = \sum_{k\geq 0} (-\lambda)^{k} \int_{A_{k}} \tau\left(\nabla_V^{k}(\cdot)\left(x^{T_{k}}\right)\right) dt_1\dots dt_{k},
	$$ 
	to converge, we have that 
	\begin{equation}
		\label{soicdmsodcskmcd}
		\alpha_0^V(\lambda,PX_{d+1}QX_{d+1}) = \alpha_0^V(\lambda,P)\alpha_0^V(\lambda,Q).
	\end{equation}
	In particular, if $K,V$ and $\lambda$ satisfy the hypotheses of Theorem \ref{maintherorem2}, then
	\begin{equation*}
		\frac{\E\left[\ts_N\left(P_1(X^N)\right)\cdots \ts_N\left(P_l(X^N)\right) e^{-\lambda N\tr_N\left(V(X^N)\right)}  \1_{\forall i, ||X_i^N||\leq K}\right]}{\E\left[ e^{-\lambda N\tr_N\left(V(X^N)\right)}  \1_{\forall i, ||X_i^N||\leq K}\right]} = \alpha_0^V(\lambda,P_1)\cdots \alpha_0^V(\lambda,P_l) + \mathcal{O}\left(\frac{1}{N^2}\right).
	\end{equation*}
\end{prop}

\begin{proof}
	First let us define 
	$$
	\begin{array}{ccccc}
		\widetilde{\nabla}_V & : & \oplus_{H} \A_{d}^H & \to & \oplus_{H} \A_{d}^{\{H,G\}} \\
		& & \oplus_H P_H & \mapsto & \oplus_H \widetilde{\nabla}_V^{H,T_H}\left(P_H\right) \\
	\end{array},
	$$
	with 
	\begin{align*}
		\widetilde{\nabla}_V^{H,T_{n+1}}(Q) \deq \frac{1}{2}\sum_{1\leq i\leq d}\ \sum_{0\leq h\leq n}\ e^{(\widetilde{t}_h-t_{n+1})/2}\ \partial_{i,h}Q(G^+(X))\# \D_iV(X^{\emptyset}),
	\end{align*}
	where the notation $\#$ is as in Definition \ref{sknclks}. Then by induction we claim that for any $P\in\A_d$, one can write $\nabla_V^{k}(P)$ and $\widetilde{\nabla}_V^{k}(P)$ as 
	\begin{equation}
		\label{ermcwmcw}
		\nabla_V^{k}(P) = \sum_l A_lB_l,\quad \widetilde{\nabla}_V^{k}(P) = \sum_l B_lA_l,
	\end{equation} 
	with $A_l,B_l\in \A_d^{\{G,\dots,G\}}$, and in particular
	\begin{align*}
		\tau\left(\nabla_V^{k}(P)\left(x^{T_{k}}\right)\right) = \tau\left(\widetilde{\nabla}_V^{k}(P)\left(x^{T_k}\right)\right).
	\end{align*}
	Indeed, if Equation \eqref{ermcwmcw} is satisfied for a given $k$, then with the notation $A\otimes B \widetilde{\#} C =BCA$,
	\begin{align*}
		\nabla_V^{k+1}(P) &= \sum_l \nabla_V(A_lB_l) \\
		&= \sum_l \frac{1}{2}\sum_{1\leq i\leq d}\ \sum_{0\leq h\leq n}\ e^{(\widetilde{t}_h-t_{n+1})/2}\ \D_{i,h}(A_lB_l)(G^+(X)) \D_iV(X^{\emptyset}) \\
		&= \sum_l \frac{1}{2}\sum_{1\leq i\leq d}\ \sum_{0\leq h\leq n}\ e^{(\widetilde{t}_h-t_{n+1})/2}\ \Big(\partial_{i,h}A_l(G^+(X))\widetilde{\#} B_l(G^+(X)) \\
		&\quad\quad\quad\quad\quad\quad\quad\quad\quad\quad\quad\quad\quad\quad\quad\quad\quad\quad\quad+ \partial_{i,h}B_l(G^+(X))\widetilde{\#} A_l(G^+(X))\Big)\D_iV(X^{\emptyset}). 
	\end{align*}
	On the other hand, 
	\begin{align*}
		\widetilde{\nabla}_V^{k+1}(P) &= \sum_l \widetilde{\nabla}_V(B_lA_l) \\
		&= \sum_l \frac{1}{2}\sum_{1\leq i\leq d}\ \sum_{0\leq h\leq n}\ e^{(\widetilde{t}_h-t_{n+1})/2}\ \partial_{i,h}(B_lA_l)(G^+(X))\# \D_iV(X^{\emptyset}) \\
		&= \sum_l \frac{1}{2}\sum_{1\leq i\leq d}\ \sum_{0\leq h\leq n}\ e^{(\widetilde{t}_h-t_{n+1})/2}\ \Big(B_l(G^+(X)) \left(\partial_{i,h}A_l(G^+(X))\# \D_iV(X^{\emptyset})\right)  \\
		&\quad\quad\quad\quad\quad\quad\quad\quad\quad\quad\quad\quad\quad\quad\quad+ \left(\partial_{i,h}B_l(G^+(X))\#\D_iV(X^{\emptyset}) \right)A_l(G^+(X))\Big). 
	\end{align*}
	Hence Equation \eqref{ermcwmcw} is satisfied for $k+1$. Consequently, we get that
	\begin{align*}
		&\tau\left(\nabla_V^{k}(PX_{d+1}QX_{d+1})\left(x^{T_k}\right)\right) \\
		&= \tau\left(\widetilde{\nabla}_V^{k}(PX_{d+1}QX_{d+1})\left(x^{T_k}\right)\right) \\
		&= \sum_{i_1,\dots,i_k\in\{0,1\}} \tau\left(\widetilde{\nabla}_V^{i_k} \cdots \widetilde{\nabla}_V^{i_1}(P)\left(x^{T_k}\right) x^{T_k}_{d+1,\{1,\dots,k\}} \widetilde{\nabla}_V^{1-i_k} \cdots \widetilde{\nabla}_V^{1-i_1}(Q)\left(x^{T_k}\right) x^{T_k}_{d+1,\{1,\dots,k\}} \right) \\
		&= \sum_{i_1,\dots,i_k\in\{0,1\}} \tau\left(\widetilde{\nabla}_V^{i_k} \cdots \widetilde{\nabla}_V^{i_1}(P)\left(x^{T_k}\right)\right) \tau\left(\widetilde{\nabla}_V^{1-i_k} \cdots \widetilde{\nabla}_V^{1-i_1}(Q)\left(x^{T_k}\right) \right) \\
		&= \sum_{i_1,\dots,i_k\in\{0,1\}} \tau\left({\nabla}_V^{i_k} \cdots {\nabla}_V^{i_1}(P)\left(x^{T_k}\right)\right) \tau\left({\nabla}_V^{1-i_k} \cdots {\nabla}_V^{1-i_1}(Q)\left(x^{T_k}\right) \right),
	\end{align*}
	where we used Proposition \ref{3SDE} in the last line, and 
	$${\nabla}_V^{1} = {\nabla}_V, \quad {\nabla}_V^{0}: P\in\A_d^H\mapsto P(G^+(X))\in \A_d^{\{H,G\}}.$$
	Besides, if $j_1,\dots,j_p$ is the list of all indices sorted by increasing order such that $i_{{j_l}}=1$ for every $l$, then with $\widetilde{T}_l^1 = \{ t_{j_1}, \dots, t_{j_l} \}$,
	\begin{align*}
		&\tau\left({\nabla}_V^{i_k} \cdots {\nabla}_V^{i_1}(P)\left(x^{T_k}\right)\right) = \tau\left({\nabla}_V^{\{G,\dots,G\},\widetilde{T}_p^1} \circ \cdots\circ {\nabla}_V^{{G},\widetilde{T}_2^1}\circ {\nabla}_V^{{G},\widetilde{T}_1^1}(P)\left(x^{T_k}\right)\right).
	\end{align*}
	Consequently, with $m_1,\dots,m_{k-p}$ the list of all indices sorted by increasing order such that $i_{{m_l}}=0$ for every $l$, and with $\widetilde{T}_l^2 = \{ t_{m_1},\dots, t_{m_l} \}$, one has that
	\begin{align*}
		&\int_{A_k} \tau\left({\nabla}_V^{i_k} \cdots {\nabla}_V^{i_1}(P)\left(x^{T_k}\right)\right) \tau\left({\nabla}_V^{1-i_k} \cdots {\nabla}_V^{1-i_1}(Q)\left(x^{T_k}\right) \right) dt_1\dots dt_{k}\\
		&= \int_{A_k} \tau\left({\nabla}_V^{\{G,\dots,G\},\widetilde{T}_p^1} \circ \cdots\circ {\nabla}_V^{{G},\widetilde{T}_2^1}\circ {\nabla}_V^{{G},\widetilde{T}_1^1}(P)\left(x^{\widetilde{T}_p^1}\right)\right) \\
		&\quad\quad\quad\quad\quad\quad\quad\quad\quad\quad\quad\quad\quad\quad\quad\quad\quad \tau\left({\nabla}_V^{\{G,\dots,G\},\widetilde{T}_{k-p}^2} \circ \cdots\circ {\nabla}_V^{{G},\widetilde{T}_2^2}\circ {\nabla}_V^{{G},\widetilde{T}_1^2}(P)\left(x^{\widetilde{T}_{k-p}^2}\right)\right) dt_1\dots dt_{k}.
	\end{align*}
	Then one makes the change of variables $\sigma_{j_1,\dots,j_p}$ defined by $(\sigma_{j_1,\dots,j_p}(t))_{j_l}=t_l$ and $(\sigma_{j_1,\dots,j_p}(t))_{m_l}=t_{p+l}$. Thus, with the convention $j_0=0$, $j_{p+1}=k+1$, $t_0=0$ and $t_{k+1}=\infty$, one has
	\begin{align*}
		\sigma_{j_1,\dots,j_p}^{-1}\left(A_k\right) = &\{t_1,\dots,t_k\in\R^+\ |\ t_p\geq\dots\geq t_1\geq 0 \text{ and } t_{k}\geq\dots\geq t_{p+1}\geq 0  \} \\
		&\cap \{t_1,\dots,t_k\in\R^+\ |\ \forall l\in [1, k-p],\ t_{l+p}\in (t_a,t_{a+1}) \text{ where } j_a<m_l<j_{a+1} \}.
	\end{align*}
	Thus $\{t_1,\dots,t_k\in\R^+\ |\ t_p\geq\dots\geq t_1\geq 0 \text{ and } t_{k}\geq\dots\geq t_{p+1}\geq 0  \}$ is the disjoint union of the sets $\sigma_{j_1,\dots,j_p}^{-1}\left(A_k\right)$ for $i_1,\dots,i_k\in\{0,1\}$ such that $\sum_l i_l = p$. Consequently, with $T_{l}\setminus T_p = \{t_{p+1},\dots, t_l\}$,
	\begin{align*}
		&\int_{A_k} \tau\left(\nabla_V^{k}(PX_{d+1}QX_{d+1})\left(x^{T_k}\right)\right) \\ 
		&=\sum_{p=0}^k\sum_{\substack{i_1,\dots,i_k\in\{0,1\},\ \sum_l i_l = p}} \int_{A_{k}} \tau\left({\nabla}_V^{i_k} \cdots {\nabla}_V^{i_1}(P)\left(x^{T_k}\right)\right) \tau\left({\nabla}_V^{1-i_k} \cdots {\nabla}_V^{1-i_1}(Q)\left(x^{T_k}\right) \right) dt_1\dots dt_{k} \\
		&=\sum_{p=0}^k\sum_{\substack{i_1,\dots,i_k\in\{0,1\},\ \sum_l i_l = p}} \int_{\sigma_{j_1,\dots,j_p}^{-1}\left(A_k\right)} \tau\left({\nabla}_V^{\{G,\dots,G\},T_p} \circ \cdots\circ {\nabla}_V^{{G},T_2}\circ {\nabla}_V^{{G},T_1}(P)\left(x^{T_p}\right)\right) \\
		&\quad\quad\quad\quad\quad\quad\quad\quad\quad\quad\quad\quad\quad\quad\quad\quad\quad \tau\left({\nabla}_V^{\{G,\dots,G\},T_k\setminus T_p} \circ \cdots\circ {\nabla}_V^{{G},T_{p+1}\setminus T_p}(P)\left(x^{T_k\setminus T_p}\right)\right) dt_1\dots dt_{k} \\
		&= \sum_{p=0}^k \int_{A_p} \tau\left(\nabla_V^{p}(P)\left(x^{T_{p}}\right)\right) dt_1\dots dt_{p} \int_{A_{k-p}}\tau\left(\nabla_V^{k-p}(P)\left(x^{T_{k-p}}\right)\right) dt_1\dots dt_{k-p}.
	\end{align*}
	Hence we get that
	\begin{align*}
		&\alpha_0^V(\lambda,PX_{d+1}QX_{d+1}) \\
		&= \sum_{k\geq 0} (-\lambda)^k \int_{A_k} \tau\left(\nabla_V^{k}(PX_{d+1}QX_{d+1})\left(x^{T_k}\right)\right) dt_1\dots dt_{k} \\
		&= \sum_{k\geq 0} (-\lambda)^k \sum_{p=0}^k \int_{A_p} \tau\left(\nabla_V^{p}(P)\left(x^{T_{p}}\right)\right) dt_1\dots dt_{p} \int_{A_{k-p}}\tau\left(\nabla_V^{k-p}(P)\left(x^{T_{k-p}}\right)\right) dt_1\dots dt_{k-p} \\
		&= \alpha_0^V(\lambda,P)\alpha_0^V(\lambda,Q).
	\end{align*}
	This proves Equation \eqref{soicdmsodcskmcd}. Finally, thanks to Gaussian integration by parts, see Equation \eqref{3IPPG2}, one has that for $K$ sufficiently large,
	\begin{align*}
		&\E\left[\ts_N\left(P(X^N)\right) \ts_N\left(Q(X^N)\right) e^{-\lambda N\tr_N\left(V(X^N)\right)}  \1_{\forall i, ||X_i^N||\leq K}\right] \\
		&= \E\left[\ts_N\left(P(X^N)X_{d+1}^NQ(X^N)X_{d+1}^N\right) e^{-\lambda N\tr_N\left(V(X^N)\right)}  \1_{\forall i, ||X_i^N||\leq K}\right].
	\end{align*}
	Then one can use Theorem \ref{maintherorem2} to conclude.
\end{proof}

\subsection{Application to free entropy}

\begin{proof}[Proof of Corollary \ref{doifvpelmv}]
	For $n$ larger than the degree of $V$, thanks to Equation \eqref{soicdjsomcs}, if $X^N$ is an element of $\Gamma_R(\alpha_0^V(\lambda,\cdot),n,N,\varepsilon)$ then $\ts_N(V(X^N)) =\alpha_0^V(\lambda,V) + \mathcal{O}(\varepsilon)$, thus
	\begin{align*}
		\P\left( X^N\in \Gamma_R(\alpha_0^V(\lambda,\cdot),n,N,\varepsilon) \right) &= \int_{\Gamma_R(\alpha_0^V(\lambda,\cdot),n,N,\varepsilon)} d\P(X^N) \\
		&= \int_{\Gamma_R(\alpha_0^V(\lambda,\cdot),n,N,\varepsilon)} e^{\lambda N\tr_N(V(X^N)) - \lambda N\tr_N(V(X^N))}d\P(X^N) \\
		&= e^{N^2(\lambda\ \alpha_0^V(\lambda,V) + \mathcal{O}(\varepsilon))} \int_{\Gamma_R(\alpha_0^V(\lambda,\cdot),n,N,\varepsilon)} e^{- \lambda N\tr_N(V(X^N))}d\P(X^N) \\
		&= e^{N^2(\lambda\ \alpha_0^V(\lambda,V) - \int_0^{\lambda} \alpha_0^V(\nu,V)\ d\nu + \mathcal{O}(\varepsilon) + \mathcal{O}(N^{-2}))}  \mu^N_{\lambda V}\left(X^N\in \Gamma_R(\alpha_0^V(\lambda,\cdot),n,N,\varepsilon)\right).
	\end{align*}
	where we used Theorem \ref{maintherorem} in the last line. Besides, thanks to Assumption \ref{kdmcslcs0}, for $R>C$,
	\begin{align*}
		&\mu^N_{\lambda V}\left(X^N\notin \Gamma_R(\alpha_0^V(\lambda,\cdot),n,N,\varepsilon)\right) \\
		&\leq \mu_{\lambda V}^N\left(\max_{i\leq i\leq d} \norm{X_i} > R\right) + \sum_{M\in\A_d \text{monomial},\ \deg M \leq n} \mu_{\lambda V}^N\left( \left| \alpha_0^V(\lambda,M) - \ts_N\left(M(X^N)\right) \right| \geq \varepsilon\right) \\
		&\leq \frac{\mu_{\lambda V}^N\left(\norm{X_1}^{2u_N}\right) +\dots+ \mu_{\lambda V}^N\left(\norm{X_d}^{2u_N}\right)}{R^{2u_N}} \\
		&\quad\quad\quad\quad\quad\quad\quad\quad\quad\quad\quad\quad\quad\quad\quad\quad\quad\quad + \sum_{M\in\A_d \text{monomial},\ \deg M \leq n} \mu_{\lambda V}^N\left( \left| \alpha_0^V(\lambda,M) - \ts_N\left(M(X^N)\right) \right|^2 \right) \varepsilon^{-2} \\
		&\leq \frac{\mu_{\lambda V}^N\left(\ts_N\left({X_1}^{2u_N}\right)\right) +\dots+ \mu_{\lambda V}^N\left(\ts_N\left({X_d}^{2u_N}\right)\right)}{R^{2u_N}}N^{\frac{1}{2u_N}} \\
		&\quad\quad\quad\quad\quad\quad\quad\quad\quad\quad\quad\quad\quad\quad\quad\quad\quad\quad + \sum_{M\in\A_d \text{monomial},\ \deg M \leq n} \mu_{\lambda V}^N\left( \left| \alpha_0^V(\lambda,M) - \ts_N\left(M(X^N)\right) \right|^2 \right) \varepsilon^{-2} \\
		&\leq d\left(\frac{C}{R}\right)^{2u_N}N^{\frac{1}{2u_N}} + \varepsilon^{-2} \sum_{M\in\A_d \text{monomial},\ \deg M \leq n} \mu_{\lambda V}^N\left( \ts_N\left(M(X^N)\right)^2 \right) - \alpha_0^V(\lambda,M)^2 + \mathcal{O}\left(N^{-2}\right)  \\
		&\leq \mathcal{O}\left(N^{-2}\right),
	\end{align*}
	where we used Theorem \ref{maintherorem2} and Proposition \ref{osmocsmcdiwnjd} in the last two lines. Hence
	\begin{align*}
		\P\left( X^N\in \Gamma_R(\alpha_0^V(\lambda,\cdot),n,N,\varepsilon) \right) &= e^{N^2(\alpha_0^V(\lambda,V) - \int_0^{\lambda} \alpha_0^V(\nu,V)\ d\nu + \mathcal{O}(\varepsilon) + \mathcal{O}(N^{-2}))} (1-\mathcal{O}(N^{-2})),
	\end{align*}
	and consequently for $R$ and $n$ large enough,
	\begin{align}
		\label{sodmclsdmk}
		\limsup_{N\to\infty} \frac{1}{N^2} \log\P\left( X^N\in \Gamma_R(\alpha_0^V(\lambda,\cdot),n,N,\varepsilon) \right) &= \lambda\ \alpha_0^V(\lambda,V) - \int_0^{\lambda} \alpha_0^V(\nu,V)\ d\nu + \mathcal{O}(\varepsilon).
	\end{align}
	Since the quantity 
	$$\limsup_{N\to\infty} \frac{1}{N^2} \log\P\left( X^N\in \Gamma_R(\alpha_0^V(\lambda,\cdot),n,N,\varepsilon) \right)$$
	is decreasing with respect to $n$, we find for $R$ large enough that
	\begin{align}
		\label{wocnsoc}
		&\inf_{n\in\N} \inf_{\varepsilon>0} \limsup_{N\to\infty} \frac{1}{N^2} \log\P\left( X^N\in \Gamma_R(\alpha_0^V(\lambda,\cdot),n,N,\varepsilon) \right) \\ \nonumber
		&= \lim_{n\to\infty} \inf_{\varepsilon>0} \limsup_{N\to\infty} \frac{1}{N^2} \log\P\left( X^N\in \Gamma_R(\alpha_0^V(\lambda,\cdot),n,N,\varepsilon) \right) \\ \nonumber
		&= \lambda\ \alpha_0^V(\lambda,V) - \int_0^{\lambda} \alpha_0^V(\nu,V)\ d\nu \\
		&= \int_0^{\lambda} \nu\ \frac{d}{d\nu}\alpha_0^V(\nu,V)\ d\nu. \nonumber
	\end{align}
	Hence the conclusion. Finally note that Equation \eqref{sodmclsdmk} still stands when one replaces $\limsup$ by $\liminf$. Besides the quantity
	$$\liminf_{N\to\infty} \frac{1}{N^2} \log\P\left( X^N\in \Gamma_R(\alpha_0^V(\lambda,\cdot),n,N,\varepsilon) \right)$$
	is also decreasing with respect to $n$, hence Equation \eqref{wocnsoc} also stands when one replaces $\limsup$ by $\liminf$.
\end{proof}

\subsection{Application to map enumeration}

\label{ksjdcdnv}

Let us first define the objects that we shall enumerate. First we give the general definition of a map, then we introduce a coloring system.
\begin{defi}
	\label{socdmsomc}
	An embedded graph of genus $g$ is a connected graph $\Gamma$ embedded into a connected orientable surface $S$ of genus g such that $S\setminus\Gamma$ is a disjoint union of connected components, called faces, each of them homeomorphic to an open disk. In this paper we assume that each vertex of $\Gamma$ is endowed with a distinguished edge. Then, two embedded graphs $\Gamma_i \subset S_i$, $i = 1, 2$, are isomorphic if there exists an orientation preserving homeomorphism $\phi : S_1 \to S_2$ such that:
	\begin{itemize}
		\item $\phi_{|\Gamma_1}$ is a graph isomorphism between $\Gamma_1$ and $\Gamma_2$,
		\item we can enumerate the vertices $v_1^i,\dots,v_n^i$ of $\Gamma_i$ such that $\phi(v_j^1) = v_j^2$,
		\item $\phi$ maps the distinguished edge of every vertex $v$ of $\Gamma_1$ to the distinguished edge of $\phi(v)$.
	\end{itemize}  
	Finally a map is an isomorphism class of embedded graphs.
\end{defi}

\begin{defi}
	\label{socdmsomc2}
	Consider $d$ colors $\{1,\dots,d\}$, one for each variable $X_i$. Let $\Gamma$ be a colored graph, i.e. an embedded graph such that every edge is associated with a unique color. We say that one of its vertices is of type $q$, for a monomial $q = X_{i_1}\dots X_{i_p}$, if starting from its distinguished edge and enumerating the edges linked to this vertex in clockwise order with respect to the orientation of the surface, the $k$-th edge is of color $i_k$. A colored map is then an isomorphism class of colored embedded graphs.
\end{defi}

The proof of Corollary \ref{sdkjnvkdm} relies mostly on the fact that there already exists a Taylor expansion similar to the one computed in Theorem \ref{maintherorem}. However in order to avoid any sort of convexity assumption on the potential $V$, we want this expansion to hold in the case where we have a cut-off, i.e. as in Theorem \ref{maintherorem2}. It turns out that adapting the proof of \cite{segala} to this case is rather straightforward. We only need the following proposition and the rest of the proof will follow.

\begin{prop}
	\label{dkcsoiknds}
	Given polynomials $P_1,\cdots, P_l\in\A_d$ we define
	$$ \nu^N(P_1\otimes\cdots\otimes P_l) =  \frac{N^{\widehat{l}}\ \E\left[ \left(\mu_N-\alpha_0^V(\lambda,\cdot)\right)^{\otimes l}(P_1\otimes\dots\otimes P_l) e^{-\lambda N\tr_N\left(V(X^N)\right)} \1_{\forall i, ||X_i^N||\leq K} \right]}{\E\left[ e^{-\lambda N\tr_N\left(V(X^N)\right)} \1_{\forall i, ||X_i^N||\leq K}\right]},$$
	where $\mu_N(P) := \ts_N(P(X^N))$ and $\widehat{l}= l$ if $l$ is even and $l+1$ otherwise. If $Q$ is a monomial, we set $\overline{Q} = \frac{Q}{\deg(Q)}$, and by linearity we extend this definition to any polynomial. Then we set
	$$\Xi(P) \deq P + \lambda\sum_{i=1}^d \D_i\overline{P}\D_i V - \sum_{i=1}^d \left(\id\otimes\alpha_0^V(\lambda,\cdot) + \alpha_0^V(\lambda,\cdot)\otimes\id\right)(\partial_i\D_i \overline{P}),$$
	and $\nu^N$ satisfies the following equations, if $l$ is even,
	\begin{align}
		\nu^N(\Xi(P_1)\otimes P_2\otimes\cdots\otimes P_l) =&\ \sum_{1\leq i\leq d,2\leq r\leq l}  \alpha_0^V(\lambda,\D_i\overline{P_1} \D_iP_r)\ \nu^N(P_2\otimes\cdots\check{P_r}\cdots\otimes P_l) \\
		&+ \frac{1}{N^2}  \sum_{1\leq i\leq d,2\leq r\leq l} \nu^N(\D_i\overline{P_1} \D_iP_r\otimes P_2\otimes\cdots\check{P_r}\cdots\otimes P_l)  \nonumber\\
		&+ \frac{1}{N^2}  \sum_{1\leq i\leq d} \nu^N(\partial_i\D_i\overline{P_1} \otimes P_2\otimes\cdots\otimes P_l) \nonumber \\
		&+ N^{l-2}e^{-\alpha N} \mathcal{E}(P_1,\dots,P_l), \nonumber
	\end{align}
	where $\check{P_r}$ means that $P_r$ is omitted, $\alpha>0$ is a universal constant, and there exists a constant $C_l$ such that $$\left| \mathcal{E}(P_1,\dots,P_l) \right| \leq C_l \norm{P_1}_K \dots \norm{P_l}_K,$$ 
	where $\norm{.}_L$ is as in Equation \eqref{sodcmosmcowkmcd}. Moreover if $l$ is odd
	\begin{align}
		\label{dsoicksm}
		\nu^N(\Xi(P_1)\otimes P_2\otimes\cdots\otimes P_l) =&\ \sum_{1\leq i\leq d,2\leq r\leq l}  \alpha_0^V(\lambda,\D_i\overline{P_1} \D_iP_r)\ \nu^N(P_2\otimes\cdots\check{P_r}\cdots\otimes P_l) \\
		&+ \sum_{1\leq i\leq d,2\leq r\leq l} \nu^N(\D_i\overline{P_1} \D_iP_r\otimes P_2\otimes\cdots\check{P_r}\cdots\otimes P_l)  \nonumber\\
		&+ \sum_{1\leq i\leq d} \nu^N(\partial_i\D_i\overline{P_1} \otimes P_2\otimes\cdots\otimes P_l) \nonumber \\
		& + N^{l-2}e^{-\alpha N} \mathcal{E}(P_1,\dots,P_l). \nonumber
	\end{align}
\end{prop}

\begin{proof}
	First let us note that if $Q$ is a monomial,
	$$\deg(Q) \left(\mu_N-\alpha_0^V(\lambda,\cdot)\right)(Q) = \sum_{i=1}^d \left(\mu_N-\alpha_0^V(\lambda,\cdot)\right)\Big(X_i\D_iQ\Big).$$
	Thus thanks to Lemma \ref{ippcut}, if $l$ is even we have that
	\begin{align*}
		&\nu^N(P_1\otimes P_2\otimes \cdots\otimes P_l) \\
		&= \sum_{i=1}^d \nu^N\left(X_i\D_i\overline{P_1}\otimes P_2\otimes\cdots\otimes P_l\right) \\
		&= \sum_{i=1}^d \left(\alpha_0^V(\lambda,\cdot)\otimes\alpha_0^V(\lambda,\cdot)\left(\partial_i\D_i\overline{P_1}\right) - \alpha_0^V\left(\lambda,(X_i+\lambda\D_i V)\D_i\overline{P_1}\right)\right) \nu^N\left(P_2\otimes\cdots\otimes P_l\right) \\
		&\quad - \sum_{i=1}^d \nu^N\left( \lambda \D_i\overline{P_1}\D_iV \otimes P_2\otimes\cdots\otimes P_l\right)  \\
		&\quad + \frac{1}{N^2} \sum_{i=1}^d \nu^N\left( \partial_i\D_i\overline{P_1}\otimes P_2\otimes\cdots\otimes P_l\right)  \\
		&\quad + \sum_{i=1}^d \nu^N\left( \left((\id\otimes\alpha_0^V(\lambda,\cdot) +\alpha_0^V(\lambda,\cdot)\otimes\id)(\partial_i\D_i \overline{P_1})\right)\otimes P_2\otimes\cdots\otimes P_l\right)  \\
		&\quad + N^l\sum_{i=1}^d \E\Bigg[ \bigg(\ts_N\left(\left(\lambda\D_i V(X^N)+X_i^N\right)\D_i\overline{P_1}(X^N)\right) - \ts_N\otimes\ts_N(\partial_i \overline{P_1}(X^N))\bigg) \\
		&\quad\quad\quad\quad\quad\quad\times \left(\mu_N-\alpha_0^V(\lambda,\cdot)\right)^{\otimes l-1}(P_2\otimes\cdots\otimes P_l)\ e^{-\lambda N\tr_N\left(V(X^N)\right)} \1_{\forall i, ||X_i^N||\leq K} \Bigg] \\
		&\quad\quad\quad\quad \times \E\left[ e^{-\lambda N\tr_N\left(V(X^N)\right)} \1_{\forall i, ||X_i^N||\leq K}\right]^{-1}.
	\end{align*}
	But thanks to Lemma \ref{ippcut} combined with Proposition \ref{osmocsmcdiwnjd}, we get that for any $P$,
	$$ \alpha_0^V(\lambda,\cdot)\otimes\alpha_0^V(\lambda,\cdot)\left(\partial_i\D_iP\right) = \alpha_0^V\left(\lambda,(X_i+\lambda\D_i V)\D_iP\right).$$
	Consequently, we have
	\begin{align*}
		&\nu^N((\Xi P_1)\otimes P_2\otimes \cdots\otimes P_l) \\
		&= \frac{1}{N^2} \sum_{i=1}^d \nu^N\left( \partial_i\D_i\overline{P_1}\otimes P_2\otimes\cdots\otimes P_l\right)  \\
		&\quad + N^l\sum_{i=1}^d \E\Bigg[ \left(\ts_N\left(\left(\lambda\D_i V(X^N)+X_i^N\right)\D_i\overline{P_1}(X^N)\right) - \ts_N\otimes\ts_N(\partial_i \overline{P_1}(X^N))\right) \\
		&\quad\quad\quad\quad\quad\quad\times	 \left(\mu_N-\alpha_0^V(\lambda,\cdot)\right)^{\otimes l-1}(P_2\otimes\cdots\otimes P_l) e^{-\lambda N\tr_N\left(V(X^N)\right)} \1_{\forall i, ||X_i^N||\leq K} \Bigg] \\
		&\quad\quad\quad\quad \times \E\left[ e^{-\lambda N\tr_N\left(V(X^N)\right)} \1_{\forall i, ||X_i^N||\leq K}\right]^{-1}.
	\end{align*}
	Thus we want to use Lemma \ref{ippcut} to compute the last term, but we have to handle the extra term $\left(\mu_N-\alpha_0^V(\lambda,\cdot)\right)^{\otimes l-1}(P_2\otimes\cdots\otimes P_l)$. However it is easy to adapt the proof. Indeed, Equation \eqref{lksmcsmssss} becomes
	\begin{align*}
		&\sum_i \E\Bigg[ \left(\partial_i P(X^N)\# H_i - NP(X^N) \tr_N\left((\D_i V(X^N)+X_i^N)H_i\right)\right) \\
		&\quad\quad\quad\times \left(\mu_N-\alpha_0^V(\lambda,\cdot)\right)^{\otimes l-1}(P_2\otimes\cdots\otimes P_l)\ e^{- N\tr_N\left(V(X^N)\right)} \1_{\forall i, ||X_i^N||\leq K} \Bigg] \\
		&=- \sum_i \E\Bigg[ P(X^N) h'(\norm{X_i}_{2p}) \frac{\ts_N(X_i^{2p-1}H_i)}{\ts_N(X_i^{2p})}\norm{X}_{2p} \\
		&\quad\quad\quad\quad\quad\times \left(\mu_N-\alpha_0^V(\lambda,\cdot)\right)^{\otimes l-1}(P_2\otimes\cdots\otimes P_l)\ e^{- N\tr_N\left(V(X^N)\right)} \1_{\forall i, ||X_i^N||\leq K} \Bigg] \\
		&\quad - \sum_{r=2}^l \sum_i \E\Bigg[ P(X^N)\ts_N\left(\D_i P_r(X^N) H_i\right) \left(\mu_N-\alpha_0^V(\lambda,\cdot)\right)^{\otimes l-2}(P_2\otimes\cdots\check{P}_r\cdots\otimes P_l) \\
		&\quad\quad\quad\quad\quad\quad\quad\quad\quad\quad\quad\quad\quad\quad\quad\quad\quad\quad\quad\quad\quad\quad\quad\quad\quad\quad\quad\quad  e^{- N\tr_N\left(V(X^N)\right)} \1_{\forall i, ||X_i^N||\leq K} \Bigg]. \\
	\end{align*}
	Thus if we fix $P=-\D_i \overline{P_1}$, $H=(0,\dots,0,E_{r,s},0,\dots,0)$ with $E_{r,s}$ in the $i$-th position, then after multiplying by $e_r^*$ on the left, $e_s$ on the right, summing over $r,s$, dividing by $N^{2}$ and finally summing over $i$, we get that
	\begin{align*}
		&\sum_i \E\Bigg[ \bigg( \ts_N\left((\D_i V(X^N)+X_i^N)\D_i\overline{P_1}(X^N)\right) - \ts_N\otimes\ts_N\left(\partial_i \D_i \overline{P_1}(X^N)\right) \bigg) \\
		&\quad\quad\quad\times \left(\mu_N-\alpha_0^V(\lambda,\cdot)\right)^{\otimes l-1}(P_2\otimes\cdots\otimes P_l)\ e^{- N\tr_N\left(V(X^N)\right)} \1_{\forall i, ||X_i^N||\leq K} \Bigg] \\
		&= \frac{1}{N^2}\sum_i \E\Bigg[ h'(\norm{X_i}_{2p}) \frac{\ts_N(X_i^{2p-1} \D_i\overline{P_1}	(X^N))}{\ts_N(X_i^{2p})}\norm{X}_{2p} \\
		&\quad\quad\quad\quad\quad\times \left(\mu_N-\alpha_0^V(\lambda,\cdot)\right)^{\otimes l-1}(P_2\otimes\cdots\otimes P_l)\ e^{- N\tr_N\left(V(X^N)\right)} \1_{\forall i, ||X_i^N||\leq K} \Bigg] \\
		&\quad +\frac{1}{N^2} \sum_{r=2}^l \sum_i \E\Bigg[ \ts_N\left(\D_i P_r(X^N) \D_i\overline{P_1}(X^N)\right) \left(\mu_N-\alpha_0^V(\lambda,\cdot)\right)^{\otimes l-2}(P_2\otimes\cdots\check{P}_r\cdots\otimes P_l) \\
		&\quad\quad\quad\quad\quad\quad\quad\quad\quad\quad\quad\quad\quad\quad\quad\quad\quad\quad\quad\quad\quad\quad\quad\quad\quad\quad\quad\quad  e^{- N\tr_N\left(V(X^N)\right)} \1_{\forall i, ||X_i^N||\leq K} \Bigg]. \\
	\end{align*}
	Thus, this yields the following equation,
	\begin{align*}
		&\nu^N((\Xi P_1)\otimes P_2\otimes \cdots\otimes P_l) \\
		&= \sum_{1\leq i\leq d,2\leq r\leq l} \alpha_0^V(\lambda,\D_i\overline{P_1} \D_iP_r)\ \nu^N(P_2\otimes\cdots\check{P_r}\cdots\otimes P_l)  \nonumber\\
		&\quad + \frac{1}{N^2}  \sum_{1\leq i\leq d,2\leq r\leq l} \nu^N(\D_i\overline{P_1} \D_iP_r\otimes P_2\otimes\cdots\check{P_r}\cdots\otimes P_l)  \nonumber\\
		&\quad + \frac{1}{N^2} \sum_i \nu^N\left( \partial_i\D_i\overline{P_1}\otimes P_2\otimes\cdots\otimes P_l\right)  \\
		&\quad + N^{l-2}e^{-\alpha N} \mathcal{E}(P_1,\dots,P_l),
	\end{align*}
	where $\alpha>0$ is a universal constant, and thanks to Lemma \ref{vkmslmlskmvlsmv}, there exists a constant $C_l$ such that
	$$\left| \mathcal{E}(P_1,\dots,P_l) \right| \leq C_l \norm{P_1}_K \dots \norm{P_l}_K.$$
	Note that in this last inequality we used the fact that $\norm{\D_i\overline{P_1}}_K \leq \norm{P_1}_K$. Finally the proof of the case where $l$ is odd is nearly identical with the difference that since $l-1$ is now even, $\widehat{l-1}$ is now equal to $l-1$ instead of $l$, and $\widehat{l}$ to $l+1$ instead of $l$, hence why we do not normalize by $N^2$ in the last two lines of Equation \eqref{dsoicksm}.
\end{proof}

Thanks to this proposition coupled with Lemma 6.2 of \cite{segala}, we can now prove Corollary \ref{sdkjnvkdm}.

\begin{proof}[Proof of Corollary \ref{sdkjnvkdm}]
	First, let us note that thanks to Theorem 3.5 of \cite{guionnet-segala}, one has that
	\begin{equation}
		\label{oidcmosimcscd}
		\alpha_0^V(\lambda,q_0) = \sum_{\mathbf{k}\in \N^m} \frac{(-\lambda \mathbf{t})^{\mathbf{k}}}{\mathbf{k}!} \mathcal{M}_0^{\mathbf{k}}(q_0),
	\end{equation}
	since both of those quantities are the limit when $N$ goes to infinity of
	\begin{equation*}
		\frac{\E\left[\ts_N\left(q_0(X^N)\right) e^{-\lambda N\tr_N\left(V(X^N)\right)}  \1_{\forall i, ||X_i^N||\leq K}\right]}{\E\left[ e^{-\lambda N\tr_N\left(V(X^N)\right)}  \1_{\forall i, ||X_i^N||\leq K}\right]}.
	\end{equation*}
	Then we define by induction, $\nu_1^N=\nu^N$ and for $h\geq 1$,
	$$\nu_{h+1}^N(P_1\otimes\cdots\otimes P_l) = N^2\left(\nu_h^N - \mathcal{I}_{\frac{\widehat{l}}{2}+h-1}(P_1\otimes\cdots\otimes P_l)\right),$$
	where we defined 
	$$ \mathcal{I}_{g}(P_1\otimes\cdots\otimes P_l) = \sum_{\mathbf{k}\in \N^m} \frac{(-\mathbf{t})^{\mathbf{k}}}{\mathbf{k}!} \mathcal{M}^{\mathbf{k}}_g(P_1\otimes\cdots\otimes P_l), $$
	with $\mathcal{M}^{\mathbf{k}}_g(P_1\otimes\cdots\otimes P_l)$ defined as in Definition 4.2 of \cite{segala}. From there on, we can simply follow the proof of Lemma 6.2 of \cite{segala} by using Proposition \ref{dkcsoiknds} and Equation \eqref{oidcmosimcscd} above instead of Property 3.2 of \cite{segala}. This yields that there exist constants $C,\eta,M$ such that if $\max_i |t_i| <\eta$, then for all $N$,
	$$\mathcal{M}^{\mathbf{k}}_g(P_1\otimes\cdots\otimes P_l) \leq C \norm{P_1}_M\dots \norm{P_l}_M.$$
	In particular, if $l=1$, then 
	\begin{equation*}
		\frac{\E\left[\ts_N\left(q_0(X^N)\right) e^{-\lambda N\tr_N\left(V(X^N)\right)}  \1_{\forall i, ||X_i^N||\leq K}\right]}{\E\left[ e^{-\lambda N\tr_N\left(V(X^N)\right)}  \1_{\forall i, ||X_i^N||\leq K}\right]} = \sum_{0\leq g \leq h } \frac{1}{N^{2g}} \sum_{\mathbf{k}\in \N^m} \frac{(-\lambda \mathbf{t})^{\mathbf{k}}}{\mathbf{k}!} \mathcal{M}_g^{\mathbf{k}}(q_0) + \mathcal{O}\left(N^{-2(h+1)}\right).
	\end{equation*}
	Thus thanks to Equation \eqref{kdjncoskd4}, we get that
	$$ \alpha_g(\lambda,q_0) = \sum_{\mathbf{k}\in \N^m} \frac{(-\lambda \mathbf{t})^{\mathbf{k}}}{\mathbf{k}!} \mathcal{M}_g^{\mathbf{k}}(q_0).$$
	Hence the conclusion.
\end{proof}

\section{The case of the GOE}

The aim of this section is to outline how to adapt the results of this paper to the case of GOE random matrices. However we do not aim to provide a complete proof in order to keep the length of the paper reasonable. GOE random matrices are the symmetric equivalent of the GUE random matrices that we have been working with so far. They are defined as follows.

\begin{defi}
	\label{3GOEdef}
	A GOE random matrix $X^N$ of size $N$ is a symmetric matrix whose coefficients are random variables with the following laws:
	\begin{itemize}
		\item For $1\leq i\leq N$, the random variables $\sqrt{N/2} X^N_{i,i}$ are independent centered Gaussian random variables of variance $1$.
		\item For $1\leq i<j\leq N$, the random variables $\sqrt{N}\ X^N_{i,j}$ are independent centered Gaussian random variables of variance $1$, independent of  $\left(X^N_{i,i}\right)_i$.
	\end{itemize}
\end{defi}

Gaussian integration by parts as described in \eqref{3IPPG2} leads to similar equations but with some notable differences which imply that in a theorem equivalent to Theorem \ref{maintherorem}, one would have an expansion in $N^{-1}$ instead of $N^{-2}$. To do so we define the transpose $P^T$ of a polynomial $P\in \A_d$: if $P=X_{i_1}\cdots X_{i_k}$ then we set $P^T=X_{i_k}\cdots X_{i_1}$, and we extend this definition to general polynomials by linearity. We then define $\widetilde{m}: A\otimes B\mapsto A B^T$. One can adapt Equations \eqref{smcosmcds} as such
\begin{align}
	\label{sdicslkmc}
	\E\left[ \ts_N\left(X^N\ Q(X^N)\right) \right] &= \frac{1}{N^{2}} \sum_{r,s} \E\left[ \tr_N\left(E_{r,s}\ \partial Q(X^N) \#( E_{s,r} + E_{r,s})\right) \right] \\
	&= \frac{1}{N^{2}} \sum_{r,s} \E\left[ e_s^* \left(\partial Q(X^N) \# e_s e_r^*\right) e_r \right] + \E\left[ e_s^* \left(\partial Q(X^N) \# e_r e_s^*\right) e_r \right]\nonumber \\
	&= \E\left[ \ts_N^{\otimes 2}\left(\partial Q(X^N)\right) \right] + \frac{1}{N}\E\left[ \ts_N\left(\widetilde{m}\circ\partial Q(X^N)\right) \right], \nonumber
\end{align}
with the notation $\#$ as in Definition \ref{sknclks} and  $(e_u)_{1\leq u\leq  N}$ the canonical basis of $\C^N$. Similarly one would have 
\begin{align*}
	\E\left[ \tr_N\left(X^N\ Q(X^N)\right) \tr_N\left(V(X^N)\right)\right] &= \frac{1}{N}\sum_{r,s} \E\left[ \tr_N\left(E_{r,s}Q(X^N)\right) \tr_{N}\left( \D V(X^N) \#( E_{s,r} + E_{r,s})\right) \right] \nonumber \\
	&= \E\left[ \ts_N\left(Q(X^N) \left(\D V(X^N) + \D V^T(X^N)\right)\right) \right].
\end{align*}
However in this paper we assume that $V$ is trace self-adjoint, i.e. that $X\in\M_N(\C)_{sa}^d\mapsto\tr_N(V(X))$ is real-valued for any $N$, this implies that $\D V(X^N) = \D V^T(X^N)$, hence 
\begin{align*}
	\E\left[ \tr_N\left(X^N\ Q(X^N)\right) \tr_N\left(V(X^N)\right)\right] &= 2 \E\left[ \ts_N\left(Q(X^N) \D V(X^N) \right) \right].
\end{align*}
Consequently, coupled with Equation \eqref{sdicslkmc}, this allows us to adapt Equation \eqref{lkmsdvmdss} to the case of the GOE and it yields a suitable ``master equation'' as in Theorem \ref{mainlemma}. Thus if we define as in the introduction
\begin{equation*}
	d\mu_{\lambda V}^N(X^N) = \frac{1}{Z_{\lambda V}^N} e^{-N\tr_N\left(\lambda V(X^N)+\frac{1}{4}\sum_{i=1}^d(X_i^N)^2\right)} dX_1^N\dots dX_d^N,
\end{equation*}
where the integral $dX_i^N$ is now with respect to the Lebesgue measure on the set of symmetric matrices of size $N$. Then if we set $\Theta : P\in\A_d \mapsto \sum_{i=1}^d \widetilde{m} \circ \partial_i\D_iP\in \A_d$, heuristically one has for any polynomial $P$,
$$ \mu_{\lambda V}^N\left[ \frac{1}{N}\tr_N\left( P(X^N) \right) \right] = \tau\left(\left(\id + 2\lambda \nabla_V - \frac{\Theta}{N} - \frac{L}{N^2} \right)^{-1}(P)(x)\right), $$
where $x$ is a $d$-tuple of free semicircular variables. Once again, in practice, it is unclear whether the operator above is ever invertible when $\lambda\neq 0$. However, using a similar method as in this paper, one can show that the formula above still holds for sufficiently small $\lambda\geq 0$ when you replace the inverse of the operator by a Taylor expansion with respect to $N^{-1}$. This would yield results similar to Theorems~\ref{maintherorem} and \ref{maintherorem2} for GOE random matrices. Note in particular that the first order asymptotic in the case of the GOE is the same as in the case of the GUE where we replaced $\lambda$ by $2\lambda$.

\subsection*{Acknowledgements}

Both of the authors are supported by the Knut and Alice Wallenberg Foundation, as well as by the Swedish Research Council (VR-2021-04703). Besides, we would also like to thank the reviewer for their detailed comments and careful proofreading of this paper.

\bibliographystyle{abbrv}

\end{document}